\documentclass[a4paper,11pt]{article}
\usepackage{anyfontsize}
\usepackage{lmodern}
\usepackage{microtype}
\usepackage{amsfonts,amssymb,amsmath,amsthm}
\usepackage{nicefrac,xfrac}
\usepackage{mathtools} 
\mathtoolsset{showonlyrefs}
\usepackage[colorlinks=true, urlcolor=blue, linkcolor=blue, citecolor=blue]{hyperref}
\usepackage[
  backend=biber,
  style=alphabetic,
  sorting=anyt,
  maxnames=6,
  maxalphanames=6,
  giveninits=true,
  doi=false,
  isbn=false,
  url=false,
]{biblatex}
\DeclareFieldFormat{extraalpha}{#1}
\DeclareLabelalphaTemplate{
  \labelelement{
    \field[final]{shorthand}
    \field{label}
    \field[strwidth=3,strside=left,ifnames=1]{labelname}
    \field[strwidth=1,strside=left]{labelname}
  }
}

\numberwithin{equation}{section}
\newtheorem{theorem}{Theorem}[section]
\newtheorem{proposition}[theorem]{Proposition}
\newtheorem{corollary}[theorem]{Corollary}

\newtheorem{lemma}[theorem]{Lemma}

\theoremstyle{definition}
\newtheorem{remark}[theorem]{Remark}
\newtheorem{definition}[theorem]{Definition}
\newtheorem{question}[theorem]{Question}
\theoremstyle{remark}

\newcommand{\R}{\mathbb{R}}  
\newcommand{\Sph}{\mathbb{S}}  
\DeclareMathAlphabet{\mathbbold}{U}{bbold}{m}{n}

\DeclareMathOperator{\id}{id}
\DeclareMathOperator{\Int}{Int}
\DeclareMathOperator{\dom}{dom}
\DeclareMathOperator{\ind}{ind}
\DeclareMathOperator{\rk}{rk}

\DeclareMathOperator{\Tr}{Tr}

\DeclareMathOperator{\supp}{supp}
\DeclareMathOperator{\Area}{Area}
\DeclareMathOperator{\Vol}{Vol}
\DeclareMathOperator{\BMO}{BMO}
\DeclareMathOperator{\Gr}{Gr}
\DeclareMathOperator{\Fl}{Flag}

\DeclarePairedDelimiterX\set[2]{\{}{\}}{\,#1 \;\delimsize\vert\; #2\,}
\newcommand*{\dv}[1]{dv_{#1}} 
\newcommand*{\SHM}[1]{u_{\odot}^{#1}} 
\newcommand*{\abs}[1]{\left|#1\right|} 
\newcommand*{\norm}[1]{\left\|#1\right\|} 
\newcommand*{\brr}[1]{\left(#1\right)} 
\newcommand*{\brs}[1]{\left[#1\right]} 
\newcommand*{\brc}[1]{\left\{#1\right\}} 
\newcommand*{\brt}[1]{\left\langle #1\right\rangle } 
\mathchardef\mhyphen="2D 
\renewcommand*{\epsilon}{\varepsilon}
\renewcommand*{\phi}{\varphi}
\newcommand*\quot[2]{
  \mathchoice
      {
          \text{\raisebox{.2em}{$#1$}\Big/\raisebox{-.2em}{$#2$}}%
      }
      {
          #1/#2
      }
      {
          #1\,/\,#2
      }
      {
          #1\,/\,#2
      }
}
\def\Xint#1{\mathchoice
{\XXint\displaystyle\textstyle{#1}}%
{\XXint\textstyle\scriptstyle{#1}}%
{\XXint\scriptstyle\scriptscriptstyle{#1}}%
{\XXint\scriptscriptstyle\scriptscriptstyle{#1}}%
\!\int}
\def\XXint#1#2#3{{\setbox0=\hbox{$#1{#2#3}{\int}$ }
\vcenter{\hbox{$#2#3$ }}\kern-.6\wd0}}

\def\dashint{\Xint-}
\makeatletter
\newcommand*{\oset}[3][0.45ex]{%
  \mathrel{\mathop{#3}\limits^{
    \vbox to#1{\kern-2\ex@
    \hbox{$\scriptstyle#2$}\vss}}}}
\makeatother
\newcommand*{\m}[1]{\brs{#1}} 
\newcommand*{\md}[1]{\mathfrak{D}\brs{#1}} 
\newcommand*{\tensor}[1]{\mathcal{N}[#1]} 

\begin{document}
\author{Denis Vinokurov}
\date{}
\title{Maximizing higher eigenvalues in dimensions three and above}
\maketitle
\begin{abstract}
  We study the problem of maximizing the $k$-th eigenvalue functional
  over the class of absolutely continuous measures on
  a closed Riemannian manifold of dimension $m\geq 3$.
  Extending the work of Karpukhin and Stern on the first eigenvalue, we prove that, for every
  $k\geq 1$, the supremum is attained by a measure induced by a harmonic map into a finite-dimensional sphere.
  The map is smooth outside a closed singular set of Hausdorff dimension at most $m-7$, and is therefore smooth when
  $3 \leq m \leq 6$.
  We further prove that this dimension bound is optimal: for every $m \geq 7$
  and every integer $0\leq d \leq m-7$, there exists a maximizing harmonic map on the $m$-dimensional round sphere
  whose singular set has Hausdorff dimension $d$.
\end{abstract}
\tableofcontents

\section{Introduction and main results}

\subsection{Eigenvalue optimization in high dimensions}
Let $(M,g)$ be a closed connected Riemannian manifold with $m = \dim M$, and let
$\Delta_g = \delta_g d \colon C^\infty(M) \to C^\infty(M)$ be the associated
Laplace-Beltrami operator, where $\delta_g$ is the formal adjoint of $d$.
The spectrum of $\Delta_g$ forms the following
nondecreasing sequence:
\begin{equation}
  0 = \lambda_0(g) < \lambda_1(g) \leq \lambda_2(g) \leq
  \cdots\nearrow \infty,
\end{equation}
where the eigenvalues are repeated according to their multiplicities. The following variational characterization
of $\lambda_k(g)$ is well known:
\begin{equation}\label{eq:eigen-var-char-metric}
  \lambda_k(g) = \inf_{F_{k+1}} \sup_{\phi \in F_{k+1}\setminus\brc{0}} \frac{\int \abs{d\phi}^2_g dv_g}{\int \phi^2 dv_g},
\end{equation}
where $F_{k+1} \subset C^\infty(M)$ ranges over $(k+1)$-dimensional subspaces.

Let $\mathcal{M}(M)$ be the space of all nonzero nonnegative Radon measures on $M$, and
let $\mathcal{M}_c(M) \subset \mathcal{M}(M)$ be the subset of continuous (i.e. atomless) Radon measures.
For $\mu \in \mathcal{M}(M)$, one defines $\lambda_k(\mu) \in [0,\infty]$
in parallel with the eigenvalues~\eqref{eq:eigen-var-char-metric} by:
\begin{equation}\label{meas-eigenval}
  \lambda_k(\mu) = \lambda_k(\mu, g) := \inf_{F_{k+1}} \sup_{\phi \in F_{k+1}\setminus\brc{0}} \frac{\int \abs{d\phi}^2_g dv_g}{\int \phi^2 d\mu},
\end{equation}
where $F_{k+1}$ is the same as in~\eqref{eq:eigen-var-char-metric}. Note that $\lambda_k(\mu) < \infty$ as long as $L^2(M,\mu)$ is at least
$(k + 1)$-dimensional. In particular, $\lambda_k(\mu) < \infty$ for each $\mu \in \mathcal{M}_c(M)$.
If $\mu$ is regular enough, e.g., the map $H^1(M) \to L^2(M, \mu) $ is compact,
the variational eigenvalues~\eqref{meas-eigenval} are the true eigenvalues
of the problem
\begin{equation}\label{eq:meas-eigenfunc}
  \Delta_g \phi = \lambda_k \phi \mu.
\end{equation}
Let us consider the following optimization problem:
\begin{equation}\label{measure-opt}
  \mathcal{V}_k(g) = \sup_{\substack{\mu \in C^\infty\\ \mu > 0}} \overline{\lambda}_k(\mu),
  \quad \text{where } \overline{\lambda}_k(\mu) := \lambda_k(\mu)\mu(M).
\end{equation}

The eigenvalues of measures~\eqref{meas-eigenval} were first introduced and
studied in~\cite*{Kokarev:2014:measure-eigenval} in the context of Riemannian surfaces.
Notice that in dimension $m = 2$, the Dirichlet energy density $\abs{d\phi}^2_g dv_g$ is
conformally invariant. Hence, $\mathcal{V}_k(g) = \mathcal{V}_k(\tilde{g})$
for all $\tilde{g} \in [g]$, and
the optimization problem~\eqref{measure-opt} coincides with the eigenvalue
optimization in the conformal class:
\begin{equation}
  \mathcal{V}_k(g) = \Lambda_k([g]) := \sup_{\tilde{g} \in [g]} \lambda_k(\tilde{g}) \Area (\tilde{g}).
\end{equation}
Thus, $\mathcal{V}_k(g)$ can be viewed as another
generalization of the eigenvalue optimization problem on surfaces
as opposed to the optimization within a conformal class in higher dimensions.

For existence and regularity results on maximizers of $\Lambda_k([g])$ on surfaces,
see~\cite{Karpukhin-Nadirashvili-Penskoi-Polterovich:2022:existence,
Petrides:2018:exist-of-max-eigenval-on-surfaces}. See also \cite{Kim:2022:second-sphere-eigenval,
Karpukhin-Metras-Polterovich:2024:dirac-eigenval-opt,
Perez-Ayala:2021:conf-laplace-sire-xu-norm,
Perez-Ayala:2022:extr-metrics-paneitz-operator,
Gursky-Perez-Ayala:2022:2d-eigenval-conform-laplacian}
for some related eigenvalue optimization results.

The motivation for the above optimization problem arises from its
connection to harmonic maps into spheres. Namely, critical measures
$\mu$ of the functional $\overline{\lambda}_k$
are generated by such maps. More precisely, one can construct
a map $u\colon (M, g) \to \Sph^n$
whose coordinate functions are $k$-th eigenfunctions of~\eqref{eq:meas-eigenfunc}. Then it follows that
$\mu = \abs{du}^2_g dv_g$, $\lambda_k(|du|^2_g dv_g) = 1$,
and the map $u$ solves the equation
\begin{equation}\label{extr-measures}
  \Delta_g u = \abs{du}^2_g u
\end{equation}
coordinatewise
(see, e.g., \cite*[Proposition~4.13]{Karpukhin-Stern:2024:harm-map-in-higher-dim}).
Recall that a map $u\colon (M, g) \to (N,h)$ is called
harmonic if it is a critical point of the energy functional
\begin{equation}
  E[u] = \int \abs{du}^2_{g,h} dv_g.
\end{equation}
In the case $(N,h) = \Sph^n$,  the associated Euler--Lagrange equation is
precisely~\eqref{extr-measures}.

Another interesting observation is that
the suprema~\eqref{measure-opt} admit an equivalent formulation
in terms of an optimization problem involving the negative eigenvalues
of Schrödinger operators. Let $V \in C^\infty(M)$ be a positive
function, and consider the Schrödinger operator $\Delta_g - V$.
We define the \emph{(spectral) index} of $V$, denoted
\begin{equation}
  \ind V := \ind (\Delta_g - V),
\end{equation}
as the maximal dimension of subspaces on which $\Delta_g - V$
is negative definite, which is the same as the number of
negative eigenvalues (counted with multiplicities) of the operator.

Using the correspondence $V = \lambda_k(\mu) \mu$, where $\mu \in C^\infty$,
one immediately sees that $\ind V \leq k$, and hence
\begin{equation}
  \mathcal{V}_k(g) = \sup \set*{\int V dv_g}{V \in C^\infty,\ V > 0,\ \ind V \leq k},
\end{equation}
see also~\cite[Section~2]{Karpukhin-Nadirashvili-Penskoi-Polterovich:2022:existence}.
The positivity condition $V > 0$ can be relaxed to $V \geq 0$, since the inequality
$\ind V \leq k$ is preserved under weak$^*$  limits of measures. In fact, it
was shown in~\cite{Grigoryan-Nadirashvili-Sire:2016:eigenval-schrodinger} that one may
dispense with the condition $V \geq 0$ entirely and take the supremum over
all smooth potentials. This observation was later used in~\cite{Karpukhin-Nadirashvili-Penskoi-Polterovich:2022:existence}
to prove the existence
and regularity of maximizers for $\mathcal{V}_k(g)$ in dimension $m=2$.

In~\cite{Grigoryan-Netrusov-Yau:2004:eignval-of-ellipt-oper}
Grigor’yan, Netrusov and Yau established upper bounds for a broad class of
eigenvalue optimization problems on metric measure spaces.
In particular, they proved (see~\cite[Remark~5.10]{Grigoryan-Netrusov-Yau:2004:eignval-of-ellipt-oper})
that
\begin{equation}
  \mathcal{V}_k(g) \leq Ck^\frac{2}{m} \Vol(g)^{1-\frac{2}{m}} < \infty.
\end{equation}

Karpukhin and
Stern~\cite*{Karpukhin-Stern:2024:harm-map-in-higher-dim} recently proved
that for dimensions
$3 \leq m \leq 5$, maximizers
for $\mathcal{V}_1(g)$ exist and are realized by smooth,
sphere-valued harmonic maps. In this paper, we extend their result by
proving the existence of maximizers for all
$k$ and in all dimensions. Furthermore, we show that the maximizers are
smooth when
$3 \leq m \leq 6$ (see Corollary~\ref{cor:smooth-maximizers} below).
\begin{theorem}\label{thm:main}
  Let $(M,g)$ be a closed connected Riemannian manifold of dimension $m \geq 3$, and let $k\geq 1$.
  Then, for some finite-dimensional sphere $\Sph^n$, there exists a harmonic map $u \in H^1(M,\Sph^n)$ such that
  $u \in C^\infty(M\setminus\Sigma, \Sph^n)$, where $\Sigma$ is a closed set of Hausdorff dimension at most $m-7$, and
  \begin{equation}
    \lambda_k(\abs{du}^2_g dv_g) = 1,
    \qquad
    \mathcal{V}_k(g) = \int_M \abs{du}^2_g dv_g.
  \end{equation}
\end{theorem}
The equality $\lambda_k(|du|^2_g dv_g) = 1$ implies that $\ind (\Delta_g - |du|^2_g) \leq k$.
Every weakly harmonic map of finite (spectral) index is locally stable and even locally energy-minimizing when regarded as a map
  into $\Sph^\infty$ (see Definition~\ref{def:stable-map}, Lemma~\ref{lem:stable-energy-min}, and
  the proof of Proposition~\ref{prop:con-bilin-form}). Moreover, in the present closed connected setting, Corollary~\ref{cor:finite-sphere-map} shows that
the image of a finite-index weakly harmonic map into $\Sph^\infty$ lies in a finite-dimensional subsphere of $\Sph^\infty$. Here
$\Sph^\infty = \Sph(\ell^2)$ denotes
the unit sphere in the Hilbert space of square-summable sequences. See Section~\ref{sec:harm-maps}
for all necessary definitions concerning harmonic maps.
\begin{remark}
  As shown in Proposition~\ref{prop:no-atoms}, the maximizing sequence exhibits
  no concentration of energy — i.e., no bubbling — unlike in the case of eigenvalue maximization
  in dimension $m = 2$.
\end{remark}
As a consequence of Theorem~\ref{thm:main}, we extend \cite*[Theorem~1.5]{Karpukhin-Stern:2024:harm-map-in-higher-dim} to all $k \geq 1$ and to
the additional dimension $m=6$.
\begin{corollary}\label{cor:smooth-maximizers}
  If $3 \leq m \leq 6$,
  the supremum $\mathcal{V}_k(g)$ is realized by a smooth sphere-valued harmonic map.
\end{corollary}
Thus, at least in dimensions $3 \leq m \leq 6$, every closed manifold admits
infinitely many smooth harmonic maps into spheres, each of which
maximizes some eigenvalue functional $\overline{\lambda}_k$.
In contrast, Theorem~\ref{thm:2d-eigen-sphere-index} provides singular maximizers in every dimension $m \geq 7$.

\begin{remark}
  An analogue of Theorem~\ref{thm:main}
  holds in the context of Steklov eigenvalue optimization -- i.e.,
  when $\partial M \neq \varnothing$ and $\overline{\lambda}_k$ is maximized
  over measures supported on the boundary, $\supp \mu \subset \partial M$.
  See Theorem~\ref{thm:main-steklov} for the precise statement.
  The main remaining question in this setting concerns the regularity
  of maximizers.
\end{remark}

\subsection{Maximizing harmonic maps on spheres}
Let $(\Sph^m, g_{\Sph^m})$ denote the $m$-dimensional unit sphere with the
standard round metric, where $m\geq 3$. In \cite*[Theorem~1.6]{Karpukhin-Stern:2024:harm-map-in-higher-dim},
Karpukhin and Stern computed the supremum of the first eigenvalue over measures
on $\Sph^m$:
\begin{equation}\label{eq:V_1-sphere}
  \mathcal{V}_1(g_{\Sph^m}) = m \sigma_m
\end{equation}
where $\sigma_m = \operatorname{Vol} (g_{\Sph^m})$ is the volume of the round sphere. The standard volume measure
$dv_{g_{\Sph^m}}$, which corresponds to the harmonic map
$\id \colon \Sph^m \to \Sph^m$, is the unique maximizer for $\mathcal{V}_1(g_{\Sph^m})$.
The value of $\mathcal{V}_1(g_{\Sph^2}) = \Lambda_1([g_{\Sph^2}])$ was first computed by Hersch~\cite*{Hersch:1970:sphere-inequality}.

From~\eqref{eq:V_1-sphere}, we see that the maximizing harmonic map -- the identity map $\id_{\Sph^m}$ --
is smooth even if
$m \geq 7$. However, this is no longer the case for higher eigenvalues and
singularities do occur.
\begin{definition}
  Let $0\leq k\leq m-3$. We refer to the map $\SHM{k} \colon \Sph^m \to \Sph^{m-1-k}$
  defined by
  \begin{equation}
    \SHM{k}\colon (x,y) \mapsto \frac{x}{\abs{x}},
    \quad\text{where } (x,y) \in \Sph^{m}\subset \R^{m-k}\times\R^{k+1},
  \end{equation}
  as a \emph{generalized equator map}. It is a smooth map away from
  $\brc{0}\times\Sph^k$. Note that for $k = 0$,
  this is simply the radial projection onto the equator.
\end{definition}

Set $n := m - 1 - k$, so $n \geq 2$. In the coordinates $\Sph^{n} \times (0,1)\times \Sph^{k} \to \Sph^m$,
\begin{equation}\label{sec:calc-on-sphere}
  (\theta, t, \omega) \mapsto \brr{\theta\sqrt{1-t}, \omega\sqrt{t}},
\end{equation}
$\SHM{k}$ is given by
\begin{equation}
  \SHM{k}\colon (\theta, t, \omega) \mapsto \theta.
\end{equation}

In these coordinates, $g_{\Sph^m} = (1-t) g_{\theta} + \frac{1}{4}\frac{dt^2}{t(1-t)} + t g_{\omega}$, so
\begin{equation}
  \Delta_{g_{\Sph^m}} \SHM{k} = \frac{\Delta_{\theta} (\id_{\theta})}{1-t}
  = \frac{\abs{d(\id_{\theta})}^2_{\theta}}{1-t}
  = \abs{d \SHM{k}}^2\SHM{k},
\end{equation}
since the identity
$\id_{\theta} \colon \Sph^{n}\to \Sph^{n}$ is harmonic,
with $\abs{d(\id_{\theta})}^2_{\theta} = n$.
The energy of $\SHM{k}$ equals
\begin{equation}\label{eq:equator-energy}
  \begin{split}
    E[\SHM{k}] &= \int \abs{d\SHM{k}}^2 dv_{g_{\Sph^m}} = n
    \tfrac{\sigma_{n}\sigma_{k}}{2}\int_{0}^1 (1-t)^{\frac{n-1}{2}-1}t^{\frac{k-1}{2}}dt
    \\ &= n \tfrac{\sigma_{n}\sigma_{k}}{2} B(\tfrac{n-1}{2}, \tfrac{k+1}{2})
    = n \tfrac{n+k}{n-1} \brs{\tfrac{\sigma_{n}\sigma_{k}}{2} B(\tfrac{n+1}{2}, \tfrac{k+1}{2})}
    \\ &=  \tfrac{n(m-1)}{n-1}  \sigma_{m},
  \end{split}
\end{equation}
where $B(x,y)$ is the beta function.
Thus, $\SHM{k} \in H^1(\Sph^m,\Sph^{n})$ is a weakly harmonic map, whose
singularity set is precisely $\brc{0}\times\Sph^k$.

\begin{theorem}\label{thm:2d-eigen-sphere-upper-bound}
  Let $m \geq 3$ and $0\leq k\leq m-3$. For any
  (nonnegative) continuous measure $\mu \in \mathcal{M}_c(\Sph^m)$,
  one has
  \begin{equation}
    \overline{\lambda}_{k+2}(\mu) \leq E[\SHM{k}].
  \end{equation}
\end{theorem}
\begin{remark}
    The case of equality implies $\lambda_{k+2}(\mu)\mu = \abs{d(T\circ\SHM{k})}^2_{g_{\Sph^m}} dv_{g_{\Sph^m}}$, where
    $T = \id_{\Sph^n}$ if $n > 2$ and $T$ is a conformal automorphism of $\Sph^2$ if $n =2$.
  Therefore, the upper bound is not optimal when $k > m - 7$.

  The reason the map $\SHM{k}$ fails to be a maximizer for $k \geq m - 6$
  is that $\ind \SHM{k} := \ind (\Delta - \abs{d\SHM{k}}^2) = \infty$ in this case.  To see this, let
  us assume for simplicity that $3\leq m \leq 6$ and $k=0$.
  Then the essential spectrum of the Laplacian $\Delta$ on
  $L^2(\Sph^m, \abs{d\SHM{0}}^2dv_{g_{\Sph^m}})$
  begins no later
  than at $\frac{(m-2)^2}{4(m-1)} < 1$ for $m\leq 6$. This can be seen by analyzing the blow-up of
  $\SHM{0}$ near a pole so that
  $\epsilon^{m-2} (M_\epsilon)_*(\abs{d\SHM{0}}^2 dv_{g_{\Sph^m}})
  \to \frac{m-1}{\abs{x}^2}dv_{g_{\R^m}}$ on $\R^m$, where $M_\epsilon(x) = x/\epsilon$.
  Then, applying
  \cite[Lemma~1.3]{Schoen-Uhlenbeck:1984:min-harm-maps}, one obtains the lower bound
  $\frac{(m-2)^2}{4(m-1)}$, which is realized by pulling back suitable
  compactly supported radial test functions from $\R^m$.
  Thus, the (spectral) index of $\SHM{0}$ is infinite.
\end{remark}
\begin{theorem}\label{thm:2d-eigen-sphere-index}
  Let $\SHM{k}\colon\Sph^{m}\to\Sph^{m-1-k}$ be a generalized equator map and $0\leq k \leq m - 7$.
  Then $\ind \SHM{k} = k+2$, so the map $\SHM{k}$ is the unique maximizer
  realizing $\mathcal{V}_{k+2}(g_{\Sph^m})$.
\end{theorem}
Thus, $\SHM{k}$ is indeed a maximizer for $\mathcal{V}_{k+2}(g_{\Sph^m})$
when $0\leq k \leq m - 7$, and one cannot always expect the
smoothness of maximizers in Theorem~\ref{thm:main}.

In Corollary~\ref{cor:sph-dim-restric}, we establish a new necessary condition
for maximizing measures, which can in fact be shown to hold for critical
measures as well. As a consequence, the identity map
$\id \colon \Sph^m \to \Sph^m$
cannot realize $\mathcal{V}_2(g_{\Sph^m})$
since it is impossible to map $\Sph^m \to \Sph^{m-1}$ using only linear functions.

On the other hand, for $3\leq m \leq 6$, we know that there exists a smooth harmonic
map $u \colon \Sph^m \to \Sph^n$ realizing $\mathcal{V}_2(g_{\Sph^m})$. Its spectral index must be
$\ind u \leq 2$.
Since the identity map maximizes $\overline{\lambda}_1$ but not
$\overline{\lambda}_2$, any maximizer of $\overline{\lambda}_2$ must have spectral
index exactly $\ind u = 2$. Unfortunately, we do not know any natural
example of a harmonic map with
$\ind u = 2$. This leads to the following question:
\begin{question}
  Which smooth harmonic maps $u \colon \Sph^m \to \Sph^n$, with
  (spectral) index $\ind u = 2$,
  realize $\mathcal{V}_2(g_{\Sph^m})$
  for $3\leq m \leq 6$?
\end{question}

It is also natural to compare $\mathcal{V}_{k}(g_{\Sph^m})$ with the conformal class suprema
$\Lambda_k([g_{\Sph^m}])$, where
\begin{equation}\label{conf-class-opt}
    \Lambda_k([g]) = \sup_{\tilde{g} \in [g]} \overline{\lambda}_k(\tilde{g}),
  \quad \text{with } \overline{\lambda}_k(\tilde{g}) := \lambda_k(\tilde{g}) \Vol(\tilde{g})^{\frac{2}{m}}.
\end{equation}
El Soufi and Ilias \cite*{ElSoufi-Ilias:1986:hersch} computed
$\Lambda_1([g_{\Sph^m}])\sigma_m^{1-2/m} = m \sigma_m = \mathcal{V}_1(g_{\Sph^m})$,
thereby generalizing the classical result of Hersch for $m = 2$.
Thanks to~\cite*{Kim:2022:second-sphere-eigenval}, the supremum of the second eigenvalue in the conformal class
$[g_{\Sph^m}]$ is also known: $\Lambda_2([g_{\Sph^m}]) = m (2\sigma_m)^{\frac{2}{m}}$.
This yields the inequalities
\begin{equation}
  \mathcal{V}_{1}(g_{\Sph^m}) < \mathcal{V}_{2}(g_{\Sph^m}) < \Lambda_2([g_{\Sph^m}])\sigma_m^{1-\frac{2}{m}}.
\end{equation}

\subsection{Ideas of the proof of existence}

Recall that the existence of maximizers
on surfaces can be established by constructing a
maximizing sequence of probability measures $\mu_\epsilon$, converging
weakly$^*$ to some limiting measure $\mu$, i.e.
$\lambda_\epsilon := \lambda_k(\mu_\epsilon) \to \mathcal{V}_k$, together with
eigenmaps
$\Phi_\epsilon \in H^1(M,\R^{n_\epsilon})$ satisfying
$\Delta \Phi_\epsilon = \lambda_\epsilon \Phi_\epsilon \mu_\epsilon$ and
$\abs{\Phi_\epsilon}^2 \to 1$.
In dimension $m = 2$, it is possible to uniformly bound $n_\epsilon \leq n$
(the multiplicity is bounded by the genus). Then any weak limit
$\Phi$ of $\brc{\Phi_\epsilon}$ satisfies
$\abs{\Phi}^2 = 1$ by the compactness of
the embedding $H^1(M,\R^n) \hookrightarrow  L^2(M,\R^n)$.
One can further show that
$\Phi$ is a harmonic map and that the continuous part of $\mu$ is equal to
$\abs{d\Phi}^2dv_g$. Thus, modulo some bubble analysis near the atoms of $\mu$,
the proof is complete.

In dimensions $m \geq 3$, however, there is no known uniform bound on $n_\epsilon$.
A natural idea is to embed $\R^{n_\epsilon} \subset \R^\infty = \ell^2$
and attempt to carry out the same strategy. However, the embedding
$H^1(M,\R^\infty) \hookrightarrow  L^2(M,\R^\infty)$ is no longer
compact, so one cannot expect the equation $\abs{\Phi}^2 = 1$ to hold in the limit,
which is essential for the proof.

Instead of working directly in $H^1(M,\R^\infty)$, one may consider a different
functional space in which the limiting object $\Phi$
shares the property $\abs{\Phi}^2 = 1$. Such a space is suggested by viewing eigenvalues as
Lipschitz functionals on the space of compact
operators on a suitable Hilbert space,
$\lambda_k \colon \mathcal{K}_{sa}[H] \to \R$. The first variation of $\lambda_k$
then lies in the dual $\mathcal{K}_{sa}[H]^* \approx \mathcal{N}_{sa}[H]$, which is known
to be the space of nuclear (i.e. trace-class) operators. This motivates working with
operators (or tensors) $\mathfrak{t}_\epsilon = \sum \phi^i_\epsilon \otimes \phi^i_\epsilon \in \mathcal{N}[H^1(M)]$ rather than with the vector-valued map
$\Phi = (\phi^1_\epsilon,\cdots,\phi^{n_\epsilon}_\epsilon)$. Fortunately, the weak$^*$ topology of $\mathcal{N}[H^1(M)]$
is strong enough to ensure $\sum \abs{\phi^i}^2 = 1$ for any weak$^*$ limit
$\mathfrak{t} = \sum \phi^i \otimes \phi^i$ of $\brc{\mathfrak{t}_\epsilon}$.
Then Lemma~\ref{lem:tensor-decomp} allows us to adapt the remainder
of the proof as in the two-dimensional case
(cf.~\cite{vinokurov:2025:sym-eigen-val}).

\subsection{Discussion on other eigenvalue optimizations}
Traditionally,
the geometric optimization problem consists of identifying the suprema
$\Lambda_k([g])$~\eqref{conf-class-opt}. As shown below,
both $\mathcal{V}_k(g)$ and $\Lambda_k([g])$ are the borderline cases
of the following eigenvalue maximization problems parametrized
by $q \in [\frac{m}{m-2}, \infty]$.

Let us define
\begin{equation}
 \lambda_k(\mu, \sigma, g) := \inf_{F_{k+1}} \sup_{\phi \in F_{k+1}}
 \frac{\int \abs{d\phi}^2_g \sigma dv_g}{\int \phi^2 d\mu},
\end{equation}
and then
\begin{equation}\label{general-eigenval}
  \mathcal{V}_{k,q}(g) := \sup_{\substack{\mu, \sigma \in C^\infty\\ \mu, \sigma > 0}}
  \lambda_k(\mu, \sigma, g) \frac{\mu(M)}{\norm{\sigma}_{L^q}}.
\end{equation}
When $q < \frac{m}{m-2}$, the supremum $\mathcal{V}_{k,q}(g)$ becomes infinite.
This can be seen by choosing $\mu_\epsilon = \sigma_\epsilon =
\chi_{B_\epsilon(p)}$ to be the characteristic functions of a geodesic
ball. The infimum also becomes arbitrarily small, e.g., by fixing $\sigma = 1$
and letting $\mu_\epsilon \to \sum_{i=0}^{k} \delta_{p_i}$.

For $q = \infty$, note that $\frac{1}{\norm{\sigma}_{L^\infty}} \lambda_k(\mu, \sigma, g)
  \leq \lambda_k(\mu, g)$
since $\frac{\sigma}{\norm{\sigma}_{L^\infty}}\leq 1$. Hence, taking
$\sigma = 1$, we obtain the identity $\mathcal{V}_{k}(g) = \mathcal{V}_{k,\infty}(g)$.
Assuming $\Vol(g) = 1$, the Hölder inequality yields
\begin{equation}
  \mathcal{V}_{k}(g) = \mathcal{V}_{k,\infty}(g)
  \leq \mathcal{V}_{k,q}(g) \leq \mathcal{V}_{k,\frac{m}{m-2}}(g)
  \leq C([g])k^\frac{2}{m},
\end{equation}
where the final bound follows from
\cite[Theorem~2.1]{Colbois-ElSoufi-Savo:2015:-laplace-with-density}
(cf. also \cite[Theorem~5.3]{Grigoryan-Netrusov-Yau:2004:eignval-of-ellipt-oper}
with $\int_{B(x,r)} \sigma \leq C\norm{\sigma}_{L^{\frac{m}{m-2}}} r^2$).

In analogy with Section~\ref{subsec:clarke-subdif} (see also
\cite[Section~5]{Petrides-Tewodrose:2024:eigenvalue-via-clarke}),
the subdifferential
calculus shows that critical pairs $(\sigma, \mu)$ correspond
to $p$-harmonic maps $u \colon M \to \Sph^n$, where $p = \frac{2q}{q-1}$,
satisfying
\begin{equation}\label{p-crit-meas}
  \sigma = \abs{du}^{p-2}_g,\ \mu = \abs{du}^p_g dv_g,\
  \text{ and }\ \delta_g(\abs{du}_g^{p-2} du) = \abs{du}^p_g u.
\end{equation}

To recover $\Lambda_k([g])$, one may set $q = \frac{m}{m-2}$ and
$\mu = \sigma^{\frac{m}{m-2}} dv_g$ in~\eqref{general-eigenval}, which gives
\begin{equation}
  \Lambda_k([g]) \leq \mathcal{V}_{k,\frac{m}{m-2}}(g).
\end{equation}
However, as the critical point characterization~\eqref{p-crit-meas} suggests,
both $\Lambda_k([g])$ and $\mathcal{V}_{k,\frac{m}{m-2}}(g)$
share the same set of critical measures -- those generated by $m$-harmonic
maps into spheres. Therefore, if $\mathcal{V}_{k,\frac{m}{m-2}}(g)$
is achieved by such a map, one should expect
$\mathcal{V}_{k,\frac{m}{m-2}}(g) = \Lambda_k([g])$.
This leads to the following natural questions:
\begin{question}
  For $q \in [\frac{m}{m-2}, \infty)$ and $p = \frac{2q}{q-1}$,
  does there exist a $p$-harmonic map that realizes $\mathcal{V}_{k,q}(g)$?
\end{question}
A more specific case, formally implied by the preceding question, is the following:
\begin{question}\label{quest-m-2-equal}
  Does the equality $\mathcal{V}_{k,\frac{m}{m-2}}(g) = \Lambda_k([g])$ hold?
\end{question}

In this generalized setting, the quadratic form
$\phi \mapsto \int \abs{d\phi}^2 \sigma dv_g$ and
$\phi \mapsto \int \phi^2 d\mu$ are naturally defined on
$H^{1,p}(M)$. While one can still construct maximizing sequences
of “almost $p$-harmonic” tensors $\mathfrak{t}_\epsilon = \sum \phi^i_\epsilon \otimes \phi^i_\epsilon$,
the analysis must be carried out in the stronger norm $\norm{\mathfrak{t}_\epsilon}_p^{p/2} :=
\int [\mathfrak{t}_\epsilon]^\frac{p}{2} + \int \mathfrak{D}[\mathfrak{t}_\epsilon]^\frac{p}{2}$,
which dominates the nuclear norm $\norm{\mathfrak{t}_\epsilon}_{\mathcal{N}[H^1]}$, see~\eqref{eq:nucler-top-H1}.
The presence of a nontrivial density $\sigma$, which can be zero at some points,
further complicates establishing strong
convergence $\mathfrak{t}_\epsilon \to \mathfrak{t}$ in the norm $\norm{\cdot}_p$.

\subsection{Acknowledgements}
This paper forms part of the author's PhD thesis under the supervision
of Mikhail Karpukhin and Iosif Polterovich. The author is grateful to
both for their valuable guidance and many fruitful discussions.
In particular, the author thanks Mikhail Karpukhin for suggesting the
problem. This work was partially supported by an ISM Scholarship.

\section{Preliminaries}
Throughout the paper, we adopt the convention that $C, c, c_1, c_2, \cdots$ are
some positive constants. The dependencies
will be specified if necessary. $\Omega$ always stands for a (bounded) open set in $M$.
The manifold $M$ will be assumed to be closed and of dimension $m \geq 3$
unless specified otherwise.
By $\epsilon \in \brc{\epsilon_i}$ we
denote a sequence of positive numbers, and $o(1)$ denotes
any sequence such that
$o(1) \to 0$ as $\epsilon \to 0$.

\subsection{Eigenvalues of measures}
In~\cite[Proposition~1.1]{Kokarev:2014:measure-eigenval},
Kokarev proved the upper semi-continuity of $\lambda_k(\mu)$ on surfaces,
but one easily sees that the argument does not actually depend on the dimension of $M$.
\begin{proposition}\label{prop:upper-cont-and-bound}
  The functional $\mu \mapsto \lambda_k(\mu)$ is weakly$^*$ upper semi-continuous
  on $\mathcal{M}(M)$, i.e.
  \begin{equation}
    \limsup_n \lambda_k(\mu_n) \leq \lambda_k(\mu),
    \quad\text{where } \mu_n \overset{w^*}{\to} \mu \in \mathcal{M}(M).
  \end{equation}
\end{proposition}
\noindent For an open set $\Omega \subset M$ and a \emph{signed} Radon measure $\nu$, let us define the index
$\ind(\nu, \Omega) \in [0,\infty]$ as
\begin{equation}\label{eq:index}
  \ind(\nu, \Omega) =  \sup \brc{\dim V\colon V \subset C_{0}^\infty(\Omega)
  \text{ s.t. }Q_\nu|_{V\setminus\brc{0}} < 0},
\end{equation}
where
\begin{equation}
  Q_\nu(\phi) = \int \abs{d\phi}^2 \dv{g} - \int \phi^2 d\nu.
\end{equation}
We write $\ind \nu$ for $\ind (\nu, M)$. If
$\Omega' \subset \Omega$, it is clear that
$\ind(\nu, \Omega') \leq \ind(\nu, \Omega)$.
Following the ideas from~\cite[Section~4.1]{Karpukhin-Nadirashvili-Penskoi-Polterovich:2022:existence},
see also~\cite[Section~3.2]{Karpukhin-Nahon-Polterovich-Stern:2021:stability},
let us introduce the following definition.
\begin{definition}\label{def:stable}
  A point $p\in M$ is said to be \emph{stable}
  for the sequence of measures
  $\brc{\nu_\epsilon}$
  if there exists a neighborhood $\Omega$ of $p$ such that, up to a subsequence,
  one has
  \begin{equation}\label{good.pt}
    \ind(\nu_{\epsilon},\Omega) = 0.
  \end{equation}
  The point $p$ is called \emph{unstable} otherwise.
\end{definition}
The same definition also makes sense for a single measure $\nu$ (one can think
of the constant sequence $\nu_\epsilon \equiv \nu$).
\begin{remark}\label{rem:size_of-bad}
Suppose that $\ind \nu_\epsilon \leq k$ for every $\epsilon$. Then there are at most $k$
distinct \emph{unstable} points. Otherwise,
we could find $k+1$ disjoint open sets $\Omega_p\cap \Omega_q=\varnothing$ on which
$$
\ind(\nu_{\epsilon},\Omega_p) \neq 0.
$$
This contradicts the bound $\ind \nu_\epsilon \leq k$.
\end{remark}
\begin{remark}\label{rem:quadr-form-bound}
  After passing to the subsequence in Definition~\ref{def:stable}, one has
  \begin{equation}\label{ineq:quadr-form-bound}
    \int \phi^2 d\nu_\epsilon \le
    \int \abs{d\phi}^2 dv_g \quad \forall \phi\in C_{0}^{\infty}(\Omega).
  \end{equation}
  In particular, when $\nu_\epsilon \geq 0$, the measures
  $\brc{\nu_\epsilon}$ induce uniformly bounded bilinear forms on $H^1_{0}(\Omega)$.
\end{remark}
\begin{lemma}\label{lem:unstable-ring}
  Let $p \in M$ be an unstable point for $\brc{\nu_\epsilon}$ and
  $\ind \nu_\epsilon \leq k$. Then up to a subsequence, there exist a neighborhood
  $\Omega$ of $p$ and a sequence $r_\epsilon \to 0$ such that
  \begin{equation}
    \ind(\nu_{\epsilon},\Omega\setminus \overline{B_{r_\epsilon}(p)}) = 0.
  \end{equation}
\end{lemma}
\begin{proof}
  Let $B_{r_0} = B_{r_0}(p)$ and $\Omega := B_{r_0}$ be a neighborhood of $p$ not containing the other unstable points.
  Consider the set
  \begin{equation}
    A_\epsilon = \set{r\in (0,r_0)}{\ind(\nu_{\epsilon},\Omega\setminus \overline{B_r}) = 0}.
  \end{equation}
  Since $\ind \nu_\epsilon \leq k$, the set $A_\epsilon$ is not empty.
  Set $r_\epsilon := \inf A_\epsilon$ and notice that $\ind(\nu_{\epsilon},\Omega\setminus \overline{B_{r_\epsilon}}) = 0$.
  If $\liminf_\epsilon r_\epsilon \geq 2r_1$ for some $r_1 > 0$,
  then $\ind(\nu_{\epsilon},\Omega\setminus \overline{B_{r_1}}) \neq 0$, and
  after resetting $\Omega := B_{r_1}$, one can repeat the reasoning above until
  we get $k+1$ disjoint open sets $U_i := B_{r_i}\setminus\overline{B_{r_{i+1}}}$ such that
  $\ind(\nu_{\epsilon},U_i) \neq 0$ contradicting $\ind \nu_\epsilon \leq k$.
  Thus, we should have a subsequence $r_\epsilon \to 0$ at some iteration step.
\end{proof}
The preceding argument also yields a useful lower form bound for signed non-atomic Radon measures.
\begin{proposition}\label{prop:con-bilin-form}
  Let $\nu$ be a signed non-atomic Radon measure with $\ind \nu < \infty$.
  Then there exists a constant $C > 0$ such that
  \begin{equation}\label{ineq:meas-bound}
    \int \phi^2 d\nu \leq \int \abs{d\phi}^2dv_g + C \int \phi^2 dv_g
    \quad \forall \phi \in C^\infty(M).
  \end{equation}
  In particular, if $\mu \in \mathcal{M}_c(M)$ and $\lambda_k(\mu) > 0$ for some $k > 0$, then $\mu$ defines a continuous bilinear form
  on $H^1(M)$.
\end{proposition}
\begin{proof}
Let us show that for a constant sequence $\nu_\epsilon \equiv \nu$, every point of $M$
 is stable. Pick an unstable point $p \in M$ and apply Lemma~\ref{lem:unstable-ring}
 to the constant sequence $\nu_\epsilon \equiv \nu$. For some neighborhood
 $\Omega$ of $p$, we obtain
 \begin{equation}
  0 = \ind (\nu, \Omega \setminus \brc{p}) = \ind (\nu, \Omega),
 \end{equation}
 where the second equality follows from the non-atomicity of $\nu$ and the fact that
 $\brc{p}$ has zero capacity. Thus, every point $p$ has a neighborhood $U_p$ where \eqref{ineq:quadr-form-bound} holds.
 Then cover $M$ with finitely many such $U_i$ and take a partition of unity with $\sum_i \eta_i^2 = 1$ and $\supp \eta_i \subset U_i$. For any $\phi \in C^\infty_0(M)$,
 inequality \eqref{ineq:quadr-form-bound} yields
  \begin{equation}
      \int (\phi\eta_i)^2 d\nu \leq \int\abs{d\phi}^2\eta_i^2dv_g
      + \frac{1}{2} \int \brt{d\eta_i^2, d\phi^2}dv_g + \int\phi^2\abs{d\eta_i}^2dv_g.
  \end{equation}
After summing over $i$, the mixed terms vanish, and we arrive at \eqref{ineq:meas-bound}.

 For the last statement, since $\lambda_k(\mu) \in (0,\infty)$, rescale $\mu : = \lambda_k(\mu) \mu$.
 Then $\ind \mu \leq k < \infty$, and inequality~\eqref{ineq:meas-bound} holds. Together with the positivity of $\mu$, this yields
 a bilinear form bound on $H^1$.
\end{proof}

Let $\mathcal{B}[H^1]$ be the space of continuous bilinear forms on
$H^1 = H^1(M)$.
As noted above, we may restrict our attention to those measures that
belong to $\mathcal{B}[H^1]$, and we will do so from now on.
Consider the Hilbert space $X$,
\begin{equation}
  X = \set*{\phi \in H^1(M)}{\int \phi dv_g = 0},
  \quad
  \brr{\phi, \psi}_X := \int \brt{d\phi, d\psi} dv_g,
\end{equation}
and the $\mu$-average
\begin{equation}
  [\phi]_\mu := \frac{1}{\mu(M)}\int \phi d\mu.
\end{equation}
Since $\int \brr{\phi - [\phi]_\mu}^2 d\mu + \mu(M)[\phi]_\mu^2 = \int \phi^2 d\mu \geq
\int \brr{\phi - [\phi]_\mu}^2 d\mu$, we observe that
for $k \geq 1$, one has
\begin{equation}\label{eq:variat-character-reduced}
  \lambda_k(\mu) = \inf_{
    F_k}
    \sup_{\phi \in F_k\setminus\brc{0}}
    \frac{\int\abs{d\phi}^2 dv_g}{\int \brr{\phi - [\phi]_\mu}^2 d\mu},
\end{equation}
where $F_k \subset C^\infty \cap X$ with $\dim F_k = k$.
The Poincaré inequality easily
implies that the quadratic form $\phi \mapsto
\int \brr{\phi - [\phi]_\mu}^2 d\mu$ belongs to $\mathcal{B}[X]$ as long as
$\mu\in \mathcal{B}[H^1]$.
Thus, taking the inverses in \eqref{eq:variat-character-reduced}, we conclude
\begin{equation}\label{eq:variat-character-reversed}
  \lambda_k^{-1}(\mu) = \sup_{
    F_k \subset X}
    \inf_{\phi \in F_k\setminus\brc{0}}
    \frac{\int \brr{\phi - [\phi]_\mu}^2 d\mu}{\norm{\phi}_X^2}.
\end{equation}

The idea of viewing measures as bilinear forms on $H^1(M)$ in the context of eigenvalue optimization problems was first
introduced in~\cite{Petrides:2024:conf-class-opt, Petrides:2024:spec-optim-on-surfaces}.
In this work, however, we choose to work instead with bilinear forms on $X$,
which allows us to interpret the variational eigenvalues
$\lambda_k^{-1}(\mu)$ as the upper part of the spectrum of the bounded
linear operator associated with the form $\phi \mapsto \int \brr{\phi - [\phi]_\mu}^2 d\mu$.

\begin{proposition}\label{prop:eigen-upper-bound}
  Let $\mu \in \mathcal{M}_c(M)$, $k > 0$, and $\lambda_{k}(\mu) > 0$. Then
  there exist $k$ functions, $1 = \phi_0, \phi_1, \cdots, \phi_{k-1} \in H^1(M)$ (some
  of which may be equal to zero),
  satisfying $\Delta \phi_i  = \lambda_i \phi_i \mu$, with $\lambda_i < \lambda_k = \lambda_{k}(\mu)$,
  such that
  \begin{equation}\label{othog-inf-var-desc}
    \lambda_{k}(\mu) = \inf \set*{\frac{\int \abs{d\phi}^2_g dv_g}{\int \phi^2 d\mu}}
    {\phi \in H^1(M),\ \int \phi \phi_i d\mu = 0\ \forall \lambda_i < \lambda_k}.
  \end{equation}
  The infimum is attained only on the eigenfunctions $\Delta \phi = \lambda_k(\mu) \phi \mu$.
\end{proposition}
\begin{proof}
  By Proposition~\ref{prop:con-bilin-form}, $\mu \in \mathcal{B}[H^1]$. In that case,
  we saw in \eqref{eq:variat-character-reversed} that $\lambda_k^{-1}$ is actually
  associated with a (self-adjoint) bounded linear operator on $X$, and
  the order $k$ means that the spectrum above $\lambda_k^{-1}$ is discrete and
  consists of at most $k-1$ eigenvalues
  (counted with multiplicities),
  $\lambda^{-1}_i > \lambda_k^{-1}$, e.g., see~\cite*[Section~4.5]{Davies:1995:spectral-theory}.
  Being an eigenvalue $\lambda^{-1}_i$ means there exists an eigenfunction
  $\tilde{\phi}_i \in X$ such that
  \begin{equation}
    \int (\tilde{\phi}_i - [\tilde{\phi}_i]_\mu)(\psi - [\psi]_\mu) d\mu
     = \lambda^{-1}_i \int \brt{d\tilde{\phi}_i, d\psi} dv_g
  \end{equation}
  for all $\psi \in X$. By setting $\phi_i = \tilde{\phi}_i - [\tilde{\phi}_i]_\mu$,
  one can check that $\Delta \phi_i = \lambda_i \phi_i \mu$ weakly on all
  $H^1$. Then the variation description \eqref{othog-inf-var-desc} is just
  a restatement of the well-known variational characterization
  in the case of the eigenvalues of a bounded operator:
  \begin{equation}
    \lambda^{-1}_{k}(\mu) = \sup \set*{\frac{\int (\tilde{\phi} - [\tilde{\phi}]_\mu)^2 d\mu}{\int \abs{d\tilde{\phi}}^2_g dv_g}}
    {\tilde{\phi} \in X,\ \int \brt{d\tilde{\phi}, d\tilde{\phi}_i}dv_g = 0\ \forall \lambda_i < \lambda_k}.
  \end{equation}
\end{proof}
Slightly abusing the notation in the next lemma, we write $\phi \mapsto \int \phi d\mu$
for functionals even if they are not induced by a measure. The idea is that
we can pick, e.g., any element $\mu \in H^{1,\frac{m}{m-1}}(M)^*$ such that
$\brt{\mu, 1} \neq 0$ and define
\begin{equation}
  \mathfrak{q}(\mu)[\phi] =  \int(\phi - [\phi]_\mu)^2 d\mu.
\end{equation}
\begin{lemma}\label{lem:sobolev-dual-meas}
  Let $m \geq 3$ and $U = \set*{\mu \in \brr{H^{1,\frac{m}{m-1}}}^*}{\brt{\mu, 1} \neq 0}$.
  Then the map $\mu \mapsto \mathfrak{q}(\mu)$ is
  continuously Fréchet differentiable as a map
  $U \to \mathcal{B}[X]$. Its derivative
  at $\mu$ is given by
  \begin{equation}
    d_\mu \mathfrak{q}(\nu) =  \int \brr{\phi - [\phi]_\mu}\brr{\psi - [\psi]_\mu} d \nu,
  \end{equation}
  where $\nu \in \brr{H^{1,\frac{m}{m-1}}}^*$.
  The bounded linear operator $T(\mu) \in \mathcal{L}[X]$
  induced by $\mathfrak{q}$ is self-adjoint and compact.
\end{lemma}
\begin{proof}
  Notice that the quadratic form
  \begin{equation}
    \int \brr{\phi - [\phi]_\mu}^2 d\mu = \int \phi^2 d\mu -
    \brr{\frac{1}{\brt{\mu, 1}}\int \phi d\mu}^2
  \end{equation}
  depends rationally on the expressions of the form
  $\int \phi\psi d\mu$, where $\phi,\psi \in H^1$. So, it is enough
  to prove the (strong) continuity of $\mu \mapsto \int \phi\psi d\mu$,
  which follows from the estimate, see~\cite[Lemma~4.9]{Girouard-Karpukhin-Lagace:2021:continuity-of-eigenval},
  \begin{equation}
    \abs{\int \phi\psi d\mu} \leq
    C \norm{\phi}_{H^1} \norm{\psi}_{H^1} \norm{\mu}_{\brr{H^{1,\frac{m}{m-1}}}^*}
  \end{equation}
  and the fact that $H^1$-norm is an equivalent
  norm on $X$ by the Poincaré inequality.

  The compactness of $T(\mu)$ follows from the facts that the
  images of $L^\infty$
  measures are dense in $\brr{H^{1,\frac{m}{m-1}}}^*$, and for
  $\mu = \rho \dv{g} \in L^\infty$, one has
  $T(\mu) = i^* M_\rho i$, where
  $i\colon X \hookrightarrow  L^2$ is
  compact.
\end{proof}
In what follows, the notation $\mathcal{M}\cap L^p$ refers to the set of nonnegative,
nonzero measures in $L^p$.
\begin{proposition}
  The following equalities hold:
  \begin{equation}
    \mathcal{V}_k(g)
    = \sup_{\mu \in \mathcal{M}\cap L^\infty} \overline{\lambda}_k(\mu)
    = \sup_{\mu \in \mathcal{M}\cap L^1} \overline{\lambda}_k(\mu).
  \end{equation}
\end{proposition}
\begin{proof}
  To deduce the first equality, notice that
  by Lemma~\ref{lem:sobolev-dual-meas},
  the closure of $\set{\mathfrak{q}(\mu)}{\mu \in C^\infty}$
  in the sup-norm on $\mathcal{B}[X]$ coincides
  with the corresponding closure of
  \begin{equation}
    \overline{\set*{\mathfrak{q}(\mu)}
    {\mu \in \brr{H^{1,\frac{m}{m-1}}}^*}}^{\mathcal{B}[X]}.
  \end{equation}
  Then $L^\infty \subset \brr{H^{1,\frac{m}{m-1}}}^*$, and eigenvalues
  $\lambda_k(\mu)$ depend continuously with respect to the sup-norm on
  $\mathcal{B}[X]$.

  Now, assume $\mu = \rho dv_g \in L^1$. Take $\mu_n = \rho_n dv_g$ with
  $\rho_n = \min \brc{n, \rho} \in L^\infty$. Since $\mu \geq \mu_n$,
  the variational characterization~\eqref{meas-eigenval} yields
  $\lambda_k(\mu) \leq \lambda_k(\mu_n)$. At the same time,
  $\mu_n(M) \to \mu(M)$, and hence
  \begin{equation}
    \sup_{\mu \in \mathcal{M}\cap L^1} \overline{\lambda}_k(\mu)
    \leq \sup_{\mu \in \mathcal{M}\cap L^\infty} \overline{\lambda}_k(\mu).
  \end{equation}
  Therefore, equality must hold.
\end{proof}

\subsection{Nuclear operators on Hilbert spaces}\label{sec:nuc-operators}

Let $H$ be a Hilbert space. We identify $H^* \approx H$. Then
the space of (bounded) bilinear forms $\mathcal{B}[H]$ becomes isomorphic
to the space of (bounded) linear operators $\mathcal{L}[H]$.
For an element $\mathfrak{t} \in H \otimes H$ (of the algebraic tensor product), consider its projective norm
\begin{equation}
  \norm{\mathfrak{t}}_\pi = \inf \set*{\sum \norm{x_i}\norm{y_i}}{\mathfrak{t} = \sum x_i \otimes y_i}.
\end{equation}
The closure of $(H \otimes H, \norm{\cdot}_{\pi})$ is denoted by $H \hat{\otimes}_\pi H$.
Then $H \hat{\otimes}_\pi H$ (or more naturally, $H \hat{\otimes}_\pi H^*$)
is isomorphic to the space of nuclear (or trace class) operators $\mathcal{N}[H]$
so that
\begin{equation}
  x \otimes y \mapsto \brr{z \mapsto x\brt{y,z}}.
\end{equation}
The predual of $\mathcal{N}[H]$ is the space of compact operators $\mathcal{K}[H]$, i.e.
$\mathcal{K}[H]^* \approx \mathcal{N}[H]$, where the duality pairing is given as
\begin{equation}\label{eq:duality}
  (\mathfrak{q} , \mathfrak{t}) \mapsto \Tr \mathfrak{q}\mathfrak{t}, \quad \mathfrak{q} \in \mathcal{K}[H],\ \mathfrak{t} \in \mathcal{N}[H],
\end{equation}
and the dual is $\mathcal{L}[H] = \mathcal{N}[H]^*$.
Given the isomorphism $H \hat{\otimes}_\pi H \approx \mathcal{N}[H]$ and
$\mathfrak{t} \in \mathcal{N}[H]$, one has
\begin{equation}
  \norm{\mathfrak{t}}_\pi = \norm{\mathfrak{t}}_{\mathcal{N}[H]} := \Tr \sqrt{\mathfrak{t}^* \mathfrak{t}}.
\end{equation}
In particular, for a positive self-adjoint tensor $\mathfrak{t} = \sum x_i \otimes x_i$,
the norm is simply the trace,
\begin{equation}\label{eq:positive-tens-norm}
  \norm{\mathfrak{t}}_{\mathcal{N}[H]} = \Tr \mathfrak{t} = \sum \norm{x_i}^2.
\end{equation}

If $\mathfrak{a},\mathfrak{b} \colon H_1 \to H_2$ are two bounded linear operators between Hilbert spaces,
one naturally defines the operator $\mathfrak{a} \hat{\otimes}_\pi \mathfrak{b} \colon H_1 \hat{\otimes}_\pi H_1
\to H_2 \hat{\otimes}_\pi H_2$ acting by $x \otimes y \mapsto \mathfrak{a}(x) \otimes \mathfrak{b}(y)$.
Note that
\begin{equation}
  \norm{\mathfrak{a} \hat{\otimes}_\pi \mathfrak{b}} \leq \norm{\mathfrak{a}} \norm{\mathfrak{b}},
\end{equation}
and the operator $\mathfrak{a} \hat{\otimes}_\pi \mathfrak{b}$ transforms into $\mathfrak{t}
\mapsto \mathfrak{a} \mathfrak{t} \mathfrak{b}^*$ under the isomorphisms $H_i \hat{\otimes}_\pi H_i \approx \mathcal{N}[H_i]$,
with $\mathfrak{t} \in \mathcal{N}[H_1]$.

\begin{proposition}\label{prop:weak-star-continuity}
  Let $\mathfrak{a} \colon H_1 \to H_2$ be a continuous linear map of Hilbert spaces. Then
  the induced map $\mathfrak{a} \hat{\otimes}_\pi \mathfrak{a} \colon \mathcal{N}[H_1] \to \mathcal{N}[H_2]$
  is weakly$^*$ continuous.
\end{proposition}
\begin{proof}
  It is enough to show that the adjoint operator maps
  $\mathcal{K}[H_2]$ into $\mathcal{K}[H_1]$.
  The duality~\eqref{eq:duality} implies that the adjoint
  of $\mathfrak{t} \mapsto \mathfrak{a} \mathfrak{t} \mathfrak{a}^*$ is simply
  \begin{equation}
    \mathfrak{q} \mapsto \mathfrak{a}^* \mathfrak{q} \mathfrak{a}, \quad \mathfrak{q}\in\mathcal{K}[H_2],
  \end{equation}
  and $\mathfrak{a}^* \mathfrak{q} \mathfrak{a}$ is again compact.
\end{proof}
\begin{proposition}\label{prop:compact-maps}
  Let $\mathfrak{c} \colon H_1 \to H_2$ be a compact linear map of Hilbert spaces. Then
  the induced map $\mathfrak{c} \hat{\otimes}_\pi \mathfrak{c} \colon \mathcal{N}[H_1] \to \mathcal{N}[H_2]$
  is compact and maps weakly$^*$ convergent sequences to strongly
  convergent ones.
\end{proposition}
\begin{proof}
  A compact map between Hilbert spaces is a norm-limit of finite
  rank operators $\mathfrak{c}_n \to \mathfrak{c}$. Then we have
  \begin{equation}
    \mathfrak{c} \hat{\otimes}_\pi \mathfrak{c} - \mathfrak{c}_n \hat{\otimes}_\pi \mathfrak{c}_n
     = (\mathfrak{c} - \mathfrak{c}_n) \hat{\otimes}_\pi \mathfrak{c} +
    \mathfrak{c}_n \hat{\otimes}_\pi (\mathfrak{c} - \mathfrak{c}_n),
  \end{equation}
  so $\norm{\mathfrak{c} \hat{\otimes}_\pi \mathfrak{c} - \mathfrak{c}_n \hat{\otimes}_\pi \mathfrak{c}_n}
  \leq (\norm{\mathfrak{c}} + \norm{\mathfrak{c}_n}) \norm{\mathfrak{c} - \mathfrak{c}_n} \to 0$.
  Thus, $\mathfrak{c} \hat{\otimes}_\pi \mathfrak{c}$ is also a norm-limit of
  finite rank operators, hence compact.

  Now, let $\mathfrak{t}_\epsilon \overset{w^*}{\to} \mathfrak{t} \in \mathcal{N}[H]$.
  Then the precompact sequence $\mathfrak{c} \mathfrak{t}_\epsilon \mathfrak{c}^*$ has
  $\mathfrak{c} \mathfrak{t} \mathfrak{c}^*$ as its only accumulation point, due to
  the weak$^*$ convergence implied by Proposition~\ref{prop:weak-star-continuity}.
  This means that $\mathfrak{c} \mathfrak{t}_\epsilon \mathfrak{c}^*$ converges
  strongly to $\mathfrak{c} \mathfrak{t} \mathfrak{c}^*$.
\end{proof}

\begin{lemma}\label{lem:tensor-decomp}
  Let $\mathfrak{t}_\epsilon \overset{w^*}{\to} \mathfrak{t} \in \mathcal{N}[H]$. Then
  there exists a collection of finite rank orthogonal projectors $\mathfrak{p}_\epsilon$ on $H$
  such that
  $\mathfrak{p}_\epsilon \mathfrak{t}_\epsilon \mathfrak{p}_\epsilon \to \mathfrak{t}$ in $\mathcal{N}[H]$
  and $ \mathfrak{q}_\epsilon \mathfrak{t}_\epsilon  \mathfrak{q}_\epsilon \overset{w^*}{\to} 0$,
  with $\mathfrak{q}_\epsilon := 1 - \mathfrak{p}_\epsilon$.
\end{lemma}
\begin{proof}
  One can write $\mathfrak{t} = \sum x_i \otimes y_i$ with
  $\sum \norm{x_i}\norm{y_i} < \infty$. Then we decompose
  it,
  \begin{equation}
    \mathfrak{t} = \sum_{i \leq n} x_i \otimes y_i + \mathfrak{r},
  \end{equation}
  so that $\norm{\mathfrak{r}}_\pi < \delta$. Let $\mathfrak{p}_n$ be the orthogonal
  projector onto $V_n = Span \brc{x_i, y_i}_{i \leq n}$. On the one hand,
  \begin{equation}
    \norm{\mathfrak{p}_n \mathfrak{t} \mathfrak{p}_n - \mathfrak{t}}_{\pi} =
    \norm{\mathfrak{p}_n \mathfrak{r} \mathfrak{p}_n - \mathfrak{r}}_{\pi} < 2\delta
    \quad\text{and}\quad
    \norm{\mathfrak{q}_n \mathfrak{t} \mathfrak{q}_n}_{\pi} =
    \norm{\mathfrak{q}_n \mathfrak{r} \mathfrak{q}_n}_{\pi} < \delta
  \end{equation}
  implying $\mathfrak{p}_n \mathfrak{t} \mathfrak{p}_n \to \mathfrak{t}$
  and $\mathfrak{q}_n \mathfrak{t} \mathfrak{q}_n \to 0$ in $\mathcal{N}[H]$.
  On the other hand, the convergence
  $\mathfrak{t}_\epsilon \overset{w^*}{\to} \mathfrak{t}$ yields
  $\mathfrak{p}_n \mathfrak{t}_\epsilon \mathfrak{p}_n \to
  \mathfrak{p}_n \mathfrak{t} \mathfrak{p}_n$
  and $\mathfrak{q}_n \mathfrak{t}_\epsilon \mathfrak{q}_n \overset{w^*}{\to}
  \mathfrak{q}_n \mathfrak{t} \mathfrak{q}_n$ for a fixed $n$.
  Thus, one can set $\mathfrak{p}_\epsilon = \mathfrak{p}_n$ for an appropriate
  $n = n(\epsilon)$.
\end{proof}
A proof of the next proposition can be found, e.g., in the appendix of~\cite*{Arazy:1981:convergence--unitary-matrix-spaces}.
\begin{proposition}\label{prop:tensor-convergence}
  Let $\mathfrak{t}_\epsilon \overset{w^*}{\to} \mathfrak{t} \in \mathcal{N}[H]$
  and $\norm{\mathfrak{t}_\epsilon}_{\mathcal{N}[H]} \to \norm{\mathfrak{t}}_{\mathcal{N}[H]}$. Then
  $\mathfrak{t}_\epsilon \to \mathfrak{t}$ (strongly) in $\mathcal{N}[H]$.
\end{proposition}

Let $U \subset M$ be an open set. Denote by
\begin{equation}\label{eq:nucler-top-H1}
  \m{\cdot}\colon L^2(U)\hat{\otimes}_\pi L^2(U) \to L^1(U)
\end{equation}
the linear map induced by the multiplication of functions, i.e.
$f \otimes g \mapsto fg$, and by
\begin{equation}
  \md{\cdot}\colon H^1(U)\hat{\otimes}_\pi H^1(U) \to L^1(U)
\end{equation}
the linear map induced by $f \otimes g \mapsto \brt{df,dg}$.
We keep the
same symbol $\m{\cdot}$ for the map $\m{\cdot}\colon H^1\hat{\otimes}_\pi H^1 \to L^1$.
Then equation~\eqref{eq:positive-tens-norm} reads as
\begin{equation}\label{eq:tens-norm-decomp}
  \norm{\mathfrak{t}}_{\tensor{H^1(U)}} = \int_U \m{\mathfrak{t}}dv_g + \int_U \md{\mathfrak{t}}dv_g
  \quad \text{when}\quad \mathfrak{t} = \sum \phi^i \otimes \phi^i.
\end{equation}
\begin{remark}\label{rem:weak-star-conv}
  Consider two open sets $V \subset U$ and the induced restriction map
  $\mathfrak{a}\colon H^1(U) \to H^1(V)$. If  $\mathfrak{t}_\epsilon \overset{w^*}{\to}
  \mathfrak{t} \in \mathcal{N}[H^1(U)]$, it follows from Proposition~\ref{prop:weak-star-continuity}
  that  $\mathfrak{t}_\epsilon|_V \overset{w^*}{\to}
  \mathfrak{t}|_V \in \mathcal{N}[H^1(V)]$, where  $\mathfrak{t}|_V := \mathfrak{a} \mathfrak{t} \mathfrak{a}^*$.
\end{remark}
As a corollary of Proposition~\ref{prop:compact-maps}, we obtain a generalization
of the weak$^*$-to-strong convergence under the compact embedding $H^1 \hookrightarrow L^2$.
\begin{corollary}
  \label{lem:l-inf-convergence}
  Let $\Omega \subset M$ be a bounded domain and $\mathfrak{t}_\epsilon \overset{w^*}{\to}
  \mathfrak{t} \in \mathcal{N}[H^1(\Omega)]$. Then
  $\mathfrak{t}_\epsilon \to \mathfrak{t}$ in $\mathcal{N}[L^2(\Omega)]$. In particular,
  $\m{\mathfrak{t}_\epsilon} \to \m{\mathfrak{t}}$ in $L^1(\Omega)$.
\end{corollary}

\subsection{Clarke subdifferential}\label{subsec:clarke-subdif}
Let $E$ be a Banach space, $U \subset E$ be an open subset,
and $f\colon U \to \mathbb{R}$ be a locally Lipschitz function.
One defines \emph{Clarke directional derivative} of $f$ at
$x \in U$ in the direction $v \in E$ as follows:
\begin{equation}
 f^{\circ}(x; v) = \limsup_{\tilde{x} \to x, \tau\downarrow 0}
 \frac{1}{\tau} \brs{f(\tilde{x} + \tau v) - f(\tilde{x})}.
\end{equation}
Since $f^{\circ}(x; v)$ is a sublinear and Lipschitz-continuous function of $v$,
it is a support functional of a convex closed subset of the dual $E^*$.
The Clarke subdifferential of $f$ at the point $x$ is defined to be
precisely this subset of $E^*$:
\begin{equation}
  \partial_C f(x) = \set*{\xi \in E^*}{\brt{\xi, v}
  \leq f^{\circ}(x; v) \ \forall v \in E}.
\end{equation}
The following properties of
$\partial_C f(x)$ can be found, e.g., in~\cite[Section~7.3]{Schirotzek:2007:nonsmooth-analysis}:
\begin{itemize}
  \item $\partial_C f(x) \neq \varnothing$;
  \item $\partial_C (\alpha f)(x) = \alpha\partial_C f(x)$ for $\alpha \in \R$;
  \item $\partial_C (f+g)(x) \subset \partial_C f(x) + \partial_C g(x)$;
  \item $\partial_C f(x)$ is a closed, convex, and bounded set, hence weakly$^*$ compact;
  \item $0 \in \partial_C f(x)$ if $x$ is a local maximum/minimum.
\end{itemize}
The Clarke subdifferential
generalizes the notion
of the derivative in the case of locally Lipschitz functions
and captures many useful properties of the derivative that
are used in optimization problems.
Different applications of the Clarke subdifferential to critical metrics
and eigenvalue optimization problems can be found in~\cite{Petrides-Tewodrose:2024:eigenvalue-via-clarke}.

\begin{proposition}[Chain rule]\label{prop:clarke-compose}
  Let $g \colon F \to E$ be a continuously Fréchet differentiable map of Banach spaces defined
  on a  neighborhood of $x\in F$. If
  $f\colon E \to \mathbb{R}$ is a Lipschitz function
  on a neighborhood of $g(x) \in E$, one has
  \begin{equation}
    \partial_C (g^*f)(x) \subset (d_x g)^*[\partial_C f(g(x))].
  \end{equation}
  Thus, the subdifferential of a pullback is contained in the pullback of the subdifferential.
\end{proposition}
\begin{proof}
  Since
  \begin{equation}
    \frac{1}{\tau}\norm{g(\tilde{x} + \tau v) - g(\tilde{x}) - d_{x} g(\tau v)}
    \leq \sup_{0\leq \theta \leq \tau} \norm{d_{\tilde{x} + \theta v} g(v) - d_{x} g(v)}
  \end{equation}
  and $d_{\tilde{x} + \theta v} g(v) \to d_x g(v)$ as $\tilde{x} \to x$, $\tau\downarrow 0$,
  one easily concludes that
  \begin{align}
    (g^*f)^{\circ}(x; v) &= \limsup_{\tilde{x} \to x, \tau\downarrow 0}
    \frac{1}{\tau} \brs{f(g(\tilde{x} + \tau v)) - f(g(\tilde{x}))}
    \\ &\leq \limsup_{\tilde{x} \to x, \tau\downarrow 0}
    \frac{1}{\tau} \brs{f(g(\tilde{x} + \tau v)) - f(g(\tilde{x})+ d_{x} g(\tau v))}
    \\ &+ \limsup_{\tilde{x} \to x, \tau\downarrow 0}
    \frac{1}{\tau} \brs{f(g(\tilde{x})+ \tau d_{x}g(v)) - f(g(\tilde{x}))}
    \\ & \leq f^{\circ}(g(x); d_x g(v)).
  \end{align}
  Then the linear map $d_x g(v) \mapsto \brt{\zeta,v}$,
  $\zeta \in \partial_C (g^*f)(x)$, is well defined and continuous
  since
  \begin{equation}
    \pm\brt{\zeta,v} \leq f^{\circ}(g(x); \pm d_x g(v)) \leq C \norm{d_x g(v)}.
  \end{equation}
  By the Hahn--Banach theorem, it can be extended to a map $\xi \in E^*$ such that
  $\brt{\xi,w} \leq f^{\circ}(g(x); w)$, $w\in E$. Thus, $\xi \in \partial_C f(g(x))$
  and $(d_x g)^*\xi = \zeta$.
\end{proof}
The following proposition is a simplified version of~\cite[Theorem~12.4.1]{Schirotzek:2007:nonsmooth-analysis}
sufficient for our purposes.
\begin{proposition}[Clarke’s Multiplier Rule]\label{prop:multip-rule}
  Let $f\colon U \to \mathbb{R}$ be a locally Lipschitz function defined
  on an open set $U$ of a Banach space $E$. Suppose that $U$ contains
  a closed convex set $S$, $x\in S$, and $f(x) = \min_{\tilde{x} \in S} f(\tilde{x})$.
  Then
  \begin{equation}
    \exists \xi \in \partial_C f(x)\colon
    \brt{\xi, \tilde{x}-x} \geq 0 \quad\forall \tilde{x} \in S.
  \end{equation}
\end{proposition}
Note that for unbounded self-adjoint operators, eigenvalues are indexed in
increasing order (from bottom to top). In contrast, for
bounded operators, we will index them in decreasing order
(from top to bottom), starting with $k=1$.
\begin{proposition}\label{prop:subdiff-calc}
  Let $H$ be a Hilbert space and
  $\mathcal{K}_{sa}[H]$ be the space of compact
  self-adjoint operators.
  If $T \in \mathcal{K}_{sa}[H]$ is a positive operator with
  $\lambda_k(T) > 0$, then the functional
  $\lambda_k\colon \mathcal{K}_{sa}[H] \to \R$ is Lipschitz on a neighborhood
  of $T$. Its Clarke subdifferential $\partial_C \lambda_k(T) \subset \mathcal{N}_{sa}[H]$
  is given by
  \begin{equation}\label{abst-clark-subdif}
    \partial_C \lambda_k(T) = co\
    \set*{\phi \otimes \phi }{
      \phi \in E_k(T),\ \norm{\phi} = 1
    },
  \end{equation}
  where $\text{co}\,K$ denotes the convex hull of the set $K$ and
  $E_k(T)$ is the eigenspace associated with $\lambda_k(T)$.
\end{proposition}
\begin{proof}
  For a direct argument establishing the inclusion
  $\partial_C \lambda_k(T) \subset$ in~\eqref{abst-clark-subdif},
  see also \cite[Lemma~2.5]{vinokurov:2025:sym-eigen-val}. We present
  here an indirect approach, reducing the computations to the finite-dimensional
  case.
  For a finite-dimensional $H$, a proof can be found, e.g.,
  in~\cite[Corollary 10]{Lewis:1999:eigenval-subdiff}.

  Now, let the dimension of $H$ be arbitrary, $E := E_k(T)$, and
  decompose $H = E \oplus E^{\bot}$. This decomposition yields a
  natural linear inclusion
  $\iota_0\colon \mathcal{K}_{sa}[E] \hookrightarrow \mathcal{K}_{sa}[H]$,
  where $\iota_0(S)$ acts by zero on $E^{\bot}$. Using this, define an affine map
  $\iota\colon \mathcal{K}_{sa}[E] \hookrightarrow \mathcal{K}_{sa}[H]$
  by
  \begin{equation}
    \iota(S) = \iota_0(S - T|_E) + T
  \end{equation}
  so that $\iota(T|_E) = T$. Since $\iota$ is affine, it is obviously $C^1$
  (in fact, smooth). Note that $T|_E = \lambda \id$, where
  $\lambda = \lambda_k(T)$.

  Let $B_\epsilon(T) \subset \mathcal{K}_{sa}[H]$ and
  $B_\epsilon(T|_E) \subset \mathcal{K}_{sa}[E]$
  denote the $\epsilon$-balls around $T$ and $T|_E$. For sufficiently small
  $\epsilon$, we construct
  a smooth (even $C^1$ suffices) left inverse $r \colon B_\epsilon(T) \to B_\epsilon(T|_E)$
  to $\iota$, i.e.
  $r \circ \iota = \id$. Indeed, if $S \in B_\epsilon(T)$ and
  $P_S$ denotes the spectral projector onto the $(\dim E)$-dimensional
  space corresponding to the eigenvalues of $S$ that are
  close to $\lambda_k(T)$, then one can apply the Gram-Schmidt process
  to construct a smooth family of orthogonal transformations
  $U_S\colon E \to im\, P_S$ such that $U_T = \id$.
  We then set
  \begin{equation}
    r(S) = (U_S^* S U_S)|_E.
  \end{equation}
  See \cite[Lemma~2.5]{vinokurov:2025:sym-eigen-val}
  or \cite{Kato:1995:perturbation-theory} for more details.
  Since $P_{\iota(S)} = P_T$ for $S \in B_\epsilon(T|_E)$
  when $\epsilon$ is small enough, the Gram-Schmidt process
  yields $U_{\iota(S)} = \id$, and thus $r \circ \iota = \id$, as claimed.

  Let $s$ be the position of $\lambda_k(T)$
  when restricted to the subspace $E$, i.e. $k = k_{min} - 1 + s$,
  where $k_{min} = \min \set{k' \in \mathbb{N}}{\lambda_{k'} = \lambda_k}$.
  Then $\lambda_s(S_1) = \lambda_k(\iota(S_1))$,
  $\lambda_k(S_2) = \lambda_s(r(S_2))$, and Proposition~\ref{prop:clarke-compose} yields the inclusions
  \begin{equation}
   \partial_C \lambda_k(T) \subset
    (d_T r)^*[\partial_C \lambda_{s}(T|_E)],
  \end{equation}
  \begin{equation}\label{3-subdif-incl}
    \partial_C \lambda_{s}(T|_E) \subset (d_{T|_E}\iota)^* [\partial_C \lambda_k(T)] \subset
    \partial_C \lambda_{s}(T|_E).
  \end{equation}
  We already know the subdifferential characterization~\eqref{abst-clark-subdif}
  is valid for self-adjoint operators on $E$.
  Let $\phi \otimes \phi$, with $\phi \in E$, and consider a perturbation
  $S = T + \tau T_1$. Then
  \begin{equation}
    \brt{(d_T r)^*[\phi \otimes \phi], T_1} = \frac{d}{d\tau}\biggr|_{\tau = 0}
    \brt{SU_S \phi, U_S \phi} = \brt{T_1 \phi, \phi}
  \end{equation}
  since $\frac{d}{d\tau} U_S \phi \bot \phi$, and $T$ is diagonal on $E$.
  Therefore, the pullback $(d_T r)^*$ acts trivially:
  $\phi \otimes \phi \mapsto \phi \otimes \phi$. Similarly, one can check that
  $(d_{T|_E}\iota)^*$ acts trivially on $\phi \otimes \phi$ as well.
  Hence, both inclusions in~\eqref{3-subdif-incl} are equalities, completing the proof.
\end{proof}
Let $E_k(\mu) \subset H^1(M)$ be
the space of eigenfunctions corresponding
to $\lambda_k(\mu)$.
\begin{corollary}\label{cor:subdiff-calc-meas}
  Let $\mu \in \mathcal{M} \cap L^\infty$. Then
  the functional $\lambda_k^{-1} \colon \mathcal{M} \cap L^\infty \to \R$ has a locally
  Lipschitz extension on a neighborhood of
  $\mu$ in $L^\infty$.
  Its Clarke subdifferential can be described by the inclusion
  \begin{equation}
    \partial_C \lambda_k^{-1}(\mu) \subset co\
    \set*{\phi^2 }{
      \phi \in E_k(\mu),\ \int \abs{d\phi}^2 dv_g = 1
    },
  \end{equation}
  where $\text{co}\,K$ denotes the convex hull of the set $K$.
\end{corollary}
\begin{proof}
  Since $\mu(M) \neq 0$, Lemma~\ref{lem:sobolev-dual-meas}
  gives a continuous extension of $\lambda_k^{-1}$. By the same lemma,
  we also know the Fréchet derivative of $\mathfrak{q}(\mu)$. Then one applies
  Propositions~\ref{prop:clarke-compose} and~\ref{prop:subdiff-calc}
  to deduce that
  \begin{equation}
    \partial_C \lambda_k^{-1}(\mu) \subset co\
    \set*{(\tilde{\phi}-[\tilde{\phi}]_\mu)^2 }{
      \tilde{\phi} \in E_k(T(\mu)) \subset X,\ \int \abs{d\tilde{\phi}}^2 dv_g = 1
    }.
  \end{equation}
  It remains to substitute $\phi = \tilde{\phi}-[\tilde{\phi}]_\mu \in E_k(\mu)$.
\end{proof}

\subsection{Approximate subdifferential}
The description of the Clarke subdifferential in Proposition~\ref{prop:subdiff-calc}
has the drawback that it does not explicitly reflect the position $k$
in the eigenvalue sequence when
$\lambda_k = \lambda$. To obtain a more refined characterization of
the first variation of $\lambda_k$, one can instead consider the \emph{approximate
subdifferential}.

Let $E$ be a locally convex topological vector space, $f\colon E \to [-\infty, \infty]$ be
a function on $E$, and $S \subset E$. We define
\begin{gather}
  \dom f = \set*{x \in E}{ \abs{f(x)} < \infty}, \\
  f_S(x) = \begin{dcases}
    f(x),                           & \quad x \in S      \\
    \infty,                           & \quad x \not \in S
  \end{dcases}.
\end{gather}
For $x\in \dom f$, the Dini-Hadamard lower directional derivative is defined as
\begin{equation}
  d^-_x f(v) = \liminf_{\tilde{v} \to v, \tau\downarrow 0}
  \frac{1}{\tau} \brs{f(x + \tau \tilde{v}) - f(x)},
\end{equation}
and Dini-Hadamard subdifferential is
\begin{equation}
  \partial^-f(x) = \set*{\xi \in E^*}{\brt{\xi, v}
  \leq d^-_xf(v) \ \forall v \in E}.
\end{equation}
Set $\partial^-f(x) = \varnothing$ if $x \not \in \dom f$.
\begin{definition}
  Let $\mathcal{F}$ denote the collection of all finite-dimensional
  subspaces of $E$. The \emph{approximate subdifferential}
  (also called the \emph{A-subdifferential}) of $f$ at $x$ is defined as
  \begin{equation}
    \partial_A f(x) = \bigcap_{L\in\mathcal{F}}
    \limsup_{\tilde{x}\to_f x} \partial^-f_{\tilde{x}+L}(\tilde{x}),
  \end{equation}
  where $\limsup_{\tilde{x}\to_f x} S_{\tilde{x}}$ denotes Painlevé-Kuratowski
  upper limit, i.e. the set of cluster points, of the directed
  sets $\cup_{\tilde{x}\in U(f,x,\delta)} S_{\tilde{x}}$, with
  $U(f,x,\delta):= \set*{\tilde{x} \in x + U}{\abs{f(\tilde{x})-f(x)}<\delta}$,
  in the weak$^*$ topology of $E^*$, so
  \begin{equation}
    \limsup_{\tilde{x}\to_f x} S_{\tilde{x}} =
    \bigcap_{\substack{\delta > 0 \\ U \text{is 0-nbh}}}
    \overline{\bigcup_{\tilde{x}\in U(f,x,\delta)} S_{\tilde{x}}}^*.
  \end{equation}
\end{definition}
For a systematic treatment, see, e.g., \cite{Ioffe:1984:approx-subdiff-fin-dim, Ioffe:1986:approx-subdiff-arbitrary}.
Unlike the Clarke subdifferential $\partial_C$, the approximate subdifferential $\partial_A$
is generally not convex. However, for Lipschitz functions $f$, one has the relation
(see \cite[Proposition~3.3]{Ioffe:1986:approx-subdiff-arbitrary})
\begin{equation}
  \partial_Cf(x) = \overline{co\ \partial_Af(x)}^*.
\end{equation}

If $x \in \dom f$ is a local minimizer of $f$, then
$0 \in \partial^-f_{x + L}(x)$ for any $L\in\mathcal{F}$, and hence
\begin{equation}
  0 \in \partial_A f(x).
\end{equation}

There are many other refinements of the Clarke subdifferential.
The reason we chose the approximate subdifferential is that it has
a very natural chain rule. For a more general statement, see
\cite[Theorem~4.3]{Ioffe:1986:approx-subdiff-arbitrary}.
\begin{proposition}[A-chain rule]
  Under the conditions of the Proposition~\ref{prop:clarke-compose},
  one has
  \begin{equation}
    \partial_{A} (g^*f)(x) \subset (d_x g)^*[\partial_{A} f(g(x))].
  \end{equation}
\end{proposition}
For an eigenvalue $\lambda_k(T)$ of the operator $T$, let us denote by
\begin{equation}
  k_{max} = \max \set{k' \in \mathbb{N}}{\lambda_{k'} = \lambda_k}.
\end{equation}
A refined first variation formula looks as follows.
\begin{proposition}\label{prop:approx-subdiff-calc}
  Let $\lambda_k(T)$ as in Proposition~\ref{prop:subdiff-calc}.
  Then the $A$-subdifferential $\partial_A \lambda_k(T) \subset \mathcal{N}[H]$
  is given by
  \begin{equation}\label{abst-approx-subdif}
    \partial_A \lambda_k(T) =
    \set*{\mathfrak{t} \in \partial_C \lambda_k(T)}{
      \rk \mathfrak{t} \leq k_{max} - k + 1
    }.
  \end{equation}
\end{proposition}
\begin{proof}
  The case of a finite-dimensional $H$ was proved
  in~\cite[Corollary 10]{Lewis:1999:eigenval-subdiff}. Note
  that \cite{Lewis:1999:eigenval-subdiff} uses a slightly
  different definition
  of an approximate subdifferential, but both of them coincide
  for Lipschitz-functions in a finite dimensional vector space,
  see~\cite[Theorem~9.2]{Mordukhovich-Shao:1996:subdiff-comparison}.

  Since we have exactly the same chain rule for the $A$-subdifferential,
  the rest of the proof is identical to Proposition~\ref{prop:subdiff-calc}.
\end{proof}
\begin{corollary}\label{cor:sph-dim-restric}
  Let $\mu \in L^\infty(M)$, $\mu \geq c > 0$,
  such that $\overline{\lambda}_k(\mu) = \mathcal{V}_k(g)$.
  Then there exists a number $n \leq k_{max} - k$ and a map
  \begin{equation}
    \Phi \colon M \to \Sph^n
  \end{equation}
  given by $k$-th eigenfunctions of $\mu$, i.e.
  $\Delta \phi_i = \lambda_k \phi_i \mu$.
\end{corollary}
\begin{proof}
  Let us rescale $\mu(M) = 1$. Then
  Lemma~\ref{lem:sobolev-dual-meas} tells us that
  $T(\tilde{\mu}) \in \mathcal{K}_{sa}[X]$ and $\lambda_k(T(\tilde{\mu})) = \lambda_k^{-1}(\tilde{\mu})$
  for $\tilde{\mu} \in L^\infty$ close enough to $\mu$.
  Since $\mu$ is a maximizer of $\overline{\lambda}_k$,
  we see that $0 \in \partial_A \bar{\lambda}_k^{-1}(\mu)$.
  To compute the $A$-subdifferential, we apply A-chain rule,
  to the composition $\tilde{\mu} \mapsto \frac{\tilde{\mu}}{\tilde{\mu}(M)} \mapsto \lambda_k^{-1}(\frac{\tilde{\mu}}{\tilde{\mu}(M)}) = \bar{\lambda}_k^{-1}(\tilde{\mu})$.
  Then
  \begin{equation}
    d_\mu \brr{\frac{\tilde{\mu}}{\tilde{\mu}(M)}}(\nu) = \nu - \nu(M)\mu,
  \end{equation}
  and Proposition~\ref{prop:approx-subdiff-calc} together with
  Corollary~\ref{cor:subdiff-calc-meas} yield
  \begin{equation}
    0 \in \partial_A \bar{\lambda}_k^{-1}(\mu) \subset
    \brc{-\lambda_k^{-1}} + \set*{\sum_{i = 0}^{k_{max} - k} \phi^2_i }{
      \phi_i \in E_k(\mu),\ \sum_{i = 0}^{k_{max} - k}\int \abs{d\phi_i}^2 dv_g = 1
    }
  \end{equation}
  because $\rk \mathfrak{t} \leq k_{max} - k + 1$ means that
  $\mathfrak{t}$ admits a decomposition
  $\mathfrak{t} = \sum_{i = 0}^{k_{max} - k} \phi_i \otimes \phi_i$.
  By rescaling $\phi_i := \sqrt{\lambda_k}\phi_i$, one constructs
  a map into the unit sphere.
\end{proof}
\begin{remark}
  When $m \geq 3$, the assumptions $\mu \in L^\infty$, $\mu \geq c > 0$ can be relaxed to
  $\mu \in \mathcal{M} \cap \brr{H^{1, \frac{m}{m-1}}}^*$. In particular, $\mu$ is allowed
  to vanish on subsets of $M$. However, the argument in this setting becomes more involved.
  One must employ $A$-subdifferential calculus together with $A$-normal
  cones (see \cite{Ioffe:1986:approx-subdiff-arbitrary}), and
  apply the corresponding multiplier rule, analogous to Proposition~\ref{prop:multip-rule}.
\end{remark}

\section{Optimization of eigenvalues on \texorpdfstring{$\Sph^{m}$}{Lg}}
\begin{lemma}\label{lem:top-grassmannian-degree}
  Let $\gamma$ be the tautological bundle (of rank $k$) over $\Gr_k(\R^m)$, and
  $w_{k}(\gamma)$ be its $k$-th Stiefel-Whitney class. Then
  \begin{equation}
  w_{k}(\gamma)^{m-k} = 1
  \in \mathbb{Z}_2 \approx H^{top}(\Gr_{k}(\R^{m}); \mathbb{Z}_2).
  \end{equation}
\end{lemma}
\begin{proof}
  Consider the manifold of complete flags in $\R^m$ denoted by
  $\Fl(\R^m)$. There is a continuous bundle projection $p\colon\Fl(\R^m)
  \to \Gr_k(\R^m)$ by taking the initial $k$-th flag. Then
  see \cite{Stong:1982:products-in-grassmannians} or
  \cite[Facts~(a)--(d)]{Korbas-Novotny:2009:sw-classes-grassmann}:
  \begin{enumerate}
    \item \begin{equation}
      H^*(\Fl(\R^m);\mathbb{Z}_2) \approx
      \quot{\mathbb{Z}_2[e_1,\cdots, e_m]}{\brr{\prod_{i=1}^{m}(1+e_i) = 1}}.
    \end{equation}
    \item $p^*\colon H^*(\Gr_k(\R^m);\mathbb{Z}_2)\to H^*(\Fl(\R^m);\mathbb{Z}_2)$
    is injective and
    \begin{equation}
      p^*[w(\gamma)] = \prod_{i=1}^{k}(1+e_i).
    \end{equation}
    \item The value of $u \in \mathbb{Z}_2 \approx H^{top}(\Gr_k(\R^m);\mathbb{Z}_2)$
    is the same as the value of $p^*(u)\cdot (\prod_{i=1}^k e_{i}^{k-i})
    (\prod_{i=k+1}^m e_{i}^{m-i}) \in \mathbb{Z}_2 \approx H^{top}(\Fl(\R^m);\mathbb{Z}_2)$.
    \item The nonzero monomials in $H^{top}(\Fl(\R^m);\mathbb{Z}_2)$ are precisely
    those of the form
    \begin{equation}
      e^{m-1}_{\sigma(1)}\cdots e^{m-i}_{\sigma(i)}\cdots e^{0}_{\sigma(m)},
    \end{equation}
    i.e. those with no repeated exponents.
  \end{enumerate}
  Therefore, $p^*[w_{k}(\gamma)^{m-k}] = \prod_{i=1}^k e_i^{m-k}$, and
  it remains to compute
  \begin{equation}
    p^*[w_{k}(\gamma)^{m-k}] \cdot \brr{\prod_{i=1}^k e_{i}^{k-i}}
    \brr{\prod_{i=k+1}^m e_{i}^{m-i}}
    = \prod_{i=1}^m e_{i}^{m-i} \neq 0.
  \end{equation}
\end{proof}

\subsection{Proof of Theorem~\ref{thm:2d-eigen-sphere-upper-bound}}
  We suppose that $\lambda_{k+2}(\mu) \neq 0$ and rescale $\mu$ so that
  $\lambda_{k+2}(\mu) = 1$. We will use the variational
  characterization from Proposition~\ref{prop:eigen-upper-bound}.
  Then there exists $\brc{\phi_i}_{i=1}^{k+1} \subset H^1$ such that
  for every $\psi \in H^1$, satisfying
  \begin{equation}\label{eq:test-func-cond}
    \int \psi d\mu = \int \psi \phi_1 d\mu = \cdots
    = \int \psi \phi_{k+1} d\mu = 0,
  \end{equation}
  one has
  \begin{equation}
    \int \psi^2 d\mu \leq \int \abs{d\psi}^2
  \end{equation}
  with the equality if and only if $\psi$ is an eigenfunction corresponding
  to $\lambda_{k+2}(\mu) = 1$.

  We are going to construct such functions $\psi$ as the coordinate
  functions of a sphere-valued map.
  Consider a family of conformal mappings $T_p\colon \Sph^m \to \Sph^m$,
  \begin{equation}
    T_p(x) = \frac{1 - \abs{p}^2}{\abs{x+p}^2}\brr{x+p} +p,
  \end{equation}
  parametrized by points in the (open) unit ball $p \in \mathbb{B}^{m+1}$. In fact,
  $T_p$ is simply a dilation in the stereographic projection coordinates on the
  plane orthogonal to $p$.
  Consider also a $(m-k)$-dimensional plane $\pi \in \Gr_{m-k}(\R^{m+1})$ and define
  $\SHM{\pi}$ to be the map $\SHM{k}\colon
  \Sph^m \to \pi \cap \Sph^m \approx \mathbb{S}^{m-1-k}$. We would like
  to choose the parameters $\pi$ and $p$ so that the coordinate functions of the map
  \begin{equation}
    \Psi_{\pi, p} := T_p \circ \SHM{\pi} \colon \Sph^m \to \pi \cap \Sph^m
  \end{equation}
  meet the orthogonality conditions~\eqref{eq:test-func-cond}.

  Since $\mu \in \mathcal{B}[H^1]$, the measure $\mu$ vanishes on sets of zero capacity. Hence, one can think of
  the push-forwarded measure
  $(\SHM{\pi})_* \mu$ as a continuous measure on $\Sph^m$ supported on $\pi \cap \Sph^m$.
  The mapping
  \begin{equation}
     \Gr_{m-k}(\R^{m+1}) \ni \pi \mapsto (\SHM{\pi})_* \mu \in \mathcal{M}_c(\Sph^m)
  \end{equation}
  is clearly weakly$^*$ continuous. Thus, \cite[Corollary~5]{Laugesen:2021:center-mass}
  yields a unique point $p(\pi)$ that depends continuously on $\pi$
  such that the map $\Psi_{\pi, p(\pi)}$ is orthogonal to constant functions. Since
  \begin{equation}
    \supp [(\SHM{\pi})_* \mu] \subset \pi \cap \Sph^m \approx \Sph^{m-1-k},
  \end{equation}
  one
  could argue similarly replacing $T_p$ by its $(m-1-k)$-dimensional version. The uniqueness
  of $p(\pi)$ then implies $p(\pi) \in \pi \cap \mathbb{B}^{m+1}$,
  so $T_{p(\pi)}$ preserves $\pi \cap \Sph^m$, and therefore
  $im\,\Psi_{\pi, p(\pi)} \subset \pi$.

  To make $\Psi_{\pi, p(\pi)}$ orthogonal to the functions $\brc{\phi_i}$,
  we recall the definition of the tautological bundle $\gamma$,
  \begin{equation}
    \gamma = \set*{(\pi, v) \in \Gr_{m-k}(\R^{m+1})\times \R^{m+1}}{v\in\pi},
  \end{equation}
  and define a continuous section $s\colon\Gr_{m-k}(\R^{m+1}) \to \oplus_{i=1}^{k+1}\gamma =: (k+1)\gamma$  by the formula
  \begin{equation}
    s(\pi) = \bigoplus_{i=1}^{k+1} \int \Psi_{\pi, p(\pi)}(x) \phi_i(x) d\mu(x).
  \end{equation}
  Note that the continuity of $s$ is immediate once we know that $p(\pi)$ is continuous.
  We then claim that $(k+1)\gamma$ does not admit a nonvanishing
  continuous section. The presence of a nonvanishing
  section would imply that the
  top Stiefel-Whitney class
  of $(k+1)\gamma$ is zero, i.e.
  $w_{(m-k)(k+1)}[(k+1)\gamma] = 0$  since
  \begin{equation}
    \operatorname{rk} (k+1)\gamma = (m-k)(k+1) = \dim \Gr_{m-k}(\R^{m+1}).
  \end{equation}
  On the other hand,
  \begin{equation}
    w[(k+1)\gamma] = [1 + w_1(\gamma)+ \cdots + w_{m-k}(\gamma)]^{k+1}
    \in H^{*}(\Gr_{m-k}(\R^{m+1}); \mathbb{Z}_2),
  \end{equation}
  so
  \begin{equation}
    H^{top}(\Gr_{m-k}(\R^{m+1}); \mathbb{Z}_2) \ni w_{(m-k)(k+1)}[(k+1)\gamma] = w_{m-k}(\gamma)^{k+1} \neq 0
  \end{equation}
  by Lemma~\ref{lem:top-grassmannian-degree}, and the section $s$ must vanish at some point,
  say $\pi_0$.

  Thus, we have found parameters $\pi_0$ and $p_0 = p(\pi_0)$ such that
  the coordinate functions of $\Psi = \Psi_{\pi_0, p_0}$ satisfy~\eqref{eq:test-func-cond}.
  Using the coordinates~\eqref{sec:calc-on-sphere} and the calculation~\eqref{eq:equator-energy},
  we obtain
  \begin{equation}\label{ineq:l2-test-func}
      \overline{\lambda}_{k+2}(\mu) = 1 \cdot \mu(\Sph^m) =
      \sum_i \int (\Psi^i)^2 d\mu \leq \int \abs{d\Psi}^2
      = \frac{E[\SHM{k}]}{n \sigma_n}\int_{\Sph^{n}} \abs{dT_{p_0}}^2,
  \end{equation}
  where $\Sph^n := \pi_0 \cap \Sph^m$.
  Since $T_{p_0}|_{\Sph^n}$ is conformal,
  the Hölder inequality implies
  \begin{equation}
    \frac{1}{n \sigma_n}\int_{\Sph^{n}} \abs{dT_{p_0}}^2 \leq
    \frac{1}{n} \brr{\frac{1}{\sigma_{n}}\int_{\Sph^{n}} \abs{dT_{p_0}}^{n}}^{\frac{2}{n}}
    =\frac{1}{n} \brr{\frac{1}{\sigma_{n}}\int_{\Sph^{n}} \abs{d(\id_{\Sph^n})}^{n}}^{\frac{2}{n}}
    = 1
  \end{equation}
  Thus, we have obtained the desired upper bound.

  The equality
  in~\eqref{ineq:l2-test-func} implies that $\Psi^i$ are eigenfunctions, i.e.
  $\Delta \Psi = \mu \Psi$, and we obtain $\mu = \abs{d\Psi}^2_{g_{\Sph^m}} dv_{g_{\Sph^m}}$
  (cf. also the argument after~\eqref{eigen.map}).

  It remains to identify the equality cases in the Hölder inequality. If $n>2$,
  equality implies that $\abs{dT_{p_0}|_{\Sph^{n}}}^2$ is constant.
  Since $T_{p_0}$ is a dilation in the appropriate stereographic projection
  coordinates, it happens only if $p_0 = 0$ and $T_0 = \id_{\Sph^n}$. So,
  $\Psi = \SHM{\pi_0} =: \SHM{k}$.

  If $n=2$, the Hölder inequality used above is an identity, reflecting the conformal invariance of the Dirichlet energy in dimension two.
  In this case, $\Psi = T_p \circ \SHM{k}$ is harmonic for any $p \in \mathbb{B}^3$. Therefore,
  when $n=2$, the equality cases are precisely the generalized equator maps up to postcomposition by conformal automorphisms of $\Sph^2$.

  Finally, suppose that $k > m -  7$.
  If $\mathcal{V}_{k+2}(g_{\Sph^m}) = E[\SHM{k}]$, we have seen that all the maximizers are singular and of the form
  $\mu = \abs{d(T_p \circ\SHM{k})}^2_{g_{\Sph^m}} dv_{g_{\Sph^m}}$.
  On the one hand, the singularities of $T_p \circ\SHM{k}$
  have Hausdorff dimension $k$, which is a contradiction with
  Theorem~\ref{thm:main}.

\subsection{Proof of Theorem~\ref{thm:2d-eigen-sphere-index}}
  Thanks to the upper bounds from Theorem~\ref{thm:2d-eigen-sphere-upper-bound},
  we only need to calculate the index of $\SHM{k}$.
  Recall that $\ind \SHM{k} = \ind (\mathfrak{q}_{\SHM{k}}|_{C^\infty(\Sph^m)})$, where
  $\mathfrak{q}_{\SHM{k}}[\phi] = \int \abs{d\phi}^2 - \int \phi^2 \abs{d\SHM{k}}^2$
  and
  \begin{equation}
    \ind (\mathfrak{q}_{\SHM{k}}|_{\mathcal{D}}) =
    \sup \set*{\dim V}{V \subset \mathcal{D},\ \mathfrak{q}_{\SHM{k}}[\phi] < 0
    \ \forall \phi \in V\setminus \brc{0}}.
  \end{equation}

  Note that the singular set of $\SHM{k}$ is $\brc{0}\times\Sph^k$, where $k \leq m-2$,
  and hence has capacity zero. Let $\phi \in C^\infty$ vanish on $\brc{0}\times\Sph^k$.
  Since $\abs{d\SHM{k}}^2 dv_{g_{\Sph^m}} \sim (1-t)^{\frac{n+1}{2}-2}
  dt dv_{\theta} dv_{\omega}$ near the singularities, the Hardy inequality
  implies that
  \begin{equation}
    \int_{t \geq 1 -\epsilon} \phi^2 \abs{d\SHM{k}}^2 dv_{g_{\Sph^m}}
    \leq C\int_{t \geq 1 -\epsilon} \abs{\partial_t \phi}^2(1-t)^{\frac{n+1}{2}}dt dv_{\theta} dv_{\omega}
    \leq C\int \abs{d\phi}^2dv_{g_{\Sph^m}},
  \end{equation}
  i.e.
  $\mathfrak{q}_{\SHM{k}}$ is continuous on $H^1(\Sph^m)$. Hence,
  \begin{equation}\label{eq:dif-index-domain}
    \ind (\mathfrak{q}_{\SHM{k}}|_{\mathcal{D}_1})
    =\ind (\mathfrak{q}_{\SHM{k}}|_{\mathcal{D}_2})
    \quad\text{if}\quad
    \overline{\mathcal{D}_1}^{H^1} = \overline{\mathcal{D}_2}^{H^1}.
  \end{equation}
  In particular, one can take $\mathcal{D}_1 = C^\infty(\Sph^m)$ and
  $\mathcal{D}_2 = C^\infty_0(\Sph^n\times (0,1)\times \Sph^k)$, where
  \begin{equation}
    \Sph^n\times (0,1)\times \Sph^k \approx \Sph^m\setminus (\brc{0}\times\Sph^k \cup
    \Sph^n\times\brc{0}),
  \end{equation}
  provided that $n = m - 1 - k \leq m-2$, i.e. $k \geq 1$.

  \paragraph{Case $\mathbf{k\geq 1}$.} Let $I = (-1,1)$. We adjust the coordinates~\eqref{sec:calc-on-sphere}
  to be more symmetrical: $t \in I$ and
  \begin{equation}\label{sym-coord}
    (\theta, t, \omega) \mapsto \brr{\theta\sqrt{\tfrac{1-t}{2}}, \omega\sqrt{\tfrac{1+t}{2}}}.
  \end{equation}
  By using spherical harmonics
  $\Delta_{\theta}X_j = \rho_{j} X_j$ and
  $\Delta_{\omega}Y_i = \nu_{i} Y_i$
  on $\Sph^{n}$ and $\Sph^k$, respectively, one easily
  derives that the following embedding is dense in
  $H^1$-norm (or even in $C^\infty_0$),
  \begin{equation}
    \mathcal{D} :=
    Span \brc{C^\infty_0(I)X_j Y_i}_{i,j \geq 0}
    \hookrightarrow C^\infty_0(\Sph^{n}\times I\times\Sph^{k}).
  \end{equation}
  Thus, $\frac{1}{4}\mathfrak{q}_{\SHM{k}}|_{\mathcal{D}} = \bigoplus \mathfrak{q}_{i,j}$,
  where $\brt{t}_{x}^{y} := (1-t)^{x}(1+t)^{y}$ and
  \begin{equation}
    \mathfrak{q}_{i,j}[\phi] = \int_{-1}^{1} \dot{\phi}^2
    \brt{t}_{\frac{n+1}{2}}^{\frac{k+1}{2}} dt
    + \frac{1}{2}\int_{-1}^{1} \phi^2 \brr{\frac{\nu_{i}}{1+t} + \frac{\rho_{j} - n}{1-t}}
    \brt{t}_{\frac{n-1}{2}}^{\frac{k-1}{2}} dt,
  \end{equation}
  and by linear algebra $\ind (\mathfrak{q}_{\SHM{k}}|_{\mathcal{D}}) =
  \sum_{i,j} \ind (\mathfrak{q}_{i,j}|_{C^\infty_0(I)}) = \sum_{i}\ind (\mathfrak{q}_{i,0}|_{C^\infty_0(I)})$
  since $\rho_{j} \geq n$ and $\mathfrak{q}_{i,j}$ are nonnegative definite when $j>0$.

  By replacing $\phi = \psi\brt{t}_{-\frac{\alpha}{2}}^{\frac{\beta}{2}}$, where $\phi, \psi \in C^\infty_0(I)$,
  and integrating by parts, we see that
  \begin{equation}
    \mathfrak{q}_{i,0}[\phi] =: \mathfrak{q}_{i}[\psi] =
   \int_{-1}^{1} \dot{\psi}^2 \brt{t}_{\frac{n+1}{2}-\alpha}^{\frac{k+1}{2}+\beta} dt
    +L \int_{-1}^{1} \psi^2 \brt{t}_{\frac{n-1}{2}-\alpha}^{\frac{k-1}{2}+\beta} dt
  \end{equation}
  provided $\alpha,\beta$ satisfy the quadratic equations
  \begin{equation}\label{eq:jacobi-root}
    \begin{cases*}
      \alpha^2 - (n-1)\alpha + n = 0,
      \\ \beta^2 + (k-1)\beta - \nu_i = 0,
    \end{cases*}
  \end{equation}
  and $L = \frac{1}{4}\brs{(\nu_i-n) + \beta (n+1)- \alpha(k+1) - 2\alpha\beta}$.

  Note that the first equation is solvable exactly iff $n \geq 6 \Leftrightarrow
  k \leq m - 7$, and $\beta = \ell$ or $\beta = -(k-1+\ell)$ since
  \begin{equation}
    \nu_i \in \set{\ell(k-1+\ell)}{\ell = 0, 1,\cdots}.
  \end{equation}

  The quadratic form $\mathfrak{q}_{i}$ is
  the form generated by a shifted version of the Jacobi differential equation. The original equation looks as follows
  \begin{equation}\label{eq:Jaconi-diff}
    -\brr{\brt{t}_{a+1}^{b+1} u'}' = \lambda \brt{t}_{a}^{b} u,
  \end{equation}
  where $a = \frac{n-1}{2} - \alpha_{\pm} = \mp\frac{1}{2}\sqrt{n^2 - 6n + 1}$ and
  $b = \frac{k-1}{2} + \beta = \pm (\frac{k-1}{2} + \ell)$. Thus,
  \begin{equation}
    \ind (\mathfrak{q}_{i}|_{C^\infty_0(I)}) = \#\set*{\lambda \in Spec\,(S_F)}{\lambda + L < 0},
  \end{equation}
  where $S_F$ is the Friedrichs extension of the minimal operator
  $S_{min}$ associated with~\eqref{eq:Jaconi-diff}, see, e.g.,
  \cite[Sections~10.3,~10.5]{Zettl:2005:sturm-liouville}. If we choose
  $\beta = \ell$ and $\alpha = \alpha_{-}$ to be
  the least root of~\eqref{eq:jacobi-root},
  we will have $a,b > -1$, in which case the eigenfunctions of $S_F$ are
  the Jacobi polynomials, see, e.g., \cite[Section~14.16]{Zettl:2005:sturm-liouville} or \cite[Section~23]{Everitt:2005:sl-diff-eq}.
  The Jacobi polynomials form a complete orthogonal system in $L^2((-1,1), \brt{t}_{a}^{b} dt)$,
  so
  \begin{equation}
    Spec\,(S_F) = \set*{s(s + a + b + 1)}{s = 0,1,2,\cdots}
  \end{equation}
  is simple and discrete. Therefore,
  $L = \frac{1}{4}\brs{\ell(m-1+\ell) - n-\alpha\brr{k+1+2\ell}}$
  and
  \begin{equation}
    \begin{split}
    \ind (\mathfrak{q}_{i}|_{C^\infty_0(I)}) &= \#
    \set*{s \in \mathbb{N}}{ s^2 + (\tfrac{m-1}{2}-\alpha +\ell)s + L < 0}
    \\                    &= \#\set*{s \in \mathbb{N}}{ s < \tfrac{\alpha - \ell}{2}},
    \end{split}
  \end{equation}
  since $\tfrac{\alpha - \ell}{2}$ is the greatest root of the quadratic equation
  on  $s$. Note that $\alpha = \alpha_{-} \in (1, 2]$, so
  $\ind (\mathfrak{q}_{i}|_{C^\infty_0(I)}) = 1$ if $\ell \in \brc{0, 1}$
  and zero otherwise. At the same time, the eigenvalue $\nu_i = 0$
  ($\ell = 0$) has multiplicity $1$, and $\nu_i = k$ ($\ell = 1$)
   has multiplicity $k+1$, so
   \begin{equation}
    \ind \SHM{k} = k+2.
   \end{equation}

  \paragraph{Case $\mathbf{k = 0}$.}

  Now, one should use the coordinates $\Sph^{m-1}\times (-1,1)
  \to \Sph^m$, $(\theta, t) \mapsto (\theta \sqrt{1-t^2}, t)$ instead
  of~\eqref{sym-coord}. Arguing as in the previous case, in place of
  $\mathfrak{q}_{i,0}$, we have
  \begin{equation}
    \mathfrak{q}_{0}[\phi] = \int_{-1}^{1} \dot{\phi}^2 \brt{t}_{\frac{m}{2}}^{\frac{m}{2}} dt
    - (m - 1) \int_{-1}^{1} \phi^2 \brt{t}_{\frac{m}{2}-2}^{\frac{m}{2}-2} dt.
  \end{equation}
  By replacing $\phi = \psi\brt{t}_{-\frac{\alpha}{2}}^{-\frac{\alpha}{2}}$,
  we see that
  \begin{equation}
    \mathfrak{q}_{0}[\phi] =: \mathfrak{q}[\psi] =
   \int_{-1}^{1} \dot{\psi}^2 \brt{t}_{\frac{m}{2}-\alpha}^{\frac{m}{2}-\alpha} dt
    - \brr{\alpha + m - 1} \int_{-1}^{1} \psi^2 \brt{t}_{\frac{m}{2}-1-\alpha}^{\frac{m}{2}-1-\alpha} dt
  \end{equation}
  provided $\alpha$ satisfies the same equation~\eqref{eq:jacobi-root}, with $n = m-1$.
  Again, the quadratic form $\mathfrak{q}$ is
  generated by a shifted Jacobi differential equation, with
  $a = b = \frac{m}{2}-1 -\alpha_{\pm}$. Thus,
  \begin{equation}
    \ind (\mathfrak{q}|_{C^\infty_0(I)}) = \#\set*{\lambda \in Spec\,(S_F)}{\lambda < \alpha + m - 1}.
  \end{equation}
  Choosing $\alpha = \alpha_{-}$ to be the least root of~\eqref{eq:jacobi-root} gives
  \begin{equation}
    \begin{split}
    \ind (\mathfrak{q}|_{C^\infty_0(I)}) &= \#\set*{s \in \mathbb{N}}{ s^2 + (m - 1 - 2\alpha)s - \brr{\alpha + m - 1} < 0}
    \\                    &= \#\set*{s \in \mathbb{N}}{ s < \alpha \in (1, 2]}
    \\                    &= 2.
    \end{split}
  \end{equation}

\section{Proof of the existence and regularity}\label{sec:exis-and-reg}

\subsection{Constructing a maximizing sequence}
Denote by $\mathcal{P} = \set{\mu \in \mathcal{M}}{\mu(M) = 1}$
the space of probability measures on $M$.
\begin{proposition}\label{prop:max-sequence}
  Let $(M,g)$ be a closed connected Riemannian manifold and
  $\lambda := \mathcal{V}_k(g) =  \sup_{\mu\in L^\infty \cap \mathcal{P}} \lambda_k(\mu)$. Then there
  exist sequences $\mu_\epsilon \in L^\infty \cap \mathcal{P}$, with eigenvalues
  $\lambda_\epsilon:= \lambda_k(\mu_\epsilon) \to \lambda$
  as $\epsilon \to 0$, and a
  sequence of tensors $\mathfrak{t}_\epsilon\in
  H^1\otimes H^1$
  such that the following conditions are satisfied:
  \begin{enumerate}
    \item $\mathfrak{t}_\epsilon = \sum \phi^i_\epsilon \otimes \phi^i_\epsilon$ and
    $\Delta \phi^i_\epsilon = \lambda_\epsilon \phi^i_\epsilon
    \mu_\epsilon$ weakly;
    \item $0 \leq \m{\mathfrak{t}_\epsilon}\leq 1$ and
    $\norm{\mathfrak{t}_\epsilon}_{\tensor{H^1}}$ are uniformly bounded;
    \item\label{it:almost-sphere} $\epsilon \norm{\mu_\epsilon}_{L^\infty} \leq 1 $ and
    $v_g(\set{x \in M}{\m{\mathfrak{t}_\epsilon}(x) < 1})\leq \epsilon$.
  \end{enumerate}
\end{proposition}
\begin{proof}
    Consider the ball
  $B_{1/\epsilon}=B_{1/\epsilon}(0) \subset L^\infty$ of
  radius $1/\epsilon$ centred at zero.
  If
  \begin{equation}
    \lambda_\epsilon := \sup \set{\lambda_k(\mu)}
    {\mu \in B_{1/\epsilon}\cap \mathcal{P}},
  \end{equation}
  it is clear that $\lambda_\epsilon \to \lambda$ as $\epsilon \to 0$.
  At the same time, $B_{1/\epsilon}\cap \mathcal{P}$ is a bounded
  weakly$^*$ closed subset of $L^\infty$, hence weakly$^*$ compact.
  Proposition~\ref{prop:upper-cont-and-bound} implies that
  the functional $\mu \to \lambda_k(\mu)$ is bounded on $L^\infty \cap\mathcal{P}$
  and upper semi-continuous under the weak$^*$ convergence.
  So, one easily finds an element $\mu_\epsilon \in B_{1/\epsilon}\cap \mathcal{P}$ achieving
  $\lambda_\epsilon = \lambda_k(\mu_\epsilon)$.

  The next step is to apply Corollary~\ref{cor:subdiff-calc-meas} together
  with Proposition~\ref{prop:multip-rule}
  to obtain $\mathfrak{t}_\epsilon \in \partial_C \lambda_k^{-1}(\mu_\epsilon)$
  such that
  \begin{equation}
    \mathfrak{t}_\epsilon := \sum_i \phi^i_\epsilon \otimes \phi^i_\epsilon,
  \end{equation}
  \begin{equation}\label{ineq:extremal-condition}
    \int \m{\mathfrak{t}_\epsilon} d\mu_\epsilon
    \leq \int \m{\mathfrak{t}_\epsilon} d\mu
  \end{equation}
  for any $\mu \in B_{1/\epsilon}\cap \mathcal{P}$,
  and
  \begin{equation}\label{ineq:sup-norm-lower-bound}
    1 = \int\md{\mathfrak{t}_\epsilon} dv_g =
  \lambda_\epsilon \int \m{\mathfrak{t}_\epsilon} d\mu_\epsilon
  \leq \lambda_\epsilon \norm{\m{\mathfrak{t}_\epsilon}}_{L^\infty}.
  \end{equation}

  Since $\mu_\epsilon = \rho_\epsilon \dv{g}$,
  $\mu = \rho \dv{g}$ for some $\rho_\epsilon, \rho \in L^\infty$,
  the elliptic regularity implies $\phi_\epsilon^i \in H^{2,p}(M)$
  for any $p < \infty$, so all $\brc{\phi_\epsilon^i}$ and
  $\m{\mathfrak{t}_\epsilon}$ are continuous.
  Then we claim that
  inequality~\eqref{ineq:extremal-condition} implies
  \begin{equation}\label{eq:const-density}
    \rho_\epsilon \equiv 1/\epsilon
    \text{ a.e. on }
    \set{x\in M}{\m{\mathfrak{t}_\epsilon}(x) < S},
  \end{equation}
  where $S:= \norm{\m{\mathfrak{t}_\epsilon}}_{L^\infty} \geq \lambda^{-1}$
  by~\eqref{ineq:sup-norm-lower-bound}. If we did so,
  one could check that the renormalized tensors
  $\tilde{\mathfrak{t}}_\epsilon := \mathfrak{t}_\epsilon/S$
  satisfy the conditions of the proposition.

  If~\eqref{eq:const-density} does not hold, then
  for sufficiently small $\delta > 0$, one has
  $(1/\epsilon-\rho_\epsilon)\chi_A \not\equiv 0$, where
  $\chi_A$ is the characteristic function of the set
  \begin{equation}
    A:= \set{x\in M}{\m{\mathfrak{t}_\epsilon}(x) < S - 2\delta}.
  \end{equation}
  Since $\m{\mathfrak{t}_\epsilon}$ is continuous, the set
  \begin{equation}
    C := \set{x \in M}{S - \delta < \m{\mathfrak{t}_\epsilon}(x)}
  \end{equation}
  is open (and nonempty by definition), and one also has $\rho_\epsilon \chi_C \not \equiv 0$,
  or else we would apply the maximum
  principle to the subharmonic function
  $\m{\mathfrak{t}_\epsilon}$ on $C$ and conclude that $\m{\mathfrak{t}_\epsilon} \equiv S$
  on $M$, which is connected.
  That allows us
  to take
  \begin{equation}
    \rho = \rho_\epsilon + t\brr{
      \frac{(1/\epsilon-\rho_\epsilon)\chi_A}
        {\norm{(1/\epsilon-\rho_\epsilon)\chi_A}}_{L^1}
      - \frac{\rho_\epsilon \chi_C}
        {\norm{\rho_\epsilon \chi_C}}_{L^1}
    } \in B_{1/\epsilon}\cap \mathcal{P}
  \end{equation}
  for small enough $t$.
  After plugging in $\mu = \rho \dv{g}$ into
  inequality~\eqref{ineq:extremal-condition}, one obtains
  \begin{equation}
    S -\delta \leq S - 2\delta,
  \end{equation}
  which is a contradiction.
\end{proof}

\subsection{Convergence and regularity near stable points}
Consider a sequence of tensors $\mathfrak{t}_\epsilon$
generated by Proposition~\ref{prop:max-sequence}. It is bounded
in $H^1\hat{\otimes}_\pi H^1 \approx \mathcal{N}[H^1]$, so
passing to a subsequence, we have a tensor
$\mathfrak{t} \in \tensor{H^1}$ and a Radon measure
$\mu$ such that $\mu_{\epsilon} \oset{w^*}{\to} \mu$
in $\mathcal{M}$ and $\mathfrak{t}_{\epsilon} \oset{w^*}{\to}\mathfrak{t}$
in $\tensor{H^1}$. By Corollary~\ref{lem:l-inf-convergence},
property~\ref{it:almost-sphere} of the proposition then implies
\begin{equation}\label{eq:limit-map-to-sphere}
  \m{\mathfrak{t}} \equiv 1 \quad\text{ on } M
\end{equation}
\begin{lemma}\label{lem:strong-conv}
  Let
  ${\mu_\epsilon \overset{w^*}{\to} \mu}$ in
  $\mathcal{M}(\Omega)$, and let all the points of $\Omega$ be stable
  for $\brc{\mu_\epsilon}$.
  Let also
  $\mathfrak{t}_\epsilon = \sum_i \phi^i_\epsilon \otimes \phi^i_\epsilon \oset{w^*}{\to} \mathfrak{t}$ in
  $\tensor{H^1(\Omega)}$
  such that
  \begin{enumerate}
    \item $\Delta\phi^i_\epsilon = \phi^i_\epsilon
    \mu_\epsilon$ weakly;
    \item \label{it:lower-than}$\begin{aligned}
      \limsup_{\epsilon\to0}\int
      \brr{\m{\mathfrak{t}_\epsilon} - \m{\mathfrak{t}}}\omega^2 d\mu_\epsilon \leq 0
      \quad\forall \omega \in C^\infty_{0}(\Omega)
    \end{aligned}$
  \end{enumerate}
  Then up to a subsequence,
  \begin{equation}
    \mathfrak{t}_\epsilon \to \mathfrak{t} \text{ (strongly) in } \tensor{H^1(\Omega')}
    \quad\text{and}\quad
    \mu_\epsilon \overset{w^*}{\to} \mu \text{ in }
    \mathcal{B}[H^1_{0}(\Omega')]
  \end{equation}
  for any $\Omega' \Subset \Omega$, i.e. $\int \m{\tau} d\mu_\epsilon \to \int \m{\tau} d\mu$ for any fixed
  $\tau \in \tensor{H^1_{0}(\Omega')}$.
\end{lemma}
\begin{proof}
  Since the norms are local, it is enough to show that there exists a
  neighborhood $W \subset \Omega$ of a point
  $p$, such that the stated
  convergences take place.

  Let us choose a neighborhood $U \subset \Omega$ and a subsequence $\brc{\mu_\epsilon}$
  from the definition of a stable point. Then
  Remark~\ref{rem:quadr-form-bound} implies
  that $\brc{\mu_\epsilon}$ is bounded in
  $\mathcal{B}[H^1_{0}(U)]$, hence precompact in the weak$^*$
  topology and $\mathcal{B}[H^1_{0}(U)] = \tensor{H^1_0(U)}^*$. By looking at the integrals
  $\int \phi^2 d\mu_\epsilon$ of smooth functions
  $\phi \in C^\infty_{0}(U)$, one sees that $\mu$ is the only
  accumulation point of $\brc{\mu_\epsilon}
  \subset\mathcal{B}[H^1_{0}]$.

  Now, $\mathfrak{t}_{\epsilon}$ converges weakly$^*$ to $\mathfrak{t}$ in $\tensor{H^{1}}$,
  and hence strongly in $\tensor{L^2}$ by Corollary~\ref{lem:l-inf-convergence}.
  Let us make use of
  Lemma~\ref{lem:tensor-decomp}. We have $\mathfrak{p}_\epsilon + \mathfrak{q}_\epsilon = 1$,
  \begin{equation}
    \mathfrak{p}_\epsilon\mathfrak{t}_\epsilon \mathfrak{p}_\epsilon = \sum_i \mathfrak{p}_\epsilon(\phi^i_\epsilon) \otimes \mathfrak{p}_\epsilon(\phi^i_\epsilon)
    \to \mathfrak{t} \text{ in } \tensor{H^{1}}
  \end{equation}
  and
  \begin{equation}
    \mathfrak{q}_\epsilon\mathfrak{t}_\epsilon \mathfrak{q}_\epsilon =
    \sum_i \mathfrak{q}_\epsilon (\phi^i_\epsilon) \otimes \mathfrak{q}_\epsilon(\phi^i_\epsilon)
    \overset{w^*}{\to} 0 \quad\text{in } \tensor{H^{1}}.
  \end{equation}
  This again implies
  \begin{equation}\label{eq:q-to-zero}
    \m{\mathfrak{q}_\epsilon\mathfrak{t}_\epsilon \mathfrak{q}_\epsilon} \to 0 \quad\text{in  } L^1.
  \end{equation}

  Let us choose a neighborhood $W \Subset U$ of $p$ and a function
  $\omega \in C^\infty_0(U)$ such that $\omega|_W \equiv 1$.
  For weakly$^*$ converging sequences (see Remark~\ref{rem:weak-star-conv}), one has
  \begin{equation}
    \norm{\mathfrak{t}}_{\tensor{H^1(W)}}
    \leq \liminf_{\epsilon\to0} \norm{\mathfrak{t}_\epsilon}_{\tensor{H^1(W)}}.
  \end{equation}
  It remains to show that
  \begin{equation}
    \limsup_{\epsilon\to0} \int \brr{\md{\mathfrak{t}_\epsilon}
    -\md{\mathfrak{p}_\epsilon\mathfrak{t}_\epsilon\mathfrak{p}_\epsilon}}\omega^2dv_g
    \leq 0.
  \end{equation}
  This would imply $\norm{\mathfrak{t}}_{\tensor{H^1(W)}} = \lim_{\epsilon\to0} \norm{\mathfrak{t}_\epsilon}_{\tensor{H^1(W)}}$, and therefore,
  $\mathfrak{t}_\epsilon \to \mathfrak{t}$ in $\tensor{H^1(W)}$ by Proposition~\ref{prop:tensor-convergence}.

  The strong convergence~\eqref{eq:q-to-zero} leads to
  \begin{gather}
    \sum_i\int \abs{d\omega}^2 \abs{\mathfrak{q}_\epsilon(\phi^i_\epsilon)}^2 dv_g
    = \int \abs{d\omega}^2 \m{\mathfrak{q}_\epsilon\mathfrak{t}_\epsilon \mathfrak{q}_\epsilon} dv_g \to 0
    \\ \shortintertext{and by the Cauchy--Schwarz inequality,}
    \abs{\sum_i\int \brt{\omega d\mathfrak{q}_\epsilon(\phi^i_\epsilon),
    \mathfrak{q}_\epsilon(\phi^i_\epsilon) d\omega} dv_g}
    =\abs{\int \sum_i\brt{\omega d\mathfrak{q}_\epsilon(\phi^i_\epsilon),
    \mathfrak{q}_\epsilon(\phi^i_\epsilon) d\omega} dv_g}
    \\\leq \int \abs{\omega}\sqrt{\md{\mathfrak{q}_\epsilon \mathfrak{t}_\epsilon \mathfrak{q}_\epsilon}}
    \sqrt{\m{\mathfrak{q}_\epsilon \mathfrak{t}_\epsilon \mathfrak{q}_\epsilon}}\abs{d\omega} dv_g \to 0,
  \end{gather}
  since $\norm{\md{\mathfrak{t}_\epsilon}}_{L^1}$ is bounded and
  $\norm{\m{\mathfrak{q}_\epsilon \mathfrak{t}_\epsilon \mathfrak{q}_\epsilon}}_{L^1} \to 0$.
  Therefore by~\eqref{ineq:quadr-form-bound},
  \begin{equation}\label{ineq:from-quadr-form-bound}
    \begin{gathered}
      \int\m{\mathfrak{q}_\epsilon \mathfrak{t}_\epsilon \mathfrak{q}_\epsilon} \omega^2d\mu_\epsilon
      =\sum_i\int \abs{\mathfrak{q}_\epsilon (\phi^i_\epsilon)\omega}^2d\mu_\epsilon
      \leq \sum_i\int \abs{\omega d\mathfrak{q}_\epsilon (\phi^i_\epsilon) +
      \mathfrak{q}_\epsilon (\phi^i_\epsilon)d\omega}^2 dv_g
      \\\leq \sum_i\int \abs{d\mathfrak{q}_\epsilon (\phi^i_\epsilon)}^2 \omega^2dv_g +o(1)
      = \int\md{\mathfrak{q}_\epsilon\mathfrak{t}_\epsilon\mathfrak{q}_\epsilon} \omega^2dv_g +o(1)
    \end{gathered}
  \end{equation}
  as $\epsilon \to 0$. In a similar way,
  \begin{gather}
    \sum_i\int \brt{d\phi_\epsilon^i, \mathfrak{q}_\epsilon(\phi^i_\epsilon)d\omega^2}dv_g
    =\int \sum_i\brt{d\phi_\epsilon^i, d\omega^2}\mathfrak{q}_\epsilon(\phi^i_\epsilon) dv_g
    \\\leq \int \abs{d\omega^2} \sqrt{\md{\mathfrak{t}_\epsilon}}\sqrt{\m{\mathfrak{q}_\epsilon \mathfrak{t}_\epsilon \mathfrak{q}_\epsilon}}dv_g
    \to  0.
  \end{gather}
  So, testing the equation $\Delta \phi^i_\epsilon = \phi^i_\epsilon \mu_\epsilon$
  against $\mathfrak{q}_\epsilon (\phi^i_\epsilon)\omega^2$, one obtains
  \begin{equation}\label{ineq:from-eigen-func-eq}
    \begin{gathered}
      \int \brr{\md{\mathfrak{q}_\epsilon\mathfrak{t}_\epsilon} + \md{\mathfrak{t}_\epsilon \mathfrak{q}_\epsilon}}\omega^2dv_g
      =2\sum_i\int \brt{d\phi_\epsilon^i, d\mathfrak{q}_\epsilon (\phi^i_\epsilon)}\omega^2dv_g
      \\= 2\sum_i\int\brt{\phi_\epsilon^i, \mathfrak{q}_\epsilon (\phi^i_\epsilon)}\omega^2d\mu_\epsilon+o(1)
      = \int \brr{\m{\mathfrak{q}_\epsilon\mathfrak{t}_\epsilon }+\m{\mathfrak{t}_\epsilon \mathfrak{q}_\epsilon}}\omega^2d\mu_\epsilon+o(1).
    \end{gathered}
  \end{equation}
  Now, one can
  sum~\eqref{ineq:from-eigen-func-eq} with~\eqref{ineq:from-quadr-form-bound}
  by using the equality
  \begin{equation}
    \mathfrak{q}_\epsilon\mathfrak{t}_\epsilon + \mathfrak{t}_\epsilon \mathfrak{q}_\epsilon
    = 2\mathfrak{q}_\epsilon\mathfrak{t}_\epsilon\mathfrak{q}_\epsilon +
    \mathfrak{p}_\epsilon\mathfrak{t}_\epsilon\mathfrak{q}_\epsilon
    +\mathfrak{q}_\epsilon\mathfrak{t}_\epsilon\mathfrak{p}_\epsilon
  \end{equation}
  to obtain
  \begin{equation}
    \int \brr{\md{\mathfrak{t}_\epsilon} - \md{\mathfrak{p}_\epsilon\mathfrak{t}_\epsilon \mathfrak{p}_\epsilon}}\omega^2dv_g
    \leq \int \brr{\m{\mathfrak{t}_\epsilon} - \m{\mathfrak{p}_\epsilon\mathfrak{t}_\epsilon \mathfrak{p}_\epsilon}}\omega^2d\mu_\epsilon
    +o(1) \leq o(1),
  \end{equation}
  which completes the proof.
\end{proof}

\begin{lemma}\label{lem:regularity-in-good-pt}
  Let
  $\mu_\epsilon \overset{w^*}{\to} \mu \in \mathcal{M}(\Omega)$. Let
  $\mathfrak{t}_\epsilon \oset{w^*}{\to} \mathfrak{t}$ in
  $\tensor{H^1(\Omega)}$ with the properties~(1)--(2)
  of Lemma~\ref{lem:strong-conv} and additionally suppose that
  \begin{itemize}
    \item $[\mathfrak{t}] \equiv 1$ on $\Omega$;
    \item $\Omega$ contains only a discrete set of unstable points denoted by $D$.
  \end{itemize}
  Then there exists
  a weakly harmonic map $\Phi\colon \Omega \to \Sph^\infty$, i.e. $\Phi \in H^1 \cap L^\infty$ and
  $\Delta \Phi = \abs{d\Phi}^2\Phi$.
  Moreover, one has the following:
  \begin{itemize}
    \item $\mu^c = \abs{d\Phi}^2 dv_g$, where $\mu^c$ denotes the continuous part of $\mu$;
    \item $\Phi$ is locally stable on $\Omega$.
  \end{itemize}
\end{lemma}
\begin{proof}
  Let us decompose $\mu$ into its continuous and atomic parts
  \begin{equation}\label{eq:meas-decomp}
    \mu=\mu^c+\sum_{p \in A(\mu)}w_p\delta_p,
  \end{equation}
  where $w_p > 0$ and $A(\mu)$ is the set of atoms of $\mu$.
  Inequality \eqref{ineq:quadr-form-bound} implies that
  $A(\mu) \subset D$.

  The operators $\brc{\mathfrak{t}_\epsilon}$ are compact, self-adjoint,
  and positive, and so is $\mathfrak{t}$, which means we can write
  \begin{equation}
    \mathfrak{t} = \sum \phi^i \otimes \phi^i
  \end{equation}
  for some mutually orthogonal $\phi^i \in H^1(\Omega)$. Since
  \begin{equation}\label{eq:l2-sphere}
    \m{\mathfrak{t}} = \sum \abs{\phi^i}^2 \equiv 1,
  \end{equation}
  let us define
  $\Phi \colon \Omega \to \Sph^\infty \subset \ell^2$ by setting $\Phi = (\cdots, \phi^i, \cdots)$.
  We are going to prove that $\Delta \phi^i = \phi^i \mu^c$.

  Let us take $\omega, \chi \in C^\infty_{0}(\Omega\setminus D)$
  such that $\chi \equiv 1$ on $\supp \omega$ and consider
  the rank $1$ operator
  \begin{equation}
    \mathfrak{w}_i = \frac{1}{\norm{\phi^i}^2_{H^1}}\omega \otimes \phi^i.
  \end{equation}
  Then we have $\mathfrak{w}_i \mathfrak{t} = \omega \otimes \phi^i$ and
  $\mathfrak{w}_i \mathfrak{t}\chi := \omega \otimes (\phi^i \chi)$.
  Since $\mathfrak{t}_\epsilon \oset{w^*}{\to} \mathfrak{t}$ and
  $\mathfrak{w}_i$ is a compact operator, we have
  \begin{equation}
    \int \brt{d\omega,d\phi^i} dv_g
    =\int \md{\mathfrak{w}_i \mathfrak{t}} dv_g
    = \lim_{\epsilon\to 0} \int \md{\mathfrak{w}_i \mathfrak{t}_\epsilon} dv_g
  \end{equation}
  At the same time, by Lemma~\ref{lem:strong-conv} applied to
  $\Omega' \Subset \Omega \setminus D$, the strong
  $H^1$-convergence
  $\mathfrak{w}_i \mathfrak{t}_\epsilon \chi \to \mathfrak{w}_i \mathfrak{t} \chi
  \in \tensor{H^1_0(\Omega')}$ and the weak convergence
  $\mu_{\epsilon} \oset{w^*}{\to} \mu$
  as bilinear forms on $H^1_{0}(\Omega')$ yield
  \begin{multline}
    \int \md{\mathfrak{w}_i \mathfrak{t}_\epsilon} dv_g
    = \int \m{\mathfrak{w}_i \mathfrak{t}_\epsilon} d\mu_\epsilon
    = \int \m{\mathfrak{w}_i \mathfrak{t}_\epsilon \chi} d\mu_\epsilon
   \\ = \int \m{\mathfrak{w}_i (\mathfrak{t}_\epsilon - \mathfrak{t}) \chi} d\mu_\epsilon
      + \int \m{\mathfrak{w}_i \mathfrak{t} \chi} d\mu_\epsilon
   \to \int \m{\mathfrak{w}_i \mathfrak{t} \chi} d\mu
    = \int \omega \phi^i d\mu,
  \end{multline}
  which means $\Delta \Phi = \Phi \mu^c$ on $\Omega\setminus D$. Since
  discrete sets have capacity zero, we obtain
  \begin{equation}\label{eigen.map}
    \Delta \Phi = \Phi \mu^c \text{ on } \Omega.
  \end{equation}

  Equation~\eqref{eq:l2-sphere} implies $\Phi \cdot d\Phi = 0$, and
  after testing equation~\eqref{eigen.map} against $\omega \Phi \in H^1$,
  we conclude that
  \begin{equation}
    \mu^c=|d\Phi|^2 \dv{g}
    \quad\text{and}\quad
    \Delta \Phi = |d\Phi|^2 \Phi
  \end{equation}
  in the weak sense. Thus, $\Phi\colon \Omega \to \mathbb{S}(\ell^2)$
  is a weakly harmonic map.

  Inequality~\eqref{ineq:stable-map} obviously holds in a neighborhood of a
  stable point. On a neighborhood $U$ of an unstable point given
  by Lemma~\ref{lem:unstable-ring}, after taking the limit, one has
  \begin{equation}
    \int \omega^2 \abs{d\Phi}^2 dv_g \leq \int \abs{d\omega}^2 dv_g
    \quad \forall \omega \in C^\infty\brr{U\setminus\brc{p}}
  \end{equation}
  Again, the inequality continues to hold over a discrete set.
\end{proof}

\subsection{No bubbling phenomenon}\label{app:blow-up}

Let $\brr{M,g}$ be a compact Riemannian manifold possibly with
a nonempty boundary $\partial M$ and $\dim M = m \geq 2$.
Denote by $\mathcal{M}_c(M)$ the space of all nonnegative continuous Radon measures
on $M$. Consider a sequence $\brc{\mu_\varepsilon} \subset \mathcal{M}_c(M)$
such that $\mu_\varepsilon \oset{w^*}{\to}\mu$. Notice that $\mu$ decomposes as its continuous and atomic parts,
\begin{equation}
  \mu = \mu^c + \sum_{p \in A(\mu)} w_p \delta_{p},
\end{equation}
where $A(\mu) = \brc{p \in M \colon \mu(\brc{p}) > 0}$ and $w_p  = \mu(\brc{p})$.

\begin{definition}\label{def:bubble-tree}
  Fix $\delta > 0$. We say that a converging sequence
  $\brc{\mu_\varepsilon} \subset \mathcal{M}_c(M)$,
  $\mu_\varepsilon \oset{w^*}{\to}\mu$,
  has a $\delta$-\emph{bubbling property} if up to a subsequence,
  there exists
  a finite collection of maps
  $T^i_\epsilon \colon U_i \to \R^m$, where
  $U_i \subset M$ is a coordinate neighborhood of $p_i \in A(\mu)$
  with the following properties.
  \begin{enumerate}
    \item The maps $T^i_\epsilon$ have the form
      \begin{equation}
        T^i_\epsilon(x) = \frac{x - p^i_\epsilon}{\alpha^i_\epsilon},
        \quad\text{ where } p^i_\epsilon \to p_i, \ \alpha^i_\epsilon \searrow 0.
      \end{equation}
    \item Either for any $R > 0$ and small enough $\epsilon$, the balls
      $B_{\alpha^i_\epsilon R}(p^i_\epsilon)$ and
      $B_{\alpha^j_\epsilon R}(p^j_\epsilon)$ do not intersect,
      or one of them is contained in the other, say
      $B_{\alpha^i_\epsilon R}(p^i_\epsilon) \subset B_{\alpha^j_\epsilon R}(p^j_\epsilon)$, and
      $(T^j_\epsilon)[B_{\alpha^i_\epsilon R}(p^i_\epsilon)]$ converge to a point in $\R^m$,
      in other words,
      \begin{equation}\label{eq:dif-rate-condition}
        \frac{\operatorname{dist}(p^i_\epsilon,p^j_\epsilon)}{\alpha^i_\epsilon + \alpha^j_\epsilon}
        +\frac{\alpha^i_\epsilon}{\alpha^j_\epsilon}
        +\frac{\alpha^j_\epsilon}{\alpha^i_\epsilon} \to \infty.
      \end{equation}
    \item $(T^i_\epsilon)_* [\mu_\epsilon] \overset{w^*}{\to} \mu_i \in \mathcal{M}(\R^m)$ and
      (one can suppose that $\mu_i^c \neq 0$)
      \begin{equation}\label{ineq:alomst-no-loss}
        (1-\delta)\mu(M) \leq \mu^c(M) + \sum_i \mu_i^c(\R^m) \leq  \mu(M).
      \end{equation}
  \end{enumerate}
  Here the convergence $(T_\epsilon)_* [\mu_\epsilon] \overset{w^*}{\to}
  \mu$ means
  \begin{equation}
    \int_{\R^m} \phi(x') d\mu(x') = \int_M \phi\brr{\frac{x - p^i_\epsilon}{\alpha^i_\epsilon}}
    d\mu_\epsilon(x),\quad \forall \phi \in C_0(\R^m).
  \end{equation}
\end{definition}
The following proposition was proved in~\cite[Appendix~A.1]{vinokurov:2025:sym-eigen-val}.
\begin{proposition}\label{prop:bubble-tree-existence}
  Let $\brc{\mu_\varepsilon} \subset \mathcal{M}_c(M)$,
  $\mu_\varepsilon \overset{w^*}{\to} \mu$,
  and $\lambda_k(\mu_\epsilon) \geq c > 0$. Then for any $\delta > 0$,
  $\brc{\mu_\varepsilon}$ has $\delta$-bubbling property
  with at most $k$ maps $\brc{T^i_\epsilon}$. If
  $\supp \mu_\epsilon \subset \partial M$, one has
  $T^i_\epsilon \colon U_i \to \R^m_+$.
\end{proposition}

\begin{proposition}\label{prop:no-atoms}
  If $\dim M \geq 3$, $\brc{\mu_\varepsilon} \subset \mathcal{M}_c(M)$,
  $\mu_\varepsilon \overset{w^*}{\to} \mu$, and $\lambda_k(\mu_\varepsilon) \ge c > 0$,
  then $\mu$ has no atoms, i.e. $\mu \in \mathcal{M}_c(M)$.
\end{proposition}
\begin{proof}
  Since $\brc{\mu_\varepsilon}$ has $\delta$-bubbling property,
  any atom gives rise to a collection of maps $T_\epsilon \colon B(1,p_i) \to \R^m$ such that
  $\tilde{\mu}_\epsilon := (T_\epsilon)_* [\mu_\epsilon] \oset{w^*}{\to} \tilde{\mu} \in
  \mathcal{M}_c(\R^m\setminus D)$, where $\tilde{\mu}\neq 0$ and $D$ is a finite set. Then, we choose any
  $k+1$ functions with pairwise disjoint supports
  $\psi_i \in C^\infty_0(\R^m\setminus D)$ such that
  \begin{equation}
    \int_{\mathbb{R}^m} \psi_i^2 d\tilde{\mu} \ge c_1 \text{ and } \int_{\mathbb{R}^m} \abs{d\psi_i}^2 dx \le c_2.
  \end{equation}
  Thus,
  \begin{equation}
    \lambda_k(\mu_\varepsilon) \le \max_{0\le i \le k}
    \frac{\int_{M} \abs{d\psi_i\brr{\frac{x - p_\varepsilon}{\alpha_\varepsilon}}}^2_g dv_g}
    {\int_{M} \psi_i\brr{\frac{x - p_\varepsilon}{\alpha_\varepsilon}}^2 d\mu_\varepsilon}
    \le  \alpha_\varepsilon^{m-2}\max_{0\le i \le k}
    \frac{c_3\int_{\mathbb{R}^m} \abs{d\psi_i}^2 dx}
    {\int_{\mathbb{R}^m} \psi_i^2 d\tilde{\mu}_\epsilon} \to 0.
  \end{equation}
\end{proof}

\subsection{Proof of Theorem~\ref{thm:main}}
  Starting with the maximizing sequence from Proposition~\ref{prop:max-sequence},
  replace $\mu_\epsilon$ by $\lambda_k(\mu_\epsilon) \mu_\epsilon$. We may then assume that
  $\lambda_k(\mu_\epsilon) = 1$, while $\mu_\epsilon(M) \to \mathcal{V}_k(g)$. In particular, $\ind \mu_\epsilon \leq k$, and by Remark~\ref{rem:size_of-bad}, the
  set $D$ of unstable points has cardinality at most $k$. Then we apply
  Lemma~\ref{lem:strong-conv} and Lemma~\ref{lem:regularity-in-good-pt}
  on $\Omega = M\setminus D$ and $\Omega = M$, respectively
  (condition~\ref{it:lower-than} of Lemma~\ref{lem:strong-conv} is trivially satisfied
  since $\m{\mathfrak{t}_\epsilon} \leq 1 = \m{\mathfrak{t}}$ by~\eqref{eq:limit-map-to-sphere}).
  Thus, we have a locally stable harmonic map $u\in H^1(M, \Sph^\infty)$.
  Moreover, the limiting measure $\mu$ has the form
  $$
  \mu=|du|^2dv_g+\sum_{p \in A(\mu)}w_p\delta_p
  $$
  and $\mathcal{V}_k(g) = \mu(M)$. Since $\lambda_k(\mu_\epsilon) = 1$,
  we conclude from Proposition~\ref{prop:no-atoms} that $\mu=|du|^2dv_g$.

  Recall that $\tilde{\mu} \mapsto \lambda_k(\tilde{\mu})$ is upper semi-continuous
  with respect to the weak$^*$ convergence, so $\lambda_k(\mu) \geq 1$, which is equivalent to $\ind \mu \leq k$, and
  \begin{equation}
    \mathcal{V}_k(g) = \mu(M) \leq \overline{\lambda}_k(\mu) \leq \mathcal{V}_k(g),
  \end{equation}
  where the last inequality holds since the suprema over $\tilde{\mu} \in L^\infty$ and
  $\tilde{\mu} \in L^1$ are the same. Hence equality holds throughout. In particular, $\lambda_k(\mu) = 1$.

  The regularity of $u$ follows from Theorem~\ref{thm:reg-harm-map}, and Corollary~\ref{cor:finite-sphere-map}
  implies that the image of $u$ lies in a finite-dimensional subsphere.

\subsection{Steklov eigenvalue optimization}
If the manifold $M$ is compact, but has a nonempty boundary $\partial M$, we define
\begin{equation}\label{Steklov-opt}
  \mathcal{V}^\partial_k(g) = \sup \set*{\overline{\lambda}_k(\mu)}{
    \mu \in C^\infty(\partial M),\ \mu > 0
  }
\end{equation}
When $\supp \mu \subset \partial M$, the variational eigenvalues come from
the corresponding weighted Steklov eigenvalue problem:
 \begin{equation}
  \begin{cases}
    \Delta_g \phi = 0 & \quad\text{on}\quad \Int M
    \\ \partial_\nu \phi = \lambda_k \phi \mu  &\quad\text{on}\quad \partial M.
  \end{cases}
 \end{equation}
The eigenfunctions of critical measures generate free boundary harmonic maps
to Euclidean unit balls, $u \colon M
\to \mathbb{B}^{n}$, i.e. $\Delta_g u = 0$,
$u(\partial M) \subset \Sph^{n-1}$, and $\partial_\nu u \parallel u$.
As a consequence (see, e.g., \cite{Karpukhin-Metras:2022:higher-dim}), this implies $\mu = \abs{\partial_\nu u} ds_g$, where $ds_g$ is the boundary
volume measure.

In this setting, one can prove the following existence result.
\begin{theorem}\label{thm:main-steklov}
  Let $(M,g)$ be a compact connected Riemannian manifold of dimension $m \geq 3$ with
  nonempty boundary $\partial M$.
  Then there exists a locally stable free boundary harmonic map
  $u \in H^1(M,\mathbb{B}^\infty)$ such that
  \begin{equation}
    \lambda_k(\abs{\partial_\nu u} ds_g) = 1
    \quad\text{and}\quad
    \mathcal{V}_k^\partial(g) = \int_{M} \abs{du}^2_g dv_g =
    \int_{\partial M} \abs{\partial_\nu u} ds_g.
  \end{equation}
\end{theorem}
\begin{proof}[Sketch of proof]
  The proof closely follows the argument used in Theorem~\ref{thm:main},
  with suitable adaptations for the Steklov case.
   In particular, analogues of Proposition~\ref{prop:max-sequence},
  Lemma~\ref{lem:strong-conv}, and Lemma~\ref{lem:regularity-in-good-pt}
  can be established with straightforward modifications. For a detailed
  version of the argument in the case of dimension $m = 2$,
  see~\cite{vinokurov:2025:sym-eigen-val}.

  For instance, one works with measures in $L^1(\partial M)$ or $L^\infty(\partial M)$
  instead, and in
  Lemmas~\ref{lem:strong-conv} and~\ref{lem:regularity-in-good-pt},
  the local Sobolev spaces $H^1_0(\Omega)$ must be replaced by
  \begin{align}
    H^1_{co}(\Omega) &:=\overline{\set{\phi \in C^\infty(M)}{\supp \phi \subset \Omega}}^{H^1}
    \\ &= \set*{\phi \in H^1(M)}{\supp \phi \subset \Omega},
  \end{align}
  i.e., functions whose support in $\Omega$ may intersect the boundary $\partial M$
  when $\Omega \cap \partial M \neq \varnothing$.

  Finally, Proposition~\ref{prop:no-atoms} remains valid in this setting, implying that the limiting boundary measure contains no atoms.
\end{proof}

\section{Harmonic maps to \texorpdfstring{$\Sph^{\infty} = \Sph(\ell^2)$}{Lg}}
\label{sec:harm-maps}

Throughout this section, we
suppose that geodesic balls are smooth, e.g. $r < \operatorname{inj} \Omega$ as long as
$B_r(p) \subset \Omega$, and constants $C = C(g)$ may depend on suitable $\norm{g}_{C^k}$-norms of $g$ and $g^{-1}$.

\subsection{Partial regularity}
The goal of this section is to extend the classical regularity
results for harmonic maps to the infinite-dimensional setting.  Consider a map $u\colon \Omega \to \Sph^\infty$ and define
\begin{equation}
  E[u] = E[u; \Omega] = \int_{\Omega} \abs{du}^2.
\end{equation}
\begin{definition}
  A harmonic map $u \in H^1(\Omega, \Sph^\infty)$ is called \emph{stationary}
  if one has
  \begin{equation}
    \frac{d}{dt}\bigg|_{t=0} E[u \circ \phi_t] = 0
  \end{equation}
  for all differentiable 1-parameter families of diffeomorphisms
  $\phi_t\colon \Omega \to \Omega$ satisfying $\phi_0 = \id$ and
  $\phi_t|_{\Omega \setminus K} = \id|_{\Omega \setminus K}$ for some
  compact subset $K \subset \Omega$.
\end{definition}
The classical energy monotonicity formula is independent of the dimension
of the target space and the same proof holds in the $\ell^2$-setting.
\begin{proposition}[Monotonicity formula, \cite*{Price:1983:monoton-formula}]
  For a stationary harmonic map $u \in H^1(\Omega, \Sph^\infty)$, there exists
  a constant $C = C(m, \kappa)$, where $m = \dim M$ and
  $\kappa \geq \operatorname{sec}_g|_\Omega$ is an upper bound on the sectional curvature, such that
  for any $B_r(p) \subset \Omega$, one has
  \begin{equation}\label{eq:monoton-formula}
    \frac{d}{dr}  \brr{e^{Cr}r^{2-m}E[u; B_r(p)]} \geq 0.
  \end{equation}
\end{proposition}

\begin{definition}\label{def:stable-map}
  We call a harmonic map $u \in H^1(\Omega, \Sph^\infty)$ into the
  \emph{infinite-dimensional} sphere \emph{stable} if
  \begin{equation}\label{ineq:stable-map}
    \int \omega^2 \abs{du}^2 dv_g \leq \int \abs{d\omega}^2 dv_g
    \quad \forall \omega \in C^\infty_0\brr{\Omega}
  \end{equation}
\end{definition}
\begin{remark}
  As one can easily see from the proof of~\cite[Theorem~2.7]{Schoen-Uhlenbeck:1984:min-harm-maps},
  this definition is equivalent to the classical definition
  of stability in the infinite-dimensional setting:
  \begin{equation}
    \frac{d^2}{dt^2}\bigg|_{t=0}\int \abs{d\brr{\frac{u + tw}{\abs{u + tw}}}}^2 dv_g
    \geq 0 \quad \forall w \in C_0^\infty(\Omega, \ell^2).
  \end{equation}
\end{remark}

\begin{lemma}\label{lem:stable-energy-min}
  A stable harmonic map $u \in H^1(\Omega, \Sph^\infty)$ is energy minimizing,
  i.e.
  \begin{equation}
    \int_{\Omega} \abs{du}^2 \leq \int_{\Omega} \abs{dw}^2
    \quad \forall w \in H^1\brr{\Omega, \Sph^\infty} \text{ s.t. } w - u \in H^1_0(\Omega, \ell^2).
  \end{equation}
  In particular, it is stationary.
\end{lemma}
\begin{proof}
  We will need the harmonic map equation $\Delta u = \abs{du}^2u$ and the identity
  $2\brt{u, u - w} = \abs{u - w}^2$ for $\abs{u} = \abs{w} = 1$. If
  $w \in H^1\brr{\Omega, \Sph^\infty}$ and
  $u - w \in H^1_0(\Omega, \ell^2)$, one has
  \begin{multline}
    2\int \brt{du, du - dw} = 2\int \brt{u, u - w} \abs{du}^2
    \\ = \int \abs{u - w}^2 \abs{du}^2 \leq \int \abs{du - dw}^2,
  \end{multline}
  where we used the stability inequality~\eqref{ineq:stable-map}. The inequality
  of the lemma then immediately follows.
\end{proof}

In what follows, we aim to generalize the well-known partial regularity
results for harmonic maps when the target is the Hilbert unit sphere $\Sph^\infty$.
Fortunately, Hilbert-valued function spaces and PDE theory exhibit
minimal differences compared to their standard counterparts.
For example, see \cite{Hytonen-Neerven-Veraar-Weis:2016:analysis-banach-1,
Hytonen-Neerven-Veraar-Weis:2023:analysis-banach-3} and the references therein
for a detailed treatment of Banach-valued analysis.

We start with the following extension lemma, whose proof can be found, e.g.,
in \cite[Theorem~2.1.9]{Hytonen-Neerven-Veraar-Weis:2016:analysis-banach-1} or \cite[Section~26.3]{Defant-Floret:1993:tensor-norms}.
\begin{lemma}
  \label{lem:scalar-to-vector}
  Let $(X, \mu)$, $(Y,\nu)$ be measure spaces, $p, q \in [1,\infty)$, and
  $T\colon L^q(\mu) \to L^p(\nu)$ be a continuous linear operator.
  For a Hilbert space $\mathcal{H}$, consider
  $L^q(\mu) \otimes \mathcal{H} \approx Span \set{f(\cdot)v}{f\in L^q(\mu),\ v \in \mathcal{H}}$,
  a dense subset of $L^q(\mu, \mathcal{H})$. Then the operator  $T\otimes \id_{\mathcal{H}} \colon
  L^q(\mu) \otimes \mathcal{H} \to L^p(\nu)\otimes \mathcal{H}$ continuously extends to
  \begin{equation}
    T\otimes \id_{\mathcal{H}} \colon L^q(\mu, \mathcal{H}) \to L^p(\nu,\mathcal{H}).
  \end{equation}
\end{lemma}
Let us recall that the Laplacian can be expressed as $\Delta = \delta d$, where
$\delta\colon C^\infty(M,\Lambda^1)\to C^\infty(M)$
is formally adjoint to $d$ and $\Lambda^1$ stands for $1$-forms on $M$.
\begin{lemma}\label{lem:lp-dirichlet}
  Let $B_1 \subset M$ be a smooth geodesic ball, $u \in H^{1,p}_0(B_1, \ell^2)$ be an $\ell^2$-valued function, $p\in (1,\infty)$, and
  $\omega \in L^p(B_1, \ell^2 \otimes \Lambda^1)$ be an $\ell^2$-valued 1-form
  such that
  \begin{equation}
    \Delta u = \delta \omega
  \end{equation}
  weakly. Then there exists a constant $C = C(p, m, g)$ such that
  \begin{equation}
    \norm{du}_{L^p(B_1)} \leq C\norm{\omega}_{L^p(B_1)}.
  \end{equation}
\end{lemma}
\begin{proof}
  In the class of smooth domains, the scalar problem $\Delta \phi = f$,
  $f\in H^{-1,p}(B_1)$ has a unique solution $\phi\in H^{1,p}_0(B_1)$ so that
  $\Delta_D^{-1}\colon H^{-1,p} \to H^{1,p}_0$ is continuous. By looking
  at the coordinate functions $u^i$, we see that the solution of
  the $\ell^2$-valued problem should be unique as well. Then Lemma~\ref{lem:scalar-to-vector}
  tells us that the operator $d \Delta_D^{-1} \delta\colon L^p(B_1, \ell^2 \otimes \Lambda^1) \to L^p(B_1, \ell^2 \otimes \Lambda^1)$
  is continuous whence the desired estimate follows.
\end{proof}
\begin{remark}
  The alternative proof of the above lemma could be done by constructing
  a function $\phi \in H^{1,q}_0(B_r, \ell^2)$ such that $\norm{du}_{L^p}\norm{d\phi}_{L^q}\leq c_1 \int \brt{du, d\phi}$,
  $1/p + 1/q = 1$. It is implicitly used in showing the solvability of the scalar
  Dirichlet problem.
  It seems such a function can be constructed directly by means of a partition of unity and
  the standard Riesz transform technique in local coordinate charts,
  see, e.g., \cite[Lemma~4.5]{Iwaniec-Scott-Stroffolini:1999:lp-hodge-decompose}.
\end{remark}
\begin{definition}
  For an arbitrary Banach space $X$ and a bounded domain $\Omega \subset M$,
  we say that a function $u\colon\Omega \to X$ belongs to $\BMO(\Omega)$ if
  \begin{equation}
    \norm{u}_{\BMO(\Omega)} := \sup_{B \subset \Omega}
    \ \dashint_{B}\abs{u - [u]_{B}} < \infty,
  \end{equation}
  where $B = B_r(p)$ and  $[u]_{B} := \dashint_{B} u$.
\end{definition}
The same proof of the John--Nirenberg inequality works for any Banach space
$X$, see~\cite*{John-Nirenberg:1961:bmo, Buckley:1999:bmo-doubling-spaces}.
Alternatively, \cite[Theorem 3.2.30]{Hytonen-Neerven-Veraar-Weis:2016:analysis-banach-1}
provides a martingale approach for Banach-valued functions.
\begin{proposition}[John--Nirenberg inequality]
For a domain $\Omega \subset M$ and a number $p\in \brr{0,\infty}$, there exists a constant
$C = C(m,p,g)$ such that
\begin{equation}
  C^{-1} \norm{u}_{\BMO(\Omega)} \leq
  \sup_{B \subset \Omega} \brr{\dashint_{B}\abs{u - [u]_{B}}^p}^{\frac{1}{p}}
  \leq C \norm{u}_{\BMO(\Omega)}.
\end{equation}
\end{proposition}
\begin{lemma}[Small energy regularity]\label{lem:small-energy-reg}
  Let $B_1(p_0) \subset M$ be a smooth geodesic ball.
  Then there exists $\epsilon = \epsilon(m, g) > 0$ and $\alpha = \alpha(m, g) \in (0,1)$ such that
  for any stationary harmonic map $u \in H^1(B_1(p_0), \Sph^\infty)$ with
  $E[u; B_1(p_0)] < \epsilon$, one has
  $u \in C^{0,\alpha}(B_{1/2}(p_0),\Sph^\infty)$.
\end{lemma}
\begin{proof}
 Our proof is an adaptation of~\cite*{Chang-Wang-Yang:1999:reg-harm-maps} for the
 $\ell^2$-valued setting. Basically, we are looking for a number $0 < s < 1/4$ such that
 \begin{equation}\label{eq:desired-bmo}
  \norm{u}_{\BMO(B_{s}(x))}\leq  \frac{1}{2} \norm{u}_{\BMO(B_1(p_0))}
  \quad \forall x \in B_{1/2}(p_0).
\end{equation}
Then one can iterate the inequality by the following
rescaling argument.  Consider a ball $B_1(0) \subset T_x M$ and
the functions
$u_s \colon B_1(0) \to \ell^2$,
$u_s(y) = u(\operatorname{exp}_{x}(s y))$, with the metric
$g_s(y) = g(\operatorname{exp}_{x}(s y))$.
By monotonicity formula~\eqref{eq:monoton-formula}, its energy
\begin{equation}\label{eq:rescale-energy}
  E[u_s; B_1(0)] = s^{2-m} E[u; B_{s}(x)],
\end{equation}
is small, so
we derive~\eqref{eq:desired-bmo} for $u_s$, and therefore,
\begin{equation}
  \norm{u}_{\BMO(B_{s^k}(x))}\leq  \frac{1}{2^k} \norm{u}_{\BMO(B_1(p_0))}.
\end{equation}
Then one could apply Campanato's lemma~\cite*[Chapter~14, Proposition~A.2]{Taylor:2023:pde-3}
either to $\brt{u(\cdot), b}$, where $b\in \ell^2$ and $\abs{b} \leq 1$, or to $u$ itself,
where the same proof works for Banach-valued functions as long as we know that $C^{0,\alpha} = B^\alpha_{\infty, \infty}$ is
a certain Besov space~\cite[Corollary 14.4.26]{Hytonen-Neerven-Veraar-Weis:2023:analysis-banach-3}.
Thus, we obtain $u \in C^{0,\alpha}(B_{1/4}(p_0), \ell^2)$.

Let us fix some $1 < q <2$, define $p^{-1} = q^{-1} - \frac{1}{m}$ and
denote $B_r := B_r(x)$.
In fact, estimate~\eqref{eq:desired-bmo} will follow if we prove
\begin{equation}\label{eq:concentric-bmo}
  \brr{\dashint_{B_{s}}\abs{u - [u]_{B_{s}}}^p}^{\frac{1}{p}}\leq  \frac{1}{2} \norm{u}_{\BMO(B_{1/4})}
  \quad \forall x \in B_{3/4}(p_0)
\end{equation}
and apply the rescaling argument with $u_\delta$, $\delta = s'/s$, establishing~\eqref{eq:concentric-bmo}
for all $s' \in \brr{0, s}$.

 To prove~\eqref{eq:concentric-bmo}, we choose
 $\frac{1}{5} < r < \frac{1}{4}$ so that
 \begin{equation}
    \int_{\partial B_{r}} \abs{u - A}^p \leq 21 \int_{B_{1/4}} \abs{u - A}^p,
 \end{equation}
 where $A := [u]_{B_{1/4}}$. Let $h\colon B_{r} \to \mathbb{B}^\infty \subset \ell^2$ be
 the harmonic extension of $u|_{\partial B_{r}}$, i.e. $\Delta h = 0$. Writing
 $h(x) - A = \int_{\partial B_{r}} \partial_{r}G_x(y)[u(y) - A]$, where $G_x(y)$ is the Green function
 $-\Delta G_x = \delta_x$, $G_x|_{\partial B_{r}} = 0$, we see that
 the previous inequality yields
 \begin{equation}
  \sup_{B_{1/6}}\abs{d h}^p \leq c_2 \int_{B_{1/4}} \abs{u - A}^p
 \end{equation}
 with $c_2 = c_2(p,m,g)$.

 Let us recall the harmonic map equation $\Delta u^i =  \abs{du}^2 u^i$.
 Since we are dealing with sphere-valued maps, one has
 $\sum_j u^j du^j = 0$, and as a consequence,
 \begin{equation}
  \Delta \brr{u^i - h^i} = \sum_j \brt{u^i du^j - u^j du^i, d(u^j-A^j)}
  = \delta \omega^i,
 \end{equation}
 weakly, where $\delta$ is the formal adjoint to $d$ and
 $\omega^i  = \sum_j \brr{u^j du^i - u^i du^j}(u^j-A^j)$. We can estimate
 $\abs{\omega^i} \leq \brr{\abs{u^i} \abs{du} + \abs{du^i}}\abs{u-A}$, so
 \begin{equation}
  \abs{\omega}^2 \leq 4 \abs{du}^2\abs{u-A}^2
 \end{equation}
 and $\omega \in L^2(\Omega, \ell^2\otimes\Lambda^1)$.
 Applying Lemma~\ref{lem:lp-dirichlet} and the Hölder inequality,
 we obtain
 \begin{align}
  \int_{B_{r}}\abs{d(u-h)}^q
      &\leq c_3 \int_{B_{r}}\abs{\omega}^q
  \\  &\leq c_4 \brr{\int_{B_{r}}\abs{du}^2}^{q/2}
                \brr{\int_{B_{r}}\abs{u - A}^{\frac{2q}{2-q}}}^{\frac{2-q}{2}}
  \\  &\leq c_4 \epsilon^{q/2} \brr{\int_{B_{1/4}}\abs{u - A}^{\frac{2q}{2-q}}}^{\frac{2-q}{2}}.
 \end{align}

 Now, for any $0 < s < 1/6$, the Sobolev inequality and the
 John-Nirenberg inequality yield
 \begin{align}
  \frac{1}{s^m}\int_{B_{s}}\abs{u-h(x)}^p
      &\leq \frac{c_5}{s^m} \int_{B_{r}}\abs{u-h}^p
          + \frac{c_5}{s^m} \int_{B_{s}}\abs{h-h(x)}^p
  \\  &\leq \frac{c_6}{s^m} \brr{\int_{B_{r}}\abs{d(u-h)}^q}^\frac{p}{q}
          +  c_5 s^p \sup_{B_{1/6}}\abs{d h}^p
  \\  &\leq c_7\frac{\epsilon^{p/2}}{s^m} \brr{\int_{B_{1/4}}\abs{u - A}^{\frac{2q}{2-q}}}^{\frac{(2-q)p}{2q}}
          +  c_7 s^p \int_{B_{1/4}} \abs{u - A}^p
  \\  &\leq c_8\brr{\frac{\epsilon^{p/2}}{s^m} + s^p}\norm{u}_{\BMO(B_{1/4})}^p,
 \end{align}
 where we used the John-Nirenberg inequality, and $c_8 = c_8(p,m,g)$. Combining it with the inequality
\begin{equation}
  \dashint_{B_{s}}\abs{u -  [u]_{B_{s}}}^p \leq
  2^p  \dashint_{B_{s}}\abs{u - h(x)}^p,
\end{equation}
one can choose $s$ and then $\epsilon$ to ensure~\eqref{eq:concentric-bmo}.
\end{proof}

\begin{lemma}\label{lem:cont-harm-smooth}
  Let $u\in H^1(\Omega,\Sph^\infty)$ be a (weakly) harmonic map and
  $u \in C^{0,\alpha}$ for some $\alpha > 0$. Then it is smooth.
\end{lemma}
\begin{proof}
  The proof can be found in~\cite*[Chapter~14, Proposition~12B.1]{Taylor:2023:pde-3}.
  Let us show that one has all the necessary ingredients
  in the Hilbert-valued case. If $\R^m \approx U \Subset \Omega$, a parametrix $E \in PDO^{-2}(\R^m)$ of
  the Laplacian $\Delta_g |_U$ has the classical mapping property,
  \begin{equation}
    E \colon H^{s,p}(\R^m, \ell^2) \to H^{s+2,p}(\R^m, \ell^2)
      \quad\text{for } s\in\R,\ p \in (1, \infty),
  \end{equation}
  by Lemma~\ref{lem:scalar-to-vector}.
  The Sobolev embeddings hold for an arbitrary Banach space $X$,
  \begin{equation}
    H^{s,p}(\R^m, X)\hookrightarrow B^s_{p,\infty}(\R^m, X) \hookrightarrow C^{k, \gamma}(\R^m, X)
      \quad\text{for } s\geq 0,\  s - m/p \geq k + \gamma
  \end{equation}
  as one can see from~\cite[Proposition~14.4.18, Corollary~14.4.27]{Hytonen-Neerven-Veraar-Weis:2023:analysis-banach-3}.
  Finally, \cite[Theorem~14.7.12]{Hytonen-Neerven-Veraar-Weis:2023:analysis-banach-3} ensures
  the complex interpolation identity
  \begin{equation}
    H^{s_\theta, p_\theta}(\R^m, \ell^2) = \brs{H^{s_0, p_0}(\R^m, \ell^2),
    H^{s_1, p_1}(\R^m, \ell^2)}_\theta
    \quad\text{ for } p_i \in (1,\infty),\ s_i \in \R,
  \end{equation}
  where $\frac{1}{p_\theta} =  \frac{1-\theta}{p_0} + \frac{\theta}{p_1},\ s_\theta =  (1-\theta)s_0+ \theta s_1$.

  Alternatively, instead of Hölder continuity of $u$, it is sufficient
  to require only $u \in C(\Omega, \ell^2)$.
  Namely, one repeats the arguments of Section~3 in~\cite{Chang-Wang-Yang:1999:reg-harm-maps} to show
  that $u \in C^{1,\gamma}(\Omega, \ell^2)$ for some $\gamma > 0$. The only caveat is that now, the embedding
  $H^1(B_1, \ell^2) \hookrightarrow L^2(B_1, \ell^2)$ is no longer compact,
  so \cite*[Lemma~3.3]{Chang-Wang-Yang:1999:reg-harm-maps} should be replaced by
  an analogous statement for the coordinate functions of $u$. Then one can find
  a harmonic function $h\colon B_{1/2} \to \ell^2$ of the form $h = (h^1,\cdots,h^N,0,0,\cdots)$ such that in the notation of
  \cite*[Lemma~3.3]{Chang-Wang-Yang:1999:reg-harm-maps}, one has
  \begin{equation}
    \int_{B_{1/2}}\abs{w - h}^2 = \sum_{i\leq N}\int_{B_{1/2}}\abs{w^i - h^i}^2
    + \sum_{i> N}\int_{B_{1/2}}\abs{w^i}^2 \leq \epsilon^2
  \end{equation}
  for large enough $N$.
  The rest of the proof remains unchanged, so that we obtain
  $u \in C^{1,\gamma}(\Omega, \ell^2)$. Then the standard elliptic regularity
  theory applied to $u_b:= \brt{u,b} \in C^{1,\gamma}(\Omega)$, where $b\in\ell^2$
  and $\abs{b} \leq 1$, implies
  locally uniform $C^k$ estimates on $\brc{u_b}$, which imply
  $u \in C^{k}(\Omega, \ell^2)$ for $k=2$, and one repeats inductively for any $k \in \mathbb{N}$.
\end{proof}

\begin{lemma}\label{lem:subseq-strong-conv}
  Let $\brc{u_\epsilon} \subset H^1(\Omega, \Sph^\infty)$ be a bounded
  sequence of stable harmonic maps. Then there exists a stable harmonic map
  $u \in H^1(\Omega, \Sph^\infty)$ such that up to a subsequence,
  $\abs{du_\epsilon}^2 \to \abs{du}^2$ in $L^1(\Omega')$ for every $\Omega'\Subset \Omega$. In fact,
  $\sum_i u_\epsilon^i \otimes u_\epsilon^i =: \mathfrak{t}_{\epsilon}$ $\to$
  $\mathfrak{t} := \sum_i u^i \otimes u^i$ in $\tensor{H^1(\Omega')}$.
\end{lemma}
\begin{proof}
  Since $\norm{\mathfrak{t}_\epsilon}_{\tensor{H^1}} = \norm{u_\epsilon}_{H^1}^2$,
  one can pick a subsequence such that
  \begin{equation}
    \mathfrak{t}_\epsilon \oset{w^*}{\to} \mathfrak{t}
  \quad \text{in } \tensor{H^1(\Omega)}.
  \end{equation}
  The positive operator $\mathfrak{t} \geq 0$ can be written as
  $\mathfrak{t} = \sum_i u^i \otimes u^i$ for some $u \in H^1(\Omega, \ell^2)$. By Corollary~\ref{lem:l-inf-convergence},
  one concludes that $\m{\mathfrak{t}}=\abs{u}^2 \equiv 1$.

  One may argue as in \cite*[Lemma~2.2]{Hong-Wang:1999:stable-harmonic-maps} by using
  partial regularity estimates for harmonic maps outside of a set of zero $(m\!-\!2)$-Hausdorff measure. However,
  in the infinite-dimensional case, the stability inequality~\eqref{ineq:stable-map} is so strong
  that one can set $\mu_\epsilon = \abs{du_\epsilon}^2dv_g$, choose
  a subsequence $\mu_\epsilon \oset{w^*}{\to} \mu$, and apply
  Lemmas~\ref{lem:strong-conv} and~\ref{lem:regularity-in-good-pt} with $D = \varnothing$.
\end{proof}

\begin{theorem}\label{thm:reg-harm-map}
  Let $M$ be a Riemannian manifold with $m = \dim M\geq 2$,
  and let $u \in H^1_{loc}(M,\Sph^\infty)$ be a \emph{locally stable} harmonic map. Then
  there exists a closed subset $\Sigma$ of Hausdorff dimension at most $m-7$ such that
  $u \in C^\infty(M\setminus \Sigma, \Sph^\infty)$. If $m = 7$, the set of singularities
  $\Sigma$ is discrete.
\end{theorem}
\begin{proof}
  The proof is a generalization of \cite*{Hong-Wang:1999:stable-harmonic-maps}. Let
  $\Sigma$ be the set of singular points of $u$ (i.e., the complement
  of the largest open set where $u$ is smooth). Since stable harmonic maps
  are stationary by Lemma~\ref{lem:stable-energy-min}, it follows from
  the monotonicity formula,
  Lemmas~\ref{lem:small-energy-reg} and~\ref{lem:cont-harm-smooth} that
  \begin{equation}
    \Sigma = \set*{p \in M}{\lim_{r\to 0} r^{2-m} \int_{B_{r}(p)} \abs{du}^2 \geq \epsilon_0}
  \end{equation}
  for some $\epsilon_0 > 0$.

  To estimate the Hausdorff dimension of $\Sigma$, one follows \cite*{Hong-Wang:1999:stable-harmonic-maps}
  with \cite*[Lemma~2.2]{Hong-Wang:1999:stable-harmonic-maps} replaced by Lemma~\ref{lem:subseq-strong-conv}.
  Namely,
  let us pick a point $p \in \Sigma$ and define $u_\epsilon(x) = u(\operatorname{exp}_p(\epsilon x))$.
  The equality~\eqref{eq:rescale-energy} implies the set $\brc{u_\epsilon}$ is bounded
  on $B_1(0) \subset T_p M$. By Lemma~\ref{lem:subseq-strong-conv}, we find a harmonic
  map $u_0$ such that for a subsequence, $\mathfrak{t}_\epsilon \to \mathfrak{t}$, and therefore,
  $\m{\partial_r\mathfrak{t}_\epsilon\partial_r^*} = \abs{\partial_r u_\epsilon}^2 \to
  \abs{\partial_r u_0}^2$.
  Then~\cite[equation~10a]{Price:1983:monoton-formula} yields
  \begin{align}
    \lim_{\epsilon\to0}\int_{B_1(0)}r^{2-m}\abs{\partial_r u_\epsilon}^2
       &= \lim_{\epsilon\to0}\int_{B_\epsilon(p)}r^{2-m}\abs{\partial_r u}^2
    \\ &\leq \lim_{\epsilon\to0} e^{C\epsilon}\epsilon^{2-m} \int_{B_\epsilon(p)}\abs{du}^2 - L
    \\ &= 0,
  \end{align}
  where $L = \lim_{\epsilon\to 0} \epsilon^{2-m} \int_{B_{\epsilon}(p)} \abs{du}^2$.
  Hence, $u_0$ is radially constant, i.e. $\partial_r u_0 = 0$. The dimension reduction
  argument~\cite[Section~5]{Schoen-Uhlenbeck:1982:regularity-theory-harm-maps} together
  with \cite*[Lemma~2.3]{Hong-Wang:1999:stable-harmonic-maps}
  yields the following conclusion. If $3 \leq d + 1 \leq 6$ and
  there are no nonconstant stable harmonic maps of the form
  $\phi(\frac{x}{\abs{x}})\colon \R^{d+1} \to \Sph^\infty$
  where $\phi \in C^\infty(\Sph^{d}, \Sph^\infty)$,
  then the Hausdorff dimension of $\Sigma$ is at most $m-7$.
  When $m = 7$, the same argument, as in the end of
  \cite[Section~5]{Schoen-Uhlenbeck:1982:regularity-theory-harm-maps}, proves discreteness.

  To see that all such stable harmonic maps are constant, one can
  follow, e.g., \cite*[Proposition~2.5]{Lin-Wang:2006:stable-harm-maps-spheres}
  and conclude that a nonconstant map leads to the inequality $\frac{(d-1)^2}{4d} \geq A = 1$,
  which is a contradiction
  (in our case, the constant in the stability inequality equals $1$ rather than $\frac{k-2}{k}$).
\end{proof}

\subsection{Finite dimensionality}
Let $M$ be a connected Riemannian manifold without boundary (not necessarily complete) with $\dim M \geq 2$.

If $\Omega$ is a bounded Lipschitz domain with $\partial\Omega \neq \varnothing$, then for each map $f \in H^{1}(\Omega, \Sph^\infty)$
that is linearly full on $\partial\Omega$, one can construct its harmonic replacement $u \in H^1(\Omega, \Sph^\infty)$ by minimizing the energy functional
over all positive nuclear operators $\mathfrak{t}_u = \sum_i u^i \otimes u^i \in \tensor{H^1(\Omega)}$ such that
$\mathfrak{t}_u|_{\partial\Omega} = \mathfrak{t}_f|_{\partial\Omega}$ and $\m{\mathfrak{t}_u} = 1$.
The direct method produces a minimizing positive tensor $\mathfrak{t}_u$. Any such minimizer gives a stable harmonic map which is linearly full.

Thus, linearly full harmonic maps $u \colon \Omega \to \Sph^\infty$ do exist. The preceding construction, however,
relies on the inequality $H^1(\Omega) \neq H^1_0(\Omega)$; by contrast, $H^1(M) = H^1_0(M)$ when $M$ is closed. In this section, we prove that every harmonic
map $u \in H^1_0(M, \Sph^\infty)$ with $\ind (\Delta - \abs{du}^2) < \infty$
takes values in a finite-dimensional subsphere $\Sph^n \subset \Sph^\infty$.

For complete manifolds, this follows almost immediately from \cite{Devyver:2012:finite-morse-index-shrodinger}. We adapt the argument to the noncomplete setting.
Recall the notions of index and stability introduced in~\eqref{eq:index} and Definition~\ref{def:stable}.
We write $\ind V := \ind (V,M)$.
\begin{lemma}\label{lem:stability-and-lower-bound}
  Let $V \in L^2_{loc}(M)$ with $\ind V < \infty$. Then there exists a compact $K \subset M$ such that
  $V$ is stable on $M\setminus K$, that is, $\ind (V,M\setminus K) = 0$.
  The symmetric operator $(\Delta - V)\colon C^\infty_0(M) \to L^2(M)$ is bounded from below on $L^2(M)$; therefore,
  its Friedrichs extension is well defined.

  Moreover, if $0 \leq V$, then one has $H^1_0(M) \subset \dom Q_V$, where $Q_V$ is the quadratic form associated with the extension.
\end{lemma}
\begin{proof}
  The proof of the first statement is a modification of Proposition~\ref{prop:con-bilin-form}.
  Since the index is finite, there exists a compact $K \subset M$ such that
  $V$ is stable on $U_0 := M\setminus K$. Indeed, if this were not the case, we would find $\phi \in C^\infty_0(U_0)$ with $Q_V(\phi) < 0$ and
  repeat the construction with $K' := K \cup \supp \phi$. After at most $\ind V$ steps, the construction must terminate.

  Then argue as in the proof of Proposition~\ref{prop:con-bilin-form}: cover $K$ with finitely many $\brc{U_i}_{i=1}^{n}$ such that $\ind (V,U_i) = 0$, and
  the partition-of-unity argument applied to the covering $M = \bigcup_{i=0}^{n} U_i$ yields a constant $C > 0$ such that
  \begin{equation}
    \int \phi^2 V \leq \int \abs{d\phi}^2 + C\int \phi^2 \quad \forall \phi \in C^\infty_0(M),
  \end{equation}
  which implies the desired lower bound on $(\Delta - V)|_{C^\infty_0(M)}$. Thus $(\Delta - V)|_{C^\infty_0(M)}$ is symmetric and bounded from below,
  and therefore admits a Friedrichs extension (see, for example, \cite[Theorem~4.4.5]{Davies:1995:spectral-theory}).
  We denote the domain of its closed quadratic form by $\dom Q_V \subset L^2$,
  where $\dom Q_V$ is the completion of $C^\infty_0(M)$ with respect to the inner product $Q_V(\cdot) + (C+1)\norm{\cdot}^2_{L^2(M)}$.

  When $0 \leq V$, the same estimate tells us that $V$ is a bounded bilinear form on $H^1_0(M)$; hence $H^1_0(M) \subset \dom Q_V$.
\end{proof}

\begin{lemma}\label{lem:positive-solution}
  Let $V \in L^\infty_{loc}(M)$ with $\ind V < \infty$. Then there exists a positive function $w \in C^{1}(M)$
  such that $V_1 = w^{-1}(\Delta - V)w$ is an $L^\infty$ function with compact support.
\end{lemma}
\begin{proof}
  Lemma~\ref{lem:stability-and-lower-bound} shows that $V$ is stable on $M\setminus K$. By \cite[Lemma~3.10]{Pigola-Rigoli-Setti:2008:vanishing-finiteness-results}, this is equivalent
  to the existence of a positive $w' \in C^1(M\setminus K)$ such that $(\Delta  - V)w' = 0$ weakly. Note that the proof of
  \cite[Lemma~3.10]{Pigola-Rigoli-Setti:2008:vanishing-finiteness-results} does not require completeness of $M$.
  Choose $\phi \in C^\infty_0(M)$ such that $0 \leq \phi \leq 1$ and $\phi \equiv 1$ on a neighborhood of $K$, and set
  $w := \phi + (1-\phi)w'$. It remains to check that $(\Delta - V)w \in L^\infty$. Indeed, weakly, one has
  $(\Delta - V)w = (\Delta - V)\phi - (\Delta \phi)w' + 2\brt{d \phi, dw'} \in L^\infty_0(M)$ (cf. \cite[Lemma~4.6]{Devyver:2012:finite-morse-index-shrodinger}).
\end{proof}

\begin{theorem}[{cf. \cite[Theorem~1.3]{Devyver:2012:finite-morse-index-shrodinger}}]
  \label{thm:fin-ind-fin-ker}
  Let $M$ be a connected Riemannian manifold of $\dim M \geq 2$, and let $V \in L^\infty_{loc}(M)$ with $\ind V < \infty$. Then the kernel of the Friedrichs extension of $\Delta - V$
  on $L^2(M)$ is finite-dimensional.
\end{theorem}
\begin{proof}
  By Lemma~\ref{lem:stability-and-lower-bound}, the operator $\Delta - V$ is bounded from below on $C^\infty_0(M) \subset L^2(M)$ and
  has the Friedrichs extension. The kernel of this extension is $\ker Q_V = \set{u \in \dom Q_V}{Q_V(u,v) = 0 \ \forall v \in \dom Q_V}$. By Lemma~\ref{lem:positive-solution},
  choose a positive function $w \in C^1(M)$ such that
  $V_1 = w^{-1}(\Delta - V)w \in L^\infty(M)$ is compactly supported and
  consider the unitary map $L^2(M, w^2) \to L^2(M)$, $\phi \mapsto \phi w$. Under this isometry, for $w\phi \in C^\infty_0(M)$, integration by parts
  yields
  \begin{multline}
    Q_V(\phi w) = \int \abs{d\phi}^2 w^2 + \int \brt{d(\phi^2w), dw} - \int \phi^2wVw
    \\ = \int \abs{d\phi}^2 w^2 + \int V_1\phi^2 w^2 =: Q_1(\phi).
  \end{multline}
  Therefore, the quadratic form $(Q_V, \dom Q_V)$ is unitarily equivalent to the quadratic form $(Q_1, H^1_0(M, w^2))$, where,
  a priori, $H^1_0(M, w^2)$ is the closure of $\set{(\phi,d\phi)}{\phi \in w^{-1}C^\infty_0}$ in $L^2(M,w^2)\oplus L^2(M,w^2, T^*M)$.
  Since $w$ is positive and $C^1$, both $w$ and $w^{-1}$ are locally bounded. Hence the space $H^1_0(M, w^2)$ is the closure of $C^\infty_0$ in
  \begin{equation}\label{eq:weighted-H1}
    H^1(M, w^2) = \set*{\phi \in H^1_{loc}(M)}{\int (\abs{d\phi}^2 + \phi^2) w^2 < \infty}.
  \end{equation}
  To estimate $\dim \ker Q_V = \dim \ker Q_1$, for any open set $\Omega \subset M$, we define
  $\dom Q_{1,\Omega} := \overline{C^\infty_0(M)|_{\Omega}}^{H^1(\Omega, w^2)} \subset L^2(\Omega, w^2)$ and
  \begin{equation}\label{eq:index-plus-null}
    \ind_{0} (Q_1,\Omega) =  \sup \set*{\dim F}{F \subset \dom Q_{1,\Omega} \text{ s.t. }Q_1|_F \leq 0}.
  \end{equation}
  Clearly, $\dim \ker Q_1 \leq \ind_{0} (Q_1, M)$.
  Now, take a smooth $\Omega \Subset M$ such that $\supp V_1 \subset \Omega$. Because $M$ is connected, the closure of each connected component of
  $M\setminus\overline{\Omega}$ meets $\partial \Omega$. A collar of each connected component of $\partial \Omega$ lies in a single exterior component,
  so the number of exterior components is no larger than the number of boundary components. Since $\partial\Omega$ is compact, it has finitely many components.

  Because of~\eqref{eq:weighted-H1} and smoothness of $\Omega$, we have $\dom Q_{1,\Omega} = H^1(\Omega, w^2) = H^1(\Omega)$ and
  $\dom Q_{1, M\setminus\overline{\Omega}}$ is the subspace of $H^1(M\setminus\overline{\Omega}, w^2)$ that has no imposed Dirichlet condition on
  $\partial \Omega$, but retains the Friedrichs condition at infinity. Also,
  the map $H^1_0(M, w^2) \hookrightarrow \dom Q_{1, M\setminus\overline{\Omega}} \oplus \dom Q_{1,\Omega}$ is injective. It follows from the definition
  of $\ind_{0}$ that
  \begin{equation}
    \ind_{0} (Q_1,M) \leq \ind_{0} (Q_{1, M\setminus\overline{\Omega}} \oplus Q_{1,\Omega}) =  \ind_{0} (Q_1, {M\setminus\overline{\Omega}}) + \ind_{0} (Q_1, \Omega).
  \end{equation}
  Since $\supp V_1 \subset \Omega$, we have $Q_1(\phi) = \int \abs{d\phi}^2 w^2$ for all $\phi \in \dom Q_{1, M\setminus\overline{\Omega}}$. Hence
  $\ind_{0} (Q_1, {M\setminus\overline{\Omega}})$ is at most the number of connected components of $M\setminus\overline{\Omega}$, which is finite.
  The second term, $\ind_{0} (Q_1, \Omega)$ is also finite because the embedding $H^1(\Omega) = \dom Q_{1,\Omega} \hookrightarrow L^2(\Omega,w^2) = L^2(\Omega)$ is compact and
  the associated self-adjoint operator has compact resolvent. Thus, $\dim \ker Q_1 < \infty$ and the proof is complete.
\end{proof}
\begin{corollary}\label{cor:finite-sphere-map}
  Let $M$ be a connected Riemannian manifold of $\dim M \geq 2$, and let
  $u \in H^1_0(M,\Sph^\infty) = \set{v \in H^1_0(M,\ell^2)}{\abs{v} = 1}$ be a harmonic map with $\ind (\Delta - \abs{du}^2) < \infty$.
  Then there exists a subsphere $\Sph^n \subset \Sph^\infty$ such that $u(x) \in \Sph^n$ for a.e. $x \in M$.

  In particular, this holds when $M$ is closed and $u \in H^1(M,\Sph^\infty)$.
\end{corollary}
\begin{proof}
 Finite index implies that $V := \abs{du}^2$ is locally stable (see the proof of Proposition~\ref{prop:con-bilin-form}), so by
 Theorem~\ref{thm:reg-harm-map}, $V$ is smooth outside of a closed set $Z$ of zero capacity. Also, $H^1_0(M) = H^1_0(M\setminus Z)$.
 The harmonic map equation gives
 $\Delta u^i = V u^i$ weakly. Also, $H^1_0 \subset \dom Q_V$ by Lemma~\ref{lem:stability-and-lower-bound} applied on $M_1 := M\setminus Z$,
 so the coordinate functions $u^i \in \dom Q_V$ lie in the kernel of the Friedrichs
 extension of $\Delta - V$.
 As $H^1(U) = H^1(U\setminus Z)$ for every open set $U \Subset M$, any locally constant function on $M_1$ is constant; hence
 $M_1$ is connected as well. Applying Theorem~\ref{thm:fin-ind-fin-ker} to $M_1$, we conclude that
 the linear span of the coordinate functions $u^i$ is finite-dimensional.
\end{proof}

\medskip
\printbibliography

@book{Pigola-Rigoli-Setti:2008:vanishing-finiteness-results,
	author = {Pigola, Stefano and Rigoli, Marco and Setti, Alberto G.},
	doi = {10.1007/978-3-7643-8642-9},
	isbn = {978-3-7643-8641-2},
	mrclass = {58-02 (35J60 35R45 53-02 53C21 58J05)},
	mrnumber = {2401291},
	mrreviewer = {David\ L.\ Finn},
	note = {A generalization of the Bochner technique},
	pages = {xiv+282},
	publisher = {Birkh{\"a}user Verlag, Basel},
	series = {Progress in Mathematics},
	title = {Vanishing and finiteness results in geometric analysis},
	url = {https://doi.org/10.1007/978-3-7643-8642-9},
	volume = {266},
	year = {2008}}

@article{Devyver:2012:finite-morse-index-shrodinger,
	author = {Devyver, Baptiste},
	date = {2012/09/01},
	doi = {10.1007/s00229-011-0522-1},
	id = {Devyver2012},
	isbn = {1432-1785},
	journal = {Manuscripta Mathematica},
	number = {1},
	pages = {249--271},
	title = {On the finiteness of the Morse index for Schr{\"o}dinger operators},
	url = {https://doi.org/10.1007/s00229-011-0522-1},
	volume = {139},
	year = {2012}}

@misc{Petrides:2024:conf-class-opt,
	archiveprefix = {arXiv},
	author = {Romain Petrides},
	eprint = {2211.15632v2},
	primaryclass = {math.AP},
	title = {A variational method for functionals depending on eigenvalues},
	url = {https://arxiv.org/abs/2211.15632v2},
	year = {2024}}

@article{Korbas-Novotny:2009:sw-classes-grassmann,
	author = {Korba{\v s}, J. and Novotn{\'y}, P.},
	date = {2009-06-01},
	doi = {10.1007/s10474-008-8116-4},
	id = {Korba{\v s}2009},
	isbn = {1588-2632},
	journal = {Acta Mathematica Hungarica},
	number = {4},
	pages = {319--330},
	title = {On the dual Stiefel-Whitney classes of some Grassmann manifolds},
	url = {https://doi.org/10.1007/s10474-008-8116-4},
	volume = {123},
	year = {2009}}

@article{Stong:1982:products-in-grassmannians,
	author = {Stong, R. E.},
	date = {1982-01-01},
	doi = {https://doi.org/10.1016/0166-8641(82)90012-8},
	isbn = {0166-8641},
	journal = {Topology and its Applications},
	number = {1},
	pages = {103--113},
	title = {Cup products in Grassmannians},
	url = {https://www.sciencedirect.com/science/article/pii/0166864182900128},
	volume = {13},
	year = {1982}}

@article{Grigoryan-Nadirashvili-Sire:2016:eigenval-schrodinger,
	author = {Grigor'yan, Alexander and Nadirashvili, Nikolai and Sire, Yannick},
	fjournal = {Journal of Differential Geometry},
	issn = {0022-040X},
	journal = {J. Differential Geom.},
	mrclass = {58C40 (31C12 35J10 35P20 35R01 58J05)},
	mrnumber = {3466803},
	mrreviewer = {Jie Yang},
	number = {3},
	pages = {395--408},
	title = {A lower bound for the number of negative eigenvalues of {S}chr{\"o}dinger operators},
	url = {http://projecteuclid.org/euclid.jdg/1456754014},
	volume = {102},
	year = {2016}}

@article{Mordukhovich-Shao:1996:subdiff-comparison,
	author = {Mordukhovich, Boris S. and Shao, Yong Heng},
	doi = {10.1090/S0002-9947-96-01543-7},
	fjournal = {Transactions of the American Mathematical Society},
	issn = {0002-9947},
	journal = {Trans. Amer. Math. Soc.},
	mrclass = {49J52 (46B20 46G99 46N10 58C20)},
	mrnumber = {1333396},
	mrreviewer = {J. Borwein},
	number = {4},
	pages = {1235--1280},
	title = {Nonsmooth sequential analysis in {A}splund spaces},
	url = {https://doi.org/10.1090/S0002-9947-96-01543-7},
	volume = {348},
	year = {1996}}

@article{Ioffe:1986:approx-subdiff-arbitrary,
	author = {Ioffe, A. D.},
	doi = {10.1112/S0025579300013930},
	fjournal = {Mathematika. A Journal of Pure and Applied Mathematics},
	issn = {0025-5793},
	journal = {Mathematika},
	mrclass = {49A51 (46G05 58C20 90C30 90C48)},
	mrnumber = {859504},
	mrreviewer = {J.\ Warga},
	number = {1},
	pages = {111--128},
	title = {Approximate subdifferentials and applications. {II}},
	url = {https://doi.org/10.1112/S0025579300013930},
	volume = {33},
	year = {1986}}

@article{Ioffe:1984:approx-subdiff-fin-dim,
	author = {Ioffe, A. D.},
	doi = {10.2307/1999541},
	fjournal = {Transactions of the American Mathematical Society},
	issn = {0002-9947,1088-6850},
	journal = {Trans. Amer. Math. Soc.},
	mrclass = {49A50 (46G05 49A51 58C20 90C48)},
	mrnumber = {719677},
	mrreviewer = {J.\ Warga},
	number = {1},
	pages = {389--416},
	title = {Approximate subdifferentials and applications. {I}. {T}he finite-dimensional theory},
	url = {https://doi.org/10.2307/1999541},
	volume = {281},
	year = {1984}}

@article{Lewis:1999:eigenval-subdiff,
	author = {Lewis, A. S.},
	doi = {10.1007/s10107980004a},
	fjournal = {Mathematical Programming. A Publication of the Mathematical Programming Society},
	issn = {0025-5610},
	journal = {Math. Program.},
	mrclass = {49J52 (15A18 65F15)},
	mrnumber = {1687292},
	mrreviewer = {Giorgio Giorgi},
	number = {1, Ser. A},
	pages = {1--24},
	title = {Nonsmooth analysis of eigenvalues},
	url = {https://doi.org/10.1007/s10107980004a},
	volume = {84},
	year = {1999}}

@book{Davies:1995:spectral-theory,
	author = {Davies, E. B.},
	doi = {10.1017/CBO9780511623721},
	isbn = {0-521-47250-4},
	mrclass = {47F05 (35Pxx 47A10 47B25 81Q10)},
	mrnumber = {1349825},
	mrreviewer = {Peter D. Hislop},
	pages = {x+182},
	publisher = {Cambridge University Press, Cambridge},
	series = {Cambridge Studies in Advanced Mathematics},
	title = {Spectral theory and differential operators},
	url = {https://doi.org/10.1017/CBO9780511623721},
	volume = {42},
	year = {1995}}

@article{Colbois-ElSoufi-Savo:2015:-laplace-with-density,
	author = {Colbois, Bruno and El Soufi, Ahmad and Savo, Alessandro},
	doi = {10.4310/CAG.2015.v23.n3.a6},
	fjournal = {Communications in Analysis and Geometry},
	issn = {1019-8385,1944-9992},
	journal = {Comm. Anal. Geom.},
	mrclass = {58J50},
	mrnumber = {3310527},
	mrreviewer = {Liang\ Zhao},
	number = {3},
	pages = {639--670},
	title = {Eigenvalues of the {L}aplacian on a compact manifold with density},
	url = {https://doi.org/10.4310/CAG.2015.v23.n3.a6},
	volume = {23},
	year = {2015}}

@article{ElSoufi-Ilias:1986:hersch,
	author = {El Soufi, A. and Ilias, S.},
	date = {1986-06-01},
	doi = {10.1007/BF01458460},
	id = {El Soufi1986},
	isbn = {1432-1807},
	journal = {Mathematische Annalen},
	number = {2},
	pages = {257--267},
	title = {Immersions minimales, premi{\`e}re valeur propre du laplacien et volume conforme},
	url = {https://doi.org/10.1007/BF01458460},
	volume = {275},
	year = {1986}}

@incollection{Grigoryan-Netrusov-Yau:2004:eignval-of-ellipt-oper,
	author = {Grigor'yan, Alexander and Netrusov, Yuri and Yau, Shing-Tung},
	booktitle = {Surveys in differential geometry. {V}ol. {IX}},
	doi = {10.4310/SDG.2004.v9.n1.a5},
	isbn = {1-57146-115-9},
	mrclass = {58J50 (31C12 31C25)},
	mrnumber = {2195408},
	mrreviewer = {Alberto\ G.\ Setti},
	pages = {147--217},
	publisher = {Int. Press, Somerville, MA},
	series = {Surv. Differ. Geom.},
	title = {Eigenvalues of elliptic operators and geometric applications},
	url = {https://doi.org/10.4310/SDG.2004.v9.n1.a5},
	volume = {9},
	year = {2004}}

@misc{vinokurov:2025:sym-eigen-val,
	archiveprefix = {arXiv},
	author = {Denis Vinokurov},
	eprint = {2502.03756v2},
	primaryclass = {math.SP},
	title = {Conformal optimization of eigenvalues on surfaces with symmetries},
	url = {https://arxiv.org/abs/2502.03756v2},
	year = {2025}}

@inbook{Everitt:2005:sl-diff-eq,
	address = {Basel},
	author = {Everitt, W. Norrie},
	booktitle = {Sturm-Liouville Theory: Past and Present},
	date = {2005},
	doi = {10.1007/3-7643-7359-8{\_}12},
	editor = {Amrein, Werner O. and Hinz, Andreas M. and Pearson, David P.},
	id = {Everitt2005},
	isbn = {978-3-7643-7359-7},
	pages = {271--331},
	publisher = {Birkh{\"a}user Basel},
	title = {A Catalogue of Sturm-Liouville Differential Equations},
	url = {https://doi.org/10.1007/3-7643-7359-8_12},
	year = {2005}}

@book{Zettl:2005:sturm-liouville,
	author = {Zettl, Anton},
	doi = {10.1090/surv/121},
	isbn = {0-8218-3905-5},
	mrclass = {34-02 (34B20 34B24 34L05 47E05)},
	mrnumber = {2170950},
	mrreviewer = {Mikl\'{o}s Horv\'{a}th},
	pages = {xii+328},
	publisher = {American Mathematical Society, Providence, RI},
	series = {Mathematical Surveys and Monographs},
	title = {Sturm-Liouville theory},
	url = {https://doi.org/10.1090/surv/121},
	volume = {121},
	year = {2005}}

@article{Laugesen:2021:center-mass,
	author = {Laugesen, R. S.},
	date = {2021-10-01},
	doi = {10.1007/s40316-020-00151-5},
	id = {Laugesen2021},
	isbn = {2195-4763},
	journal = {Annales math{\'e}matiques du Qu{\'e}bec},
	number = {2},
	pages = {363--390},
	title = {Well-posedness of Hersch--Szeg{\H o}'s center of mass by hyperbolic energy minimization},
	url = {https://doi.org/10.1007/s40316-020-00151-5},
	volume = {45},
	year = {2021}}

@article{Karpukhin-Stern:2024:harm-map-in-higher-dim,
	author = {Karpukhin, Mikhail and Stern, Daniel},
	doi = {10.1007/s00222-024-01247-3},
	fjournal = {Inventiones Mathematicae},
	issn = {0020-9910},
	journal = {Invent. Math.},
	mrclass = {53C43},
	mrnumber = {4728241},
	number = {2},
	pages = {713--778},
	title = {Existence of harmonic maps and eigenvalue optimization in higher dimensions},
	url = {https://doi.org/10.1007/s00222-024-01247-3},
	volume = {236},
	year = {2024}}

@article{Lin-Wang:2006:stable-harm-maps-spheres,
	author = {Lin, Fang Hua and Wang, Chang You},
	date = {2006-04-01},
	doi = {10.1007/s10114-005-0673-7},
	id = {Lin2006},
	isbn = {1439-7617},
	journal = {Acta Mathematica Sinica},
	number = {2},
	pages = {319--330},
	title = {Stable Stationary Harmonic Maps to Spheres},
	url = {https://doi.org/10.1007/s10114-005-0673-7},
	volume = {22},
	year = {2006}}

@article{Schoen-Uhlenbeck:1982:regularity-theory-harm-maps,
	author = {Richard Schoen and Karen Uhlenbeck},
	date = {1982-01-01},
	doi = {10.4310/jdg/1214436923},
	journal = {Journal of Differential Geometry},
	journal1 = {Journal of Differential Geometry},
	journal2 = {Journal of Differential Geometry},
	month = {1},
	number = {2},
	pages = {307--335},
	title = {A regularity theory for harmonic maps},
	url = {https://doi.org/10.4310/jdg/1214436923},
	volume = {17},
	year = {1982}}

@article{Hong-Wang:1999:stable-harmonic-maps,
	author = {Hong, Min-Chun and Wang, Chang-You},
	date = {1999-09-01},
	doi = {10.1007/s005260050135},
	id = {Hong1999},
	isbn = {1432-0835},
	journal = {Calculus of Variations and Partial Differential Equations},
	number = {2},
	pages = {141--156},
	title = {On the singular set of stable-stationary harmonic maps},
	url = {https://doi.org/10.1007/s005260050135},
	volume = {9},
	year = {1999}}

@article{Buckley:1999:bmo-doubling-spaces,
	author = {Buckley, Stephen M.},
	date = {1999-12-01},
	doi = {10.1007/BF02788242},
	id = {Buckley1999},
	isbn = {1565-8538},
	journal = {Journal d'Analyse Math{\'e}matique},
	number = {1},
	pages = {215--240},
	title = {Inequalities of John---Nirenberg type in doubling spaces},
	url = {https://doi.org/10.1007/BF02788242},
	volume = {79},
	year = {1999}}

@article{John-Nirenberg:1961:bmo,
	author = {John, F. and Nirenberg, L.},
	doi = {10.1002/cpa.3160140317},
	fjournal = {Communications on Pure and Applied Mathematics},
	issn = {0010-3640},
	journal = {Comm. Pure Appl. Math.},
	mrclass = {26.00},
	mrnumber = {131498},
	mrreviewer = {L. C. Young},
	pages = {415--426},
	title = {On functions of bounded mean oscillation},
	url = {https://doi.org/10.1002/cpa.3160140317},
	volume = {14},
	year = {1961}}

@book{Hytonen-Neerven-Veraar-Weis:2016:analysis-banach-1,
	author = {Hyt{\"o}nen, Tuomas and van Neerven, Jan and Veraar, Mark and Weis, Lutz},
	isbn = {978-3-319-48519-5; 978-3-319-48520-1},
	mrclass = {46-02 (42B35 46E30)},
	mrnumber = {3617205},
	mrreviewer = {Adam\ Os{\k e}kowski},
	pages = {xvi+614},
	publisher = {Springer, Cham},
	series = {Ergebnisse der Mathematik und ihrer Grenzgebiete. 3. Folge. A Series of Modern Surveys in Mathematics [Results in Mathematics and Related Areas. 3rd Series. A Series of Modern Surveys in Mathematics]},
	title = {Analysis in Banach spaces. Vol. I. Martingales and Littlewood-Paley theory},
	volume = {63},
	year = {2016}}

@book{Hytonen-Neerven-Veraar-Weis:2023:analysis-banach-3,
	author = {Hyt{\"o}nen, Tuomas and van Neerven, Jan and Veraar, Mark and Weis, Lutz},
	doi = {10.1007/978-3-031-46598-7},
	isbn = {978-3-031-46597-0; 978-3-031-46598-7},
	mrclass = {46-02 (35Kxx 42Bxx 46Bxx 46Exx)},
	mrnumber = {4696978},
	mrreviewer = {Pierre\ Portal},
	pages = {xxi+826},
	publisher = {Springer, Cham},
	series = {Ergebnisse der Mathematik und ihrer Grenzgebiete. 3. Folge. A Series of Modern Surveys in Mathematics [Results in Mathematics and Related Areas. 3rd Series. A Series of Modern Surveys in Mathematics]},
	title = {Analysis in Banach spaces. Vol. III. Harmonic analysis and spectral theory},
	url = {https://doi.org/10.1007/978-3-031-46598-7},
	volume = {76},
	year = {2023}}

@book{Taylor:2023:pde-3,
	author = {Taylor, Michael E.},
	doi = {10.1007/978-3-031-33928-8},
	edition = {Third},
	isbn = {978-3-031-33927-1; 978-3-031-33928-8},
	mrclass = {35-01 (46N20 47F05 47N20)},
	mrnumber = {4703941},
	pages = {xxiii+755},
	publisher = {Springer, Cham},
	series = {Applied Mathematical Sciences},
	title = {Partial differential equations {III}. {N}onlinear equations},
	url = {https://doi.org/10.1007/978-3-031-33928-8},
	volume = {117},
	year = {2023}}

@article{Perez-Ayala:2021:conf-laplace-sire-xu-norm,
	author = {P{\'e}rez-Ayala, Samuel},
	doi = {10.1016/j.na.2021.112308},
	fjournal = {Nonlinear Analysis. Theory, Methods \& Applications. An International Multidisciplinary Journal},
	issn = {0362-546X,1873-5215},
	journal = {Nonlinear Anal.},
	mrclass = {58C40 (35P30 53C18 58J05)},
	mrnumber = {4226423},
	mrreviewer = {Seunghyeok\ Kim},
	pages = {Paper No. 112308, 12},
	title = {Extremal eigenvalues of the conformal {L}aplacian under {S}ire-{X}u normalization},
	url = {https://doi.org/10.1016/j.na.2021.112308},
	volume = {208},
	year = {2021}}

@article{Gursky-Perez-Ayala:2022:2d-eigenval-conform-laplacian,
	author = {Gursky, Matthew J. and P{\'e}rez-Ayala, Samuel},
	doi = {10.1016/j.jfa.2021.109371},
	fjournal = {Journal of Functional Analysis},
	issn = {0022-1236,1096-0783},
	journal = {J. Funct. Anal.},
	mrclass = {58J50 (35P30 49Q05 49R05 53C18)},
	mrnumber = {4372143},
	number = {8},
	pages = {Paper No. 109371, 60},
	title = {Variational properties of the second eigenvalue of the conformal {L}aplacian},
	url = {https://doi.org/10.1016/j.jfa.2021.109371},
	volume = {282},
	year = {2022}}

@article{Perez-Ayala:2022:extr-metrics-paneitz-operator,
	author = {P{\'e}rez-Ayala, Samuel},
	doi = {10.1016/j.geomphys.2022.104666},
	fjournal = {Journal of Geometry and Physics},
	issn = {0393-0440,1879-1662},
	journal = {J. Geom. Phys.},
	mrclass = {58J60 (35P15 58C40 58E20)},
	mrnumber = {4486158},
	pages = {Paper No. 104666, 26},
	title = {Extremal metrics for the Paneitz operator on closed four-manifolds},
	url = {https://doi.org/10.1016/j.geomphys.2022.104666},
	volume = {182},
	year = {2022}}

@article{Karpukhin-Metras-Polterovich:2024:dirac-eigenval-opt,
	author = {Karpukhin, Mikhail and M{\'e}tras, Antoine and Polterovich, Iosif},
	doi = {10.1093/imrn/rnae216},
	fjournal = {International Mathematics Research Notices. IMRN},
	issn = {1073-7928,1687-0247},
	journal = {Int. Math. Res. Not. IMRN},
	mrclass = {58E20 (32L99)},
	mrnumber = {4819873},
	number = {21},
	pages = {13758--13784},
	title = {Dirac eigenvalue optimisation and harmonic maps to complex projective spaces},
	url = {https://doi.org/10.1093/imrn/rnae216},
	year = {2024}}

@article{Kim:2022:second-sphere-eigenval,
	author = {Kim, Hanna N.},
	doi = {10.1090/proc/15908},
	fjournal = {Proceedings of the American Mathematical Society},
	issn = {0002-9939},
	journal = {Proc. Amer. Math. Soc.},
	mrclass = {35P15 (58C40 58J50)},
	mrnumber = {4439471},
	number = {8},
	pages = {3501--3512},
	title = {Maximization of the second {L}aplacian eigenvalue on the sphere},
	url = {https://doi.org/10.1090/proc/15908},
	volume = {150},
	year = {2022}}

@misc{Petrides:2024:spec-optim-on-surfaces,
	archiveprefix = {arXiv},
	author = {Romain Petrides},
	eprint = {2410.13347v1},
	primaryclass = {math.DG},
	title = {Geometric spectral optimization on surfaces},
	url = {https://arxiv.org/abs/2410.13347v1},
	year = {2024}}

@article{Iwaniec-Scott-Stroffolini:1999:lp-hodge-decompose,
	author = {Iwaniec, T. and Scott, C. and Stroffolini, B.},
	doi = {10.1007/BF02505905},
	fjournal = {Annali di Matematica Pura ed Applicata. Serie Quarta},
	issn = {0003-4622},
	journal = {Ann. Mat. Pura Appl. (4)},
	mrclass = {58J32 (58A10 58A14)},
	mrnumber = {1747627},
	mrreviewer = {J\"{u}rgen Eichhorn},
	pages = {37--115},
	title = {Nonlinear {H}odge theory on manifolds with boundary},
	url = {https://doi.org/10.1007/BF02505905},
	volume = {177},
	year = {1999}}

@book{Defant-Floret:1993:tensor-norms,
	author = {Defant, Andreas and Floret, Klaus},
	isbn = {0-444-89091-2},
	mrclass = {46M05 (46-02 46B20 47B10 47D50)},
	mrnumber = {1209438},
	mrreviewer = {A. Pietsch},
	pages = {xii+566},
	publisher = {North-Holland Publishing Co., Amsterdam},
	series = {North-Holland Mathematics Studies},
	title = {Tensor norms and operator ideals},
	volume = {176},
	year = {1993}}

@article{Price:1983:monoton-formula,
	author = {Price, Peter},
	date = {1983-06-01},
	doi = {10.1007/BF01165828},
	id = {Price1983},
	isbn = {1432-1785},
	journal = {manuscripta mathematica},
	number = {2},
	pages = {131--166},
	title = {A monotonicity formula for Yang-Mills fields},
	url = {https://doi.org/10.1007/BF01165828},
	volume = {43},
	year = {1983}}

@article{Chang-Wang-Yang:1999:reg-harm-maps,
	author = {Chang, Sun-Yung A. and Wang, Lihe and Yang, Paul C.},
	date = {1999-09-01},
	doi = {https://doi.org/10.1002/(SICI)1097-0312(199909)52:9<1099::AID-CPA3>3.0.CO;2-O},
	isbn = {0010-3640},
	journal = {Communications on Pure and Applied Mathematics},
	journal1 = {Communications on Pure and Applied Mathematics},
	journal2 = {Communications on Pure and Applied Mathematics},
	journal3 = {Comm. Pure Appl. Math.},
	month = {9},
	n2 = {Abstract We present an elementary argument of the regularity of weak harmonic maps of a surface into the spheres, as well as the partial regularity of stationary harmonic maps of a higher-dimensional domain into the spheres. The argument does not make use of the structure of Hardy spaces. ? 1999 John Wiley \& Sons, Inc.},
	number = {9},
	pages = {1099--1111},
	publisher = {John Wiley \& Sons, Ltd},
	title = {Regularity of harmonic maps},
	url = {https://doi.org/10.1002/(SICI)1097-0312(199909)52:9<1099::AID-CPA3>3.0.CO;2-O},
	volume = {52},
	year = {1999},
	year1 = {1999}}

@article{Schoen-Uhlenbeck:1984:min-harm-maps,
	author = {Schoen, Richard and Uhlenbeck, Karen},
	date = {1984-02-01},
	doi = {10.1007/BF01388715},
	id = {Schoen1984},
	isbn = {1432-1297},
	journal = {Inventiones mathematicae},
	number = {1},
	pages = {89--100},
	title = {Regularity of minimizing harmonic maps into the sphere},
	url = {https://doi.org/10.1007/BF01388715},
	volume = {78},
	year = {1984}}

@misc{Petrides-Tewodrose:2024:eigenvalue-via-clarke,
	archiveprefix = {arXiv},
	author = {Romain Petrides and David Tewodrose},
	eprint = {2403.07841v1},
	primaryclass = {math.DG},
	title = {Critical metrics of eigenvalue functionals via Clarke subdifferential},
	url = {https://arxiv.org/abs/2403.07841v1},
	year = {2024}}

@incollection{Karpukhin-Nadirashvili-Penskoi-Polterovich:2022:existence,
	author = {Karpukhin, Mikhail and Nadirashvili, Nikolai and Penskoi, Alexei V. and Polterovich, Iosif},
	booktitle = {Surveys in differential geometry 2019. {D}ifferential geometry, {C}alabi-{Y}au theory, and general relativity. {P}art 2},
	mrclass = {58J50 (53C42 58E11)},
	mrnumber = {4479722},
	pages = {205--256},
	publisher = {Int. Press, Boston, MA},
	series = {Surv. Differ. Geom.},
	title = {Conformally maximal metrics for {L}aplace eigenvalues on surfaces},
	volume = {24},
	year = {2022}}

@misc{Karpukhin-Nahon-Polterovich-Stern:2021:stability,
	archiveprefix = {arXiv},
	author = {Mikhail Karpukhin and Micka{\"e}l Nahon and Iosif Polterovich and Daniel Stern},
	eprint = {2106.15043v1},
	primaryclass = {math.DG},
	title = {Stability of isoperimetric inequalities for Laplace eigenvalues on surfaces},
	url = {https://arxiv.org/abs/2106.15043v1},
	year = {2021}}

@article{Arazy:1981:convergence--unitary-matrix-spaces,
	author = {Arazy, Jonathan},
	doi = {10.2307/2043888},
	fjournal = {Proceedings of the American Mathematical Society},
	issn = {0002-9939,1088-6826},
	journal = {Proc. Amer. Math. Soc.},
	mrclass = {46A45 (46B20 47D45)},
	mrnumber = {619978},
	mrreviewer = {W.\ H.\ Ruckle},
	number = {1},
	pages = {44--48},
	title = {More on convergence in unitary matrix spaces},
	url = {https://doi.org/10.2307/2043888},
	volume = {83},
	year = {1981}}

@article{Hersch:1970:sphere-inequality,
	author = {Hersch, Joseph},
	fjournal = {Comptes Rendus Hebdomadaires des S{\'e}ances de l'Acad{\'e}mie des Sciences. S{\'e}ries A et B},
	issn = {0151-0509},
	journal = {C. R. Acad. Sci. Paris S{\'e}r. A-B},
	mrclass = {73.35},
	mrnumber = {292357},
	mrreviewer = {W.\ E.\ Boyce},
	pages = {A1645--A1648},
	title = {Quatre propri{\'e}t{\'e}s isop{\'e}rim{\'e}triques de membranes sph{\'e}riques homog{\`e}nes},
	volume = {270},
	year = {1970}}

@article{Petrides:2018:exist-of-max-eigenval-on-surfaces,
	author = {Petrides, Romain},
	doi = {10.1093/imrn/rnx004},
	isbn = {1073-7928},
	journal = {International Mathematics Research Notices},
	journal1 = {Int Math Res Notices},
	number = {14},
	pages = {4261--4355},
	title = {On the Existence of Metrics Which Maximize Laplace Eigenvalues on Surfaces},
	url = {https://doi.org/10.1093/imrn/rnx004},
	volume = {2018},
	year = {2018},
	year1 = {2018/07/20}}

@book{Kato:1995:perturbation-theory,
	author = {Kato, Tosio},
	isbn = {3-540-58661-X},
	mrclass = {47A55 (46-00 47-00)},
	mrnumber = {1335452},
	note = {Reprint of the 1980 edition},
	pages = {xxii+619},
	publisher = {Springer-Verlag, Berlin},
	series = {Classics in Mathematics},
	title = {Perturbation theory for linear operators},
	year = {1995}}

@book{Schirotzek:2007:nonsmooth-analysis,
	author = {Schirotzek, Winfried},
	doi = {10.1007/978-3-540-71333-3},
	isbn = {978-3-540-71332-6; 3-540-71332-8},
	mrclass = {49-01 (49J52 49J53 90C48)},
	mrnumber = {2330778},
	mrreviewer = {G{\'e}rard Lebourg},
	pages = {xii+373},
	publisher = {Springer, Berlin},
	series = {Universitext},
	title = {Nonsmooth analysis},
	url = {https://doi.org/10.1007/978-3-540-71333-3},
	year = {2007}}

@article{Girouard-Karpukhin-Lagace:2021:continuity-of-eigenval,
	author = {Girouard, Alexandre and Karpukhin, Mikhail and Lagac{\'e}, Jean},
	date = {2021-06-01},
	doi = {10.1007/s00039-021-00573-5},
	id = {Girouard2021},
	isbn = {1420-8970},
	journal = {Geometric and Functional Analysis},
	number = {3},
	pages = {513--561},
	title = {Continuity of eigenvalues and shape optimisation for Laplace and Steklov problems},
	url = {https://doi.org/10.1007/s00039-021-00573-5},
	volume = {31},
	year = {2021}}

@article{Kokarev:2014:measure-eigenval,
	author = {Gerasim Kokarev},
	doi = {https://doi.org/10.1016/j.aim.2014.03.006},
	issn = {0001-8708},
	journal = {Advances in Mathematics},
	pages = {191-239},
	title = {Variational aspects of Laplace eigenvalues on Riemannian surfaces},
	url = {https://www.sciencedirect.com/science/article/pii/S0001870814001005},
	volume = {258},
	year = {2014}}

@article{Karpukhin-Metras:2022:higher-dim,
	author = {Karpukhin, Mikhail and M{\'e}tras, Antoine},
	date = {2022-02-26},
	doi = {10.1007/s12220-022-00891-6},
	id = {Karpukhin2022},
	isbn = {1559-002X},
	journal = {The Journal of Geometric Analysis},
	number = {5},
	pages = {154},
	title = {Laplace and Steklov Extremal Metrics via n-Harmonic Maps},
	url = {https://doi.org/10.1007/s12220-022-00891-6},
	volume = {32},
	year = {2022}}
\end{document}